\documentclass[11pt]{article}
\usepackage{amsmath,amssymb,amsthm,mathtools,xcolor}
\usepackage{algorithm, algpseudocode}
\usepackage{paralist,tocloft}

\usepackage{verbatim}
\usepackage{bbm}
\usepackage{hyperref}
\hypersetup{
	colorlinks,
	citecolor=black,
	filecolor=black,
	linkcolor=black,
	urlcolor=black,
	linktoc=all
}

\theoremstyle{plain}
\newtheorem{theorem}{Theorem}[section]
\newtheorem{lemma}[theorem]{Lemma}
\newtheorem{claim}[theorem]{Claim}

\newtheorem{corollary}[theorem]{Corollary}
\newtheorem{observation}[theorem]{Observation}

\theoremstyle{definition}
\newtheorem{definition}[theorem]{Definition}
\newtheorem*{remark}{Remark}

\theoremstyle{definition}

\theoremstyle{plain}

\newcommand{\Var}{\operatorname{{\bf Var}}}

\newcommand{\Ex}{\mathop{{\bf E}\/}}
\renewcommand{\Pr}{\operatorname{{\bf Pr}}}
\newcommand{\Prx}{\mathop{{\bf Pr}\/}}

\newcommand{\supp}{\mathrm{supp}}

\newcommand{\balpha}{{\boldsymbol{\alpha}} }

\newcommand{\eps}{\varepsilon}
\newcommand{\la}{\langle}
\newcommand{\ra}{\rangle}

\newcommand{\wh}[1]{\widehat{#1}}

\renewcommand{\hat}{\wh}

\newcommand{\ol}[1]{\overline{#1}}

\newcommand{\bA}{\mathbf{A}}

\newcommand{\bG}{\mathbf{G}}
\newcommand{\bH}{\mathbf{H}}

\newcommand{\bM}{\mathbf{M}}

\newcommand{\bR}{\mathbf{R}}

\newcommand{\bX}{\mathbf{X}}
\newcommand{\bY}{\mathbf{Y}}
\newcommand{\bZ}{\mathbf{Z}}

\newcommand{\ba}{\boldsymbol{a}}
\newcommand{\bb}{\boldsymbol{b}}

\newcommand{\bg}{\boldsymbol{g}}

\newcommand{\bk}{\boldsymbol{k}}

\newcommand{\bs}{\boldsymbol{s}}

\newcommand{\eqdef}{\stackrel{\rm def}{=}}

\newcommand{\calA}{\mathcal{A}}
\newcommand{\calB}{\mathcal{B}}
\newcommand{\calC}{\mathcal{C}}
\newcommand{\calD}{\mathcal{D}}
\newcommand{\calE}{\mathcal{E}}
\newcommand{\calF}{\mathcal{F}}
\newcommand{\calG}{\mathcal{G}}
\newcommand{\calH}{\mathcal{H}}
\newcommand{\calI}{\mathcal{I}}

\newcommand{\calK}{\mathcal{K}}
\newcommand{\calL}{\mathcal{L}}
\newcommand{\calM}{\mathcal{M}}

\newcommand{\calP}{\mathcal{P}}
\newcommand{\calQ}{\mathcal{Q}}
\newcommand{\calR}{\mathcal{R}}
\newcommand{\calS}{\mathcal{S}}
\newcommand{\calT}{\mathcal{T}}

\newcommand{\calW}{\mathcal{W}}

\newcommand{\barHU}{\overline{\calH_U}}

\def\FullBox{\hbox{\vrule width 8pt height 8pt depth 0pt}}
\newcommand{\QED}{\;\;\;\FullBox}
\renewenvironment{proof}{\noindent{\bf Proof:~}}{\hfill\QED}
\newenvironment{proofof}[1]{\noindent{\bf Proof of {#1}:~}}{\hfill\(\QED\)}

\newcommand{\shift}[1]{\mathrm{Shift}({#1})}
\newcommand{\seq}[2]{\{#1\}_{#2 \in \mathbb{N}}}
\newcommand{\csdist}{d_{\triangle}}
\newcommand{\barH}{\overline{\calH}}
\newcommand{\barP}{\overline{\calP}}

\makeatletter
\newtheorem*{rep@theorem}{\rep@title}
\newcommand{\newreptheorem}[2]{%
	\newenvironment{rep#1}[1]{%
		\def\rep@title{#2 \ref{##1}}%
		\begin{rep@theorem}}%
		{\end{rep@theorem}}}
\makeatother
\newreptheorem{theorem}{Theorem}

\newcommand{\indi}{\boldsymbol{1}}

\newcommand{\onote}[1]{\textcolor{blue}{\bf{Omri: #1}}}

\def\authornameAL{Amit Levi}
\def\authoraffiAL{University of Waterloo, Canada. Email: \href{mailto: amit.levi@uwaterloo.ca}{amit.levi@uwaterloo.ca}. Research supported by the David R. Cheriton Graduate Scholarship. Part of this work was done while the author was visiting the Technion.}
\def\authornameEF{Eldar Fischer}
\def\authoraffiEF{Technion - Israel Institute of Technology, Israel. Email: \href{mailto:eldar@cs.technion.ac.il}{eldar@cs.technion.ac.il}.}
\def\authornameOBE{Omri Ben-Eliezer}
\def\authoraffiOBE{Massachusetts Institute of Technology, USA. Email: \href{mailto: omrib@mit.edu}{omrib@mit.edu}.
	Part of this work was done while the author was at Tel Aviv University and later at Harvard University.}
\def\authornameYY{Yuichi Yoshida}
\def\authoraffiYY{National Institute of Informatics (NII), Japan. Email: \href{mailto: yyoshida@nii.ac.jp}{yyoshida@nii.ac.jp}. Research supported by JSPS KAKENHI Grant Number JP17H04676.}

\title{Ordered Graph Limits and Their Applications}
\author{
	\authornameOBE\thanks{\authoraffiOBE}
	\and
	\authornameEF\thanks{\authoraffiEF}
	\and
	\authornameAL\thanks{\authoraffiAL}
	\and
	\authornameYY\thanks{\authoraffiYY}
}
\date{}
\begin{document}

	\maketitle
	\begin{abstract}
	The emerging theory of graph limits exhibits an analytic perspective on graphs, showing that many important concepts and tools in graph theory and its applications can be described more naturally (and sometimes proved more easily) in analytic language.
	We extend the theory of graph limits to the ordered setting, presenting a limit object for dense vertex-ordered graphs, which we call an \emph{orderon}. As a special case, this yields limit objects for matrices whose rows and columns are ordered, and for dynamic graphs that expand (via vertex insertions) over time.

	Along the way, we devise an ordered locality-preserving variant of the cut distance between ordered graphs, showing that two graphs are close with respect to this distance if and only if they are similar in terms of their ordered subgraph frequencies. We show that the space of orderons is compact with respect to this distance notion, which is key to a successful analysis of combinatorial objects through their limits. For the proof we combine techniques used in the unordered setting with several new techniques specifically designed to overcome the challenges arising in the ordered setting.

	We derive several applications of the ordered limit theory in extremal combinatorics, sampling, and property testing in ordered graphs. In particular, we prove a new ordered analogue of the well-known result by Alon and Stav [RS\&A'08] on the furthest graph from a hereditary property; this is the first known result of this type in the ordered setting.
	Unlike the unordered regime, here the Erd\H{o}s–R\'enyi random graph $\bG(n, p)$ with an ordering over the vertices is \emph{not} always asymptotically the furthest from the property for some $p$. However, using our ordered limit theory, we show that random graphs generated by a stochastic block model, where the blocks are consecutive in the vertex ordering, are (approximately) the furthest.
	Additionally, we describe an alternative analytic proof of the ordered graph removal lemma [Alon et al., FOCS'17].
	\end{abstract}

\thispagestyle{empty}

\newpage
\thispagestyle{empty}
\setcounter{tocdepth}{2}
\tableofcontents
\thispagestyle{empty}
\newpage
\setcounter{page}{1}
\newpage
	\section{Introduction}\label{sec:intro}

	Large graphs appear in many applications across all scientific areas. Naturally, it is interesting to try to understand their structure and behavior: When can we say that two graphs are similar (even if they do not have the same size)? How can the convergence of graph sequences be defined? What properties of a large graph can we capture by taking a small sample from it?

	The theory of graph limits addresses such questions from an analytic point of view.
	The investigation of convergent sequences of dense graphs was started to address three seemingly unrelated questions asked in different fields: statistical physics, theory of networks and the Internet, and quasi-randomness. A comprehensive series of papers~\cite{borgs2006graph, BCLSV_discrete06, LS06, Freedman2007reflection, LS07, BCLSV_advances08, BCL10, LS10, BCLSV_annals12} laid the infrastructure for a rigorous study of the theory of dense graph limits, demonstrating various applications in many areas of mathematics and computer science. The book of Lov\'asz on graph limits~\cite{Lov12} presents these results in a unified form.

	A sequence $\{G_n\}_{n=1}^{\infty}$ of finite graphs, whose number of vertices tends to infinity as $n \to \infty$, is considered \emph{convergent}\footnote{In unordered graphs, this is also called \emph{convergence from the left}; see the discussion on~\cite{BCLSV_advances08}.} if the frequency\footnote{The frequency of $F$ in $G$ is roughly defined as the ratio of induced subgraphs of $G$ isomorphic to $F$ among all induced subgraphs of $G$ on $|F|$ vertices.} of any fixed graph $F$ as a subgraph in $G_n$ converges as $n \to \infty$.
	The limit object of a convergent sequence of (unordered) graphs in the dense setting, called a \emph{graphon}, is a measurable symmetric function $W \colon [0,1]^2 \to [0,1]$, and it was proved in~\cite{LS06} that, indeed, for any convergent sequence $\{G_n\}$ of graphs there exists a graphon serving as the limit of $G_n$ in terms of subgraph frequencies.
	Apart from their role in the theory of graph limits, graphons are useful in probability theory, as they give rise to exchangeable random graph models; see e.g.\@~\cite{DiaconisJanson08, OrbanzRoy15}.
	An analytic theory of convergence has been established for many other types of discrete structures. These include sparse graphs, for which many different (and sometimes incomparable) notions of limits exist---see e.g.~\cite{BorgsChayes17, BCG17} for two recent papers citing and discussing many of the works in this field; permutations, first developed in~\cite{HKMRS13} and further investigated in several other works; partial orders~\cite{Janson11}; and high dimensional functions over finite fields~\cite{Yoshida16}. The limit theory of dense graphs has also been extended to hypergraphs, see~\cite{ElekSzegedy12,Zhao15} and the references within.

	In this work we extend the theory of dense graph limits to the ordered setting, establishing a limit theory for vertex-ordered graphs in the dense setting, and presenting several applications of this theory. An \emph{ordered graph} is a symmetric function $G \colon [n]^2 \to \{0,1\}$. $G$ is \emph{simple} if $G(x,x) = 0$ for any $x$.
	A \emph{weighted ordered graph} is a symmetric function $F \colon [n]^2 \to [0,1]$.
	Unlike the unordered setting, where $G,G' \colon [n]^2 \to \Sigma$ are considered isomorphic if there is a permutation $\pi$ over $[n]$ so that $G(x,y) = G'(\pi(x), \pi(y))$ for any $x \neq y \in [n]$, in the ordered setting, the automorphism group of a graph $G$ is trivial: $G$ is only isomorphic to itself through the identity function.

	For simplicity, we consider in the following only graphs (without edge colors).
	All results here can be generalized in a relatively straightforward manner to edge-colored graph-like ordered structures, where pairs of vertices may have one of $r \geq 2$ colors (the definition above corresponds to the case $r=2$). This is done by replacing the range $[0,1]$ with the $(r-1)$-dimensional simplex (corresponding to the set of all possible distributions over $[r]$).

	Two interesting special cases of two-dimensional ordered structures for which our results naturally yield a limit object are \emph{images}, i.e., ordered matrices, and \emph{dynamic graphs} with vertex insertions. Specifically, (binary) $m \times n$ images can be viewed as ordered bipartite graphs $I \colon [m] \times [n] \to \{0,1\}$, and our results can be adapted to get a bipartite ordered limit object for them as long as $m = \Theta(n)$. Meanwhile, a dynamic graph with vertex insertions can be viewed as a sequence $\{G_i\}_{i=1}^{\infty}$ of ordered graphs, where $G_{i+1}$ is the result of adding a vertex to $G_i$ and connecting it to the previous vertices according to some prescribed rule.
	It is natural to view such dynamic graphs that evolve with time as ordered ones, as the time parameter induces a natural ordering. Thus, our work gives, for example, a limit object for time-series where there are pairwise relations between events occurring at different times.

	As we shall see in Subsection~\ref{subsec:main_results}, the main results proved in this paper are, in a sense, natural extensions of results in the unordered setting. However, proving them requires machinery that is heavier than that used in the unordered setting: the tools used in the unordered setting are not rich enough to overcome the subtleties materializing in the ordered setting. In particular, the limit object we use in the ordered setting---which we call an \emph{orderon}---has a $4$-dimensional structure that is more complicated than the analogous $2$-dimensional structure of the graphon, the limit object for the unordered setting.
	The tools required to establish the ordered theory are described next.

	\subsection{Main ingredients}\label{subsec:main_ingredients}

	Let us start by considering a simple yet elusive sequence of ordered graphs, which has the makings of convergence. The \emph{odd-clique} ordered graph $H_n$ on $2n$ vertices is defined by setting $H_n(i,j) = 1$, i.e., having an edge between vertices $i$ and $j$, if and only if $i \neq j$ and $i,j$ are both odd, and otherwise setting $H_n(i, j) = 0$.
	In this subsection we closely inspect this sequence to demonstrate the challenges arising while trying to establish a theory for ordered graphs, and the solutions we propose for them.
	First, let us define the notions of subgraph frequency and convergence.

	The (induced) frequency $t(F, G)$ of a simple ordered graph $F$ on $k$ vertices in an ordered graph $G$ with $n$ vertices is the probability that, if one picks $k$ vertices of $G$ uniformly and independently (repetitions are allowed) and reorders them as $x_1 \leq \cdots \leq x_k$, $F$ is isomorphic to the induced ordered subgraph of $G$ over $x_1, \ldots, x_k$. (The latter is defined as the ordered graph $H$ on $k$ vertices satisfying $H(i, j) = G(x_i, x_j)$ for any $i,j \in [k]$.)  A sequence $\{G_n\}_{n=1}^{\infty}$ of ordered graphs is \emph{convergent} if $|V(G_n)| \to \infty$ as $n \to \infty$, and the frequency $t(F, G_n)$ of any simple ordered graph $F$ converges as $n \to \infty$. Observe that the odd-clique sequence $\{H_n\}$ is indeed convergent: The frequency of the empty $k$-vertex graph in $H_n$ tends to $(k+1)2^{-k}$ as $n \to \infty$, the frequency of any non-empty $k$-vertex ordered graph containing only a clique and a (possibly empty) set of isolated vertices tends to $2^{-k}$, and the frequency of any other graph in $H_n$ is $0$.\footnote{To see why the sum of frequencies is $1$, note that for $k \geq \ell \geq 2$, the number of $k$-vertex ordered graphs consisting of an $\ell$-vertex clique and $k-\ell$ isolated vertices is $\binom{k}{\ell}$.}

	In light of previous works on the unordered theory of convergence, we look for a limit object for ordered graphs that has the following features.
	\begin{description}
		\item[Representation of finite ordered graphs] The limit object should have a natural and consistent representation for finite ordered graphs. As is the situation with graphons, we allow graphs $H$ and $G$ to have the same representation when one is a blowup\footnote{A graph $G$ on $nt$ vertices is an ordered $t$-blowup of $H$ on $n$ vertices if $G(x,y) = H(\lceil x/t \rceil, \lceil y/t \rceil)$ for any $x$ and $y$.} of the other.
		\item[Usable distance notion] Working directly with the definition of convergence in terms of subgraph frequencies is difficult. The limit object we seek should be endowed with a metric, like the cut distance for unordered graphs (see discussion below), that should be easier to work with and must have the following property: A sequence of ordered graph is convergent (in terms of frequencies) if and only if it is Cauchy in the metric.
		\item[Completeness and compactness] The space of limit objects must be complete with respect to the metric: Cauchy sequences should converge in this metric space. Combined with the previous requirements, this will ensure that any convergent sequence of ordered graphs has a limit (in terms of ordered frequencies), as desired. It is even better if the space is compact, as compactness is essentially an ``ultimately strong'' version of Szemer\'edi's regularity lemma~\cite{Szemeredi78}, and will help to develop applications of the theory in other areas.
	\end{description}
	Additionally, we would like the limit object to be as simple as possible, without unnecessary over-representation.
	In the unordered setting, the metric used is the \emph{cut distance}, introduced by Frieze and Kannan~\cite{FriezeKannan1996, FriezeKannan1999} and defined as follows. First, we define the \emph{cut norm} $\|W\|_{\square}$ of a function $W \colon [0,1]^2 \to \mathbb{R}$ as the supremum of $|\int_{S \times T} W(x,y)dxdy|$ over all measurable subsets $S, T \subseteq [0,1]$. The \emph{cut distance} between graphons $W$ and $W'$ is the infimum of $\|W^{\phi} - W'\|_{\square}$ over all measure-preserving bijections $\phi \colon [0,1] \to [0,1]$, where $W^\phi(x,y) \eqdef W(\phi(x), \phi(y))$.

	For the ordered setting, we look for a similar metric; the cut distance itself does not suit us, as measure-preserving bijections do not preserve ordered subgraph frequencies in general. A first intuition is then to try graphons as the limit object, endowed with the metric $d_{\square}(W, W') \eqdef \| W - W' \|_{\square}$. However, this metric does not satisfy the second requirement: the odd-clique sequence is convergent, yet it is not Cauchy in $d_{\square}$, since $d_{\square}(H_n, H_{2n}) = 1/2$ for any $n$.
	Seeing that $d_{\square}$ seems ``too strict'' as a metric and does not capture the similarities between large odd-clique graphs well, it might make sense to use a slightly more ``flexible'' metric, which allows for measure-preserving bijections, as long as they do not move any of the points too far from its original location. In view of this, we define the \emph{cut-shift distance} between two graphons $W, W'$ as
	\begin{align}
	\label{eq:cut_shift}
	\csdist(W,W')\eqdef\inf_{f}\left(\shift{f}+\|W^f-W'\|_{\square}\right),
	\end{align}
	where $f \colon [0,1] \to [0,1]$ is a measure-preserving bijection, $\shift{f} = \sup_{x \in [0,1]} |f(x)-x|$, and $W^f(x, y) = W(f(x), f(y))$ for any $x,y \in [0,1]$. As we show in this paper (Theorem~\ref{thm:convergence_equiv} below), the cut-shift distance settles the second requirement: a sequence of ordered graphs is convergent \emph{if and only if} it is Cauchy in the cut-shift distance.

	Consider now graphons as a limit object, coupled with the cut-shift distance as a metric. Do graphons satisfy the third requirement? In particular, does there exist a graphon whose ordered subgraph frequencies are equal to the limit frequencies for the odd-clique sequence? The answers to both of these questions are negative: it can be shown that such a graphon cannot exist in view of Lebesgue's density theorem, which states that there is no measurable subset of $[0,1]$ whose density in every interval $(a,b)$ is $(b-a)/2$ (see e.g.\@ Theorem 2.5.1 in the book of Franks on Lebesgue measure~\cite{Franks09}).
	Thus, we need a somewhat richer ordered limit object that will allow us to ``bypass'' the consequences of Lebesgue's density theorem. Consider for a moment the graphon representations of the odd clique graphs. In these graphons, the domain $[0, 1]$ can be partitioned into increasingly narrow intervals that alternately represent odd and even vertices. Intuitively, it seems that our limit object needs to be able to contain infinitesimal odd and even intervals at any given location, leading us to the following limit object candidate, which we call an \emph{orderon}.

	An orderon is a symmetric measurable function $W \colon ([0,1]^2)^2 \to [0,1]$ viewed, intuitively and loosely speaking, as follows. In each point $(x, a) \in [0, 1]^2$, corresponding to an infinitesimal ``vertex'' of the orderon, the first coordinate, $x$, represents a location in the linear order of $[0,1]$. Each set $\{x\} \times [0,1]$ can thus be viewed as an infinitesimal probability space of vertices that have the same location in the linear order. The role of the second coordinate is to allow ``variability'' (in terms of probability) of the infinitesimal ``vertex'' occupying this point in the order. The definition of the frequency $t(F, W)$ of a simple ordered graph $F = ([k], E)$ in an orderon $W$ is a natural extension of frequency in graphons. First, define the random variable $\bG(k, W)$ as follows: Pick $k$ points in $[0,1]^2$ uniformly and independently, order them according to the first coordinate as $(x_1, a_1), \ldots, (x_k, a_k)$ with $x_1 \leq \cdots \leq x_k$, and then return a $k$-vertex graph $G$, in which the edge between each pair of vertices $i$ and $j$ exists with probability $W((x_i, a_i), (x_j, a_j))$, independently of other edges. The frequency $t(F, W)$ is defined as the probability that the graph generated according to $\bG(k, W)$ is isomorphic to $F$.

	Consider the orderon $W$ satisfying $W((x, a), (y,b)) = 1$ if and only if $a,b \leq 1/2$, and otherwise $W((x, a), (y, b)) = 0$. $W$ now emerges as a natural limit object for the odd-clique sequence: one can verify that the subgraph frequencies in it are as desired.

	The cut-shift distance for orderons is defined similarly to~\eqref{eq:cut_shift}, except that $f$ is now a measure-preserving bijection from $[0,1]^2$ to $[0,1]^2$ and $\shift{f} = \sup_{(x,a) \in [0,1]^2} |\pi_1(f(x,a)) - x|$, where $\pi_1(y, b) \eqdef y$ is the projection to the first coordinate.

	\subsection{Main results}\label{subsec:main_results}

	Let $\calW$ denote the space of orderons endowed with the cut-shift distance.
	In view of Lemma~\ref{lem:pseudo_metric} below, $\csdist$ is a pseudo-metric for $\calW$. By identifying $W,U\in\calW$ whenever $\csdist(W,U)=0$, we get a metric space $\widetilde{\calW}$. The following result is the main component for the viability of our limit object, settling the third requirement above.
	\begin{theorem}\label{thm:Compactness}
		The space $\widetilde{\calW}$ is compact with respect to $\csdist$.
	\end{theorem}
	The proof of Theorem~\ref{thm:Compactness} is significantly more involved than the proof of its unordered analogue. While at a very high level, the roadmap of the proof is similar to that of the unordered one, our setting induces several new challenges, and to handle them we develop new \emph{shape approximation} techniques. These are presented along the proof of the theorem in Section~\ref{sec:compactness}.

	The next result shows that convergence in terms of frequencies is equivalent to being Cauchy in $d_{\triangle}$. This settles the second requirement.

	\begin{theorem}\label{thm:convergence_equiv}
		Let $\{W_n\}_{n=1}^{\infty}$ be a sequence of orderons. Then $\{W_n\}$ is Cauchy in $d_{\triangle}$ if and only if $t(F, W_n)$ converges for any fixed simple ordered graph $F$.
	\end{theorem}

	As a corollary of the last two results, we get the following.

	\begin{corollary}
		For every convergent sequence of ordered graphs $\{G_n\}_{n\in\mathbb{N}}$, there exists an orderon $W\in\calW$ such that $t(F,G_n)\to t(F,W)$ for every ordered graph $F$.
	\end{corollary}

	The next main result is a sampling theorem, stating that a large enough sample from an orderon is almost always close to it in cut-shift distance.  For this, we define the orderon representation $W_G$ of an $n$-vertex ordered graph $G$ by setting $W_G((x, a), (y, b)) = G(Q_n(x), Q_n(y))$ for any $x,a,y,b$, where we define $Q_n(x) = \lceil nx \rceil$ for $x > 0$ and $Q_n(0) = 1$. This addresses the first requirement.

	\begin{theorem}\label{thm:Sampling}
		Let $k$ be a positive integer and let $W\in\calW$ be an orderon. Let $G \sim \bG(k, W)$. Then,
		$$
		\csdist(W,W_{G}) \leq C\left(\frac{\log \log k}{\log k}\right)^{1/3}
		$$
		holds with probability at least $1-C\exp(-\sqrt{k}/C)$ for some constant $C>0$.
	\end{theorem}

	Theorem~\ref{thm:Sampling} implies, in particular, that ordered graphs are a dense subset in $\calW$.

	\begin{corollary}\label{cor:density_simple}
		For every orderon $W$ and every $\eps>0$, there exists a simple ordered graph $G$ on at most $2^{\eps^{-3+o(1)}}$ vertices such that $\csdist(W,W_G)\le\eps$.
	\end{corollary}
	
			Our last main result asserts that any orderon $W \in \calW$ can be approximated in $L_1$-distance by an orderon $U$ with a finite block structure, with the added property that any ordered finite structure that appears with positive density in $U$ also has positive density in $W$.\footnote{A weaker result, in which the $L_1$-distance is replaced by the cut-shift distance, is not hard to prove using our previous main results; note that it is indeed strictly weaker since the $L_1$-distance between any two orderons $U$ and $W$ is always no larger (and sometimes strictly smaller) than $\csdist(U, W)$.}
The orderon $U$ is described as follows. the point set $[0,1]^2$ is divided into $b$ ``blocks'', which are subsets of the form $[(i-1)/b, i/b] \times [0,1]$ for some $i \in [b]$. Each block is decomposed into $l$ ``layers'', of the form $[(i-1)/b, i/b] \times [(j-1)/l, j/l]$ where $j \in [l]$. The value of $U((x,a), (y,b))$ is now only dependent on which blocks $x,y$ belong to, which layers $a,b$ belong to, and possibly whether $x < y$. For example, the orderon $U$ representing the limit of the odd-clique sequence (defined by $U((x,a), (y,b))=1$ if $a,b \leq 1/2$, and  $U((x,a), (y,b))=0$ elsewhere) has one block and two layers in it. Roughly speaking, one can think of such $U$ as the orderon representation of a ``pixelized'' ordered graph, where each vertex (block) consists of multiple ``pixels'' (block-layer pairs), and there is a weighted edge\footnote{In fact, a weighted bi-directed edge, with possibly different weights in the the different directions.} between each pair of pixels.
Therefore we call an orderon $U$ with such structure a \emph{pixelized} orderon and term our result the \emph{pixelization lemma}.
	\begin{theorem}[Pixelization lemma; informal]
	\label{thm:pixel}
	For any orderon $W$ and $\eps > 0$, there exists a pixelized orderon $U$ so that $d_1(U, W) \leq \eps$, satisfying the following: for all ordered graphs $F$ with $t(F,U) > 0$, we have $t(F,W)>0$.
	\end{theorem} 
	We note that the pixelized structure of $U$ is necessary for this statement to be correct; it breaks down if we only allow $U$ to be the orderon representation of a standard edge-weighted ordered graph. The pixelization lemma is formally stated and proved in Section \ref{sec:hereditary}; see Lemma \ref{lem:unifomDecision} there.
	
	The pixelization lemma is especially useful for applications where the $L_1$-distance comes into play. Two such applications, reproving the ordered graph removal lemma~\cite{ABF_FOCS17} and proving a new result in extremal combinatorics, are described next.

	\subsection{The furthest ordered graph from a hereditary property}\label{subsec:furthest_intro}
	Here and in the next subsection we describe three applications of our ordered limit theory. We start with an extensive discussion on the first application: A new result on the maximum edit\footnote{For our purposes, define the edit (or Hamming) distance between two ordered graphs $G$ and $G'$ on $n$ vertices as the smallest number of entries that one needs to change in the adjacency matrix $A_G$ of $G$ to make it equal to $A_{G'}$, divided by $n^2$. For this matter, the adjacency matrix $A_G$ of a graph $G$ over vertices $v_1 < \ldots < v_n$ is a binary $n \times n$ matrix where $A_G(i,j) = 1$ if and only if there is an edge between $v_i$ and $v_j$ in $G$. 
	The distance between $G$ and a property $\calP$ of ordered graphs is $\min_{G'} d_1(G,G')$ where $G'$ ranges over all graphs $G'$ of the same size as $G$. The definition for unordered graphs is similar; the only difference is in the notion of isomorphism.} distance $d_1(G, \calH)$ of an ordered graph $G$ from a hereditary\footnote{A property of (ordered or unordered) graphs is \emph{hereditary} if it is closed under taking induced subgraphs.} property $\calH$.

	For a hereditary property $\calH$ of simple ordered graphs, define $\overline{d_\calH} = \sup_G d_1(G, \calH)$ where $G$ ranges over all simple graphs (of any size).
	The parameter $d_{\calH}$ has been widely investigated for unordered graphs. A well-known surprising result of Alon and Stav~\cite{AlonStav2008} states, roughly speaking, that $d_{\calH}$ is always ``achieved'' by the Erd\H{o}s–R\'enyi random graph $\bG(n,p)$ for an appropriate choice of $p$ and large enough $n$.
	\begin{theorem}[\cite{AlonStav2008}]\label{thm:Alon-Stav-unordered}
	For any hereditary property $\calH$ of unordered graphs there exists $p_\calH \in [0,1]$ satisfying the following. A graph $G \sim \bG(n,p_{\calH})$ satisfies $d_1(G, \calH) \geq \overline{d_\calH} - o(1)$ with high probability.
	\end{theorem}
	In other words, a random graph $\bG(n,p_{\calH})$ is with high probability asymptotically (that is, up to relative edit distance of $o(1)$) the furthest from the property $\calH$.	From the analytic perspective, Lov\'asz and Szegedy~\cite{LS10} were able to reprove (and extend) this result using graph limits. 

	The surprising result of Alon and Stav has led naturally to a very interesting and highly non-trivial question, now known as the (extremal) \emph{graph edit distance problem}~\cite{Martin2016}, which asks the following: Given a hereditary property of interest $\calH$, what is the value (or values) $p_{\calH}$ that maximizes the distance of $\bG(n,p)$ from $\calH$? The general question of determining $p_{\calH}$ given any $\calH$ is currently wide open, although there have been many interesting developments for various classes of hereditary properties; see~\cite{Martin2016} for an extensive survey of previous works and useful techniques.

	While the situation in unordered graphs, and even in (unordered) directed graphs~\cite{Axenovich2011MulticolorAD} and matrices~\cite{Ryan06} has been thoroughly investigated, for ordered graphs no result in the spirit of Theorem~\ref{thm:Alon-Stav-unordered} is known. The first question that comes to mind is whether the behavior in the ordered setting is similar to that in the unordered case:
	Is it true that for any hereditary property $\calH$ of \emph{ordered} graphs there exists $p = p_{\calH}$ for which $G \sim \bG(n,p)$ satisfies $d_1(G, \calH) \geq \overline{d_\calH} - o(1)$  with high probability?

	As we show, the answer is in fact \emph{negative}. Consider the ordered graph property $\calH$ defined as follows: $G \in \calH$ if and only if there do not exist vertices $u_1 < u_2 \leq u_3 < u_4$ in $G$ where $u_1 u_2$ is a non-edge and $u_3 u_4$ is an edge. $\calH$ is clearly a hereditary property, defined by a finite family of forbidden ordered subgraphs.
	In the beginning of Section~\ref{sec:furthest}, we prove that the typical distance of $G \sim \bG(n,p)$ from $\calH$ is no more than $1/4 + o(1)$ (the maximum is asymptotically attained for $p=1/2$). In contrast, we show there exists a graph $G$ satisfying $d_1(G, \calH) = 1/2 - o(1)$, which is clearly the furthest possible up to the $o(1)$ term (every graph $G$ is $1/2$-close to either the complete or the empty graph, which are in $\calH$), and is substantially further than the typical distance of $\bG(n,p)$ for any choice of $p$. This shows that Theorem~\ref{thm:Alon-Stav-unordered} \emph{cannot be true} for the ordered setting. For the exact details, see Subection~\ref{subsec:furthest_tehnical}, which is completely elementary and self-contained.

	However, the news are not all negative: We present a positive result in the ordered setting, which generalizes the unordered statement in some sense, and whose proof makes use of our ordered limit theory. While it is no longer true that $\bG(n,p)$ generates graphs that are asymptotically the furthest from $\calH$, we show that a random graph generated according to a \emph{consecutive stochastic block model} is approximately the furthest. A \emph{stochastic block model}~\cite{Abbe2018} with $M$ blocks is a well-studied  generalization of $\bG(n,p)$, widely used in the study of community detection, clustering, and various other problems in mathematics and computer science.
	 A stochastic block model is defined according to the following three parameters: $n$, the total number of vertices; $(q_1, \ldots, q_M)$, a vector of probabilities that sum up to one; and a symmetric $M \times M$ matrix of probabilities $p_{ij}$. A graph on $n$ vertices is generated according to this model as follows. First, we assign each of the vertices independently\footnote{In some contexts, the stochastic block model is defined by determining the \emph{exact} number of vertices in each $A_i$ in advance, rather than assigning the vertices independently; all results here are also true for this alternative definition.} to one of $M$ parts $A_1, \ldots, A_M$, where the probability of any given vertex to fall in $A_i$ is $q_i$. Then, for any $(i,j) \in [M]^2$, and any pair of disjoint vertices $u \in A_i$ and $v \in A_j$, we add an edge between $u$ and $v$ with probability $p_{ij}$. By \emph{consecutive}, we mean that all vertices assigned to $A_i$ precede (in the vertex ordering) all vertices assigned to $A_{i+1}$, for any $i \in [M-1]$. Our main result now is as follows.

	\begin{theorem}\label{thm:ordered_Alon_Stav}
		Let $\calH$ be a hereditary property of simple ordered graphs and let $\eps > 0$. There exists a consecutive stochastic block model with at most $M = M_{\calH}(\eps)$ blocks with equal containment probabilities (i.e., $q_i = 1/M$ for any $i \in [M]$), satisfying the following. A graph $G$ on $n$ vertices generated by this model satisfies $d_1(G, \calH) \geq \overline{d_{\calH}} - \eps$ with probability that tends to one as $n \to \infty$.
	\end{theorem}
	The proof, given in Subsection~\ref{subsec:furthest-proof}, is a good example of the power of the analytic perspective, combining our ordered limit theory with standard measure-theoretic tools and a few simple lemmas proved in~\cite{LS10}.


	\subsection{Sampling and property testing}
	We finish by showing two additional applications of the ordered limit theory. These applications are somewhat more algorithmically oriented---concerning sampling and property testing---and illustrate the use of our theory for algorithmic purposes.
	The first of them is concerned with naturally estimable ordered graph parameters, defined as follows.
	\begin{definition} [naturally estimable parameter]\label{def:naturally_estimable}
		An ordered graph parameter $f$ is \emph{naturally estimable} if for every $\eps>0$ and $\delta>0$ there is a positive integer $k=k(\eps,\delta)>0$ satisfying the following. If $G$ is an ordered graph with at least $k$ nodes and $G|_{\bk}$ is the subgraph induced by a uniformly random ordered set of exactly $k$ nodes of $G$, then \[\Prx_{G|_{\bk}}[|f(G)-{f}(G|_{\bk})|>\eps]<\delta. \]
	\end{definition}
	The following result provides an analytic characterization of ordered natural estimability, providing a method to study estimation problems on ordered graphs from the analytic perspective.
	\begin{theorem}\label{thm:Param-Testing} Let $f$ be a bounded simple ordered graph parameter. Then, the following are equivalent:
		\begin{enumerate}
			\item \label{thm:Estim-b} $f$ is naturally estimable.
			\item \textbf{\label{thm:Estim-c}}For every convergent sequence $\{G_n\}_{n\in\mathbb{N}}$ of ordered simple graphs with $|V(G_n)|\to \infty$, the sequence of numbers $\{f(G_n)\}_{n\in\mathbb{N}}$ is convergent.
			\item \label{thm:Estim-d}There exists a functional $\hat{f}(W)$ over $\calW$ that satisfies the following:
			\begin{enumerate}
				\item  \label{thm:Estim-d1}$\hat{f}(W)$ is continuous with respect to $\csdist$.
				\item \label{thm:Estim-d2} For every $\eps>0$, there is $k=k(\eps)$ such that for every ordered graph $G$ with $|V(G)|\ge k$, it holds that $\left|\hat{f}(W_G)-f(G)\right|\le \eps$.
			\end{enumerate}
		\end{enumerate}
	\end{theorem}

	Our third application is a new analytic proof of the ordered graph removal lemma of~\cite{ABF_FOCS17}, implying that every hereditary property of ordered graphs (and images over a fixed alphabet) is testable, with one-sided error, using a constant number of queries.
	(For the relevant definitions, see~\cite{ABF_FOCS17} and Definition~\ref{def:naturally_estimable} here.)
	\begin{theorem}[\cite{ABF_FOCS17}]\label{thm:removal_lemma}
		Let $\calH$ be a hereditary property of simple ordered graphs, and fix $\eps, c > 0$. Then there exists $k = k(\calH, \eps, c)$ satisfying the following: For every ordered graph $G$ on $n \geq k$ vertices that is $\eps$-far from $\calH$, the probability that $G|_{\bk}$ does not satisfy $\calH$ is at least $1-c$.
	\end{theorem}

The proof of Theorem \ref{thm:removal_lemma} utilizes the analytic tools developed in this work, and  bypasses the need for many of the sophisticated combinatorial techniques from \cite{ABF_FOCS17}, resulting in an arguably cleaner proof.

	\subsection{Related and subsequent work}\label{subsec:related_work}

	The theory of graph limits has strong ties to the area of property testing, especially in the dense setting. Regularity lemmas for graphs, starting with the well-known regularity lemma of Szemer\'edi~\cite{Szemeredi78}, later to be joined by the weaker (but more efficient) versions of Frieze and Kannan~\cite{FriezeKannan1996, FriezeKannan1999} and the stronger variants of Alon~et~al.\@~\cite{AFKS2000}, among others, have been very influential in the development of property testing. For example, regularity was used to establish the testability of all hereditary properties in graphs~\cite{AlonShapira08}, the relationship between the testability and estimability of graph parameters~\cite{FischerNewman07}, and combinatorial characterizations of testability~\cite{AlonFischerNewmanShapira2009}.

	The analytic theory of convergence, built using the cut distance and its relation to the weak regularity lemma, has proved to be an interesting alternative perspective on these results. Indeed, the aforementioned results have equivalent analytic formulations, in which both the statement and the proof seem cleaner and more natural.
	A recent line of work has shown that many of the classical results in property testing of dense graphs can be extended to dense ordered graph-like structures, including vertex-ordered graphs and images. In~\cite{ABF_FOCS17}, it was shown that the testability of hereditary properties extends to the ordered setting (see Theorem~\ref{thm:removal_lemma} above). Shortly after, in~\cite{BF_CCC18} it was proved that characterizations of testability in unordered graphs can be partially extended to similar characterizations in ordered graph-like structures, provided that the property at stake is sufficiently ``well-behaved'' in terms of order.

	Graphons and their sparse analogues have various applications in different areas of mathematics, computer science, and even social sciences.
	The connections between graph limits and real-world large networks have been very actively investigated; see the survey of Borgs and Chayes~\cite{BorgsChayes17}.
	Graph limits have applications in probability and data analysis~\cite{OrbanzRoy15}.
	Graphons were used to provide new analytic proofs of results in extremal graph theory; see Chapter 16 in~\cite{Lov12}. Through the notion of free energy, graphons were also shown to be closely connected to the field of statistical physics~\cite{BCLSV_annals12}.
	We refer the reader to~\cite{Lov12} for more details.

	A subsequent independent work, by Garbe, Hancock, Hladky, and Sharifzadeh, investigates an alternative limit object for the ordered setting in the context of latin squares. See \cite{Latin_squares_2019} for their findings, as well as connections between orderons and their limit object, called a latinon. Compared to our work, which aims to construct a limit object for orderons from simple ``basic principles'', their techniques are more general but require heavier analytic machinery, including the disintegration theorem in measure theory. In particular, they show that a limit object equivalent to our orderons can be obtained through their framework. 
		
	Inspired by limit objects like permutons, orderons and latinons, an intriguing subsequent work by Simkin \cite{simkin2021number} asymptotically settles the famous $n$-queens problem, first raised in 1848. This problem asks for the total number of valid configurations involving $n$ queens (not allowed to share the same row, column, or diagonal) on an $n \times n$ chessboard. A central ingredient in Simkin's proof is a new limit object that he develops for queen configurations, called a \emph{queenon}. Interestingly, the result of the $n$-queens problem turns out to have a measure theoretic flavour, being the solution of an optimization problem in a space of Borel probability measures, and serving as an intriguing example for the usefulness of measure-theoretic tools in modern combinatorics.  
	
Finally, subsequent works by Coregliano and Razborov \cite{SemanticLimits2020} and by Towsner \cite{Towsner2021} investigate limit objects for ordered structures in generalized and/or higher order settings. The first work establishes a very general framework for limit theories by building upon deep connections to the area of flag algberas, applying ideas from model theory and logic. A variant of ordered graph removal that allows for modifying the order of the vertices can be derived from their theory. The second work is more closely related to ours, proving a generalization of the ordered graph removal lemma to ordered hypergraphs. As in our case, the latter work also relies on a corresponding limit theory, but it is interestingly different from ours: whereas our limit object is in a sense 4-dimensional even though it represents discrete objects (ordered graphs) in $2$-dimensions, the latter work is able to define a limit object whose ``dimensionality'' equals to that of the hypergraph, but at the expense of working with a more complicated measure space, which is based on Keisler graded probability spaces. We refer the interested reader to \cite{Towsner2021} for more details.

	\section{Preliminaries}\label{sec:preliminaries}
In this section we formally describe some of the basic ingredients of our theory, including the limit object---the \emph{orderon}, and several distance notions including the cut-norm for orderons (both unordered and ordered variants are presented), and the cut-shift distance. We then show that the latter is a pseudo-metric for the space of orderons. This will later allow us to view the space of orderons as a metric space, by identifying orderons of cut-shift distance $0$.

The measure used here is the Lebesgue measure, denoted by $\lambda$. We start with the formal definition of an orderon.

	\begin{definition}[orderon] An \emph{orderon} is a measurable function $W: \left([0,1]^2\right)^2\to [0,1]$ that is symmetric in the sense that $W((x,a),(y,b))=W((y,b),(x,a))$ for all $(x,a),(y,b)\in[0,1]^2$. For the sake of brevity, we also denote $W((x,a),(y,b))$ by $W(v_1,v_2)$ for $v_1,v_2\in[0,1]^2$.
	\end{definition}
	We denote the set of all orderons by $\calW$.

\begin{definition}[measure-preserving bijection]

	A map $g \colon[0,1]^2\to[0,1]^2$ is \emph{measure preserving} if the pre-image $g^{-1}(X)$ is measurable for every measurable set $X$ and $\lambda(g^{-1}(X))=\lambda(X)$.
	A \emph{measure preserving bijection} is a measure preserving map whose inverse map exists (and is also measure preserving).
\end{definition}
	Let $\mathcal{F}$ denote the collection of all measure preserving bijections from $[0,1]^2$ to itself. Given an orderon $W\in\calW$ and $f\in \calF$, we define $W^f$ as the unique orderon satisfying $W^f((x,a),(y,b)) = W(f(x,a),f(y,b))$ for any $x,a,y,b \in [0,1]$.
	Additionally, denote by $\pi_1 \colon [0,1]^2\to[0,1]$ the projection to the first coordinate, that is, $\pi_1(x,a)=x$ for any $(x,a)\in[0,1]^2$.

\subsection{Cut-norm and ordered cut-norm}
The definition of the (unordered) cut-norm for orderons is analogous to the corresponding definition for graphons.

\begin{definition}[cut-norm] Given a symmetric measurable function $W \colon ([0,1]^2)^2 \to \mathbb{R}$, we define the \emph{cut-norm} of $W$ as \[\|W\|_\square\eqdef \sup_{S,T\subseteq[0,1]^2}\left|\int_{(x,a)\in S \;(y,b)\in T}W((x,a),(y,b))dxdadydb\right|\;.\]
\end{definition}


As we are working with ordered objects, the following definition of \emph{ordered cut-norm} will sometimes be of use (in particular, see Section~\ref{sec:statistics}). Given $v_1,v_2\in[0,1]^2$, we write $v_1\le v_2$ to denote that $\pi_1(v_1)\le \pi_1(v_2)$. Let $\indi_{E}$ be the indicator function for the event $E$.
\begin{definition}[ordered cut-norm] Let $W \colon ([0,1]^2)^2 \to \mathbb{R}$ be a symmetric measurable function. The \emph{ordered cut norm} of $W$ is defined as
	\[  \|W\|_{\square'}=\sup_{S,T\subseteq[0,1]^2}\left|\int_{ (v_1,v_2)\in S \times T}W(v_1,v_2)\indi_{v_1\le v_2}dv_1dv_2 \right| \;.\]
\end{definition}
We mention two important properties of the ordered-cut norm. The first is a standard smoothing lemma, and the second is a relation between the ordered cut-norm and the unordered cut-norm.
\begin{lemma}\label{lem:smoothning}
	Let $W\in\calW$ and $\mu,\nu:[0,1]^2\to [0,1]$. Then, \[\left|\int_{v_1,v_2} \mu(v_1)\nu(v_2)W(v_1,v_2)\indi_{v_1\le v_2}dv_1dv_2\right|\le\|W\|_{\square'}.\]
\end{lemma}
\begin{proof}
  Fix partitions $\calS = \{S_i\}$ and $\calT = \{T_j\}$ of $[0,1]^2$.
  We show below that the claim holds when $\mu$ and $\nu$ are step functions on $\calS$ and $\calT$, respectively.
  Then, the proof is complete by the fact that all integrable functions are approximable in $L^1(([0,1]^2)^2)$ by step functions.

  Since $\mu$ and $\nu$ are step functions, we can write $\mu = \sum_i a_i \indi_{S_i}$ and $\nu =  \sum_j b_j \indi_{T_j}$ for some vectors $a \in  [0,1]^{|\calS|}$ and $b \in [0,1]^{|\calT|}$.
  We define
  \[
    f(a,b)\eqdef\int_{v_1,v_2} \mu(v_1)\nu(v_2)W(v_1,v_2)\indi_{v_1\le v_2}dv_1dv_2.
  \]
  When $a \in \{0,1\}^{|\calS|}$ and $b \in \{0,1\}^{|\calT|}$, we have
  \begin{align*}
    |f(a,b)|
    & = \left|\int \sum_i \sum_j a_i b_j \indi_{S_i}(v_1) \indi_{T_i}(v_2) W(v_1,v_2) dv_1 dv_2 \right| \\
    & = \left|\int_{\bigcup_{i:a_i=1}S_i }\int_{\bigcup_{j:b_j=1}T_j } W(v_1,v_2) dv_1 dv_2\right| \leq \|W\|_{\square'},
  \end{align*}
  where the last inequality follows from the definition of the ordered cut-norm.
  As $f(a,b)$ is bilinear in $a$ and $b$, and $|f(a,b)| \leq \|W\|_{\square'}$ for any $a \in \{0,1\}^{|\calS|}$ and $b \in \{0,1\}^{|\calT|}$, we have $|f(a,b)| \leq \|W\|_{\square'}$ for any $a \in [0,1]^{|\calS|}$ and $b \in [0,1]^{|\calT|}$.
\end{proof}
\begin{lemma}\label{lem:OrederedtoNoOrder}
	Let $W\colon ([0,1]^2)^2\to [-1,1]$ be a symmetric measurable function. Then,   \[\frac{\|W\|_{\square'}^2}{4}\le \|W\|_{\square}\le 2\|W\|_{\square'}\;.\]
\end{lemma}
\begin{proof}
	The inequality $\|W\|_{\square}\le 2\|W\|_{\square'}$ follows immediately from the fact that $W$ is symmetric.
	For the other inequality, let $\xi=\|W\|_{\square'}$, fix $\gamma>0$, and let $S,T\subseteq[0,1]^2$ be a pair of sets satisfying \[\left|\int_{ S \times T}W(v_1,v_2)\indi_{v_1\le v_2}dv_1dv_2\right|\ge\xi-\gamma\;.\]
	We partition $[0,1]^2$ into strips $\calI=\{I_1,\dots,I_{2/\xi}\}$, such that for every $j\in[2/\xi]$, $I_j=\left[\frac{(j-1)\xi}{2}, \frac{j\xi}{2}\right]\times[0,1]$.  For every $j\in[2/\xi]$, let $I^{(< j)}=\bigcup_{i<j}I_i$ (where $I^{(< 1)}=\emptyset$). Then,
	\begin{align*}
	\xi-\gamma&\le\left|\int_{ S \times T}W(v_1,v_2)\indi_{v_1\le v_2}dv_1dv_2\right|\\
	&\le\sum_{i\in [2/\xi]}   \left|\int_{ (S \cap I_i)\times (T\cap I_i)}W(v_1,v_2)\indi_{v_1\le v_2}dv_1dv_2\right|+\sum_{\substack{j\in [2/\xi]}}\left|\int_{ (S \cap I^{(<j)})\times (T\cap I_j)}W(v_1,v_2)\indi_{v_1\le v_2}dv_1dv_2\right|
	\end{align*}
	Note that by the fact that $|W(v_1,v_2)|\le1$ for all $v_1,v_2\in [0,1]^2$,
	\begin{align*}
	\sum_{i\in [2/\xi]}   \left|\int_{ (S \cap I_i)\times (T\cap I_i)}W(v_1,v_2)\indi_{v_1\le v_2}dv_1dv_2\right|\le \sum_{i\in[2/\xi]} \lambda(I_i\times I_i)\le \xi/2,
	\end{align*}
	and therefore,
	\[ \sum_{\substack{j\in [2/\xi]}}\left|\;\int_{ (S \cap I^{(<j)})\times (T\cap I_j)}W(v_1,v_2)\indi_{v_1\le v_2}dv_1dv_2\right|\ge\xi/2-\gamma.\]
	On the other hand, the above implies that there exists $j\in[2/\xi]$ such that   \[ \left|\int_{ (S \cap I^{(<j)})\times (T\cap I_j)}W(v_1,v_2)\indi_{v_1\le v_2}dv_1dv_2\right|\ge\xi^2/4-\xi\gamma/2,	\]
	Note that for every $(v_1,v_2) \in (S \cap I^{(<j)})\times (T\cap I_j)$, we have that $\indi_{v_1\le v_2}=1$, and thus \[ \|W\|_\square \ge \left|\int_{ (S \cap I^{(<j)})\times (T\cap I_j)}W(v_1,v_2)dv_1dv_2\right|\ge\frac{\xi^2}{4}-\frac{\gamma\xi}{2}\;.
	\]
	Since the choice of $\gamma$ is arbitrary, the lemma follows.
\end{proof}

\subsection{The cut and shift distance}
The next notion of distance is a central building block in this work. It can be viewed as a locality preserving variant of the unordered cut distance, which accounts for order changes resulting from applying a measure preserving function.
\begin{definition}
	Given two orderons $W,U\in \calW$ we define the \emph{CS-distance} (cut-norm+shift distance) as:
	\begin{align*}
	\csdist(W,U)\eqdef\inf_{f \in \mathcal{F}}\left(\shift{f}+\|W-U^f\|_\square\right),
	\end{align*}

	where $\shift{f}\eqdef \sup_{x,a \in[0,1]}\left|x-\pi_1(f(x,a))\right|$.
\end{definition}
\begin{lemma}\label{lem:pseudo_metric}
$\csdist$ is a pseudo-metric on the space of orderons.
\end{lemma}
\begin{proof}First note that non-negativity follows trivially from the definition. In addition, it is easy to see that $\csdist(W,W)=0$ for any orderon $W$.
For symmetry,
\begin{align*}
\csdist(W,U&)=
\inf_{g \in \mathcal{F}}\left(\shift{g}
+\|W-U^g\|_\square\right) =\inf_{g \in \mathcal{F}}\left(\shift{g^{-1}}
+\|W-U^g\|_\square\right)\\
&=\inf_{g^{-1} \in \mathcal{F}}\left(\shift{g^{-1}}
+\|W^{g^{-1}}-U\|_\square\right)=\inf_{f \in \mathcal{F}}\left(\shift{f}
+\|U-W^f\|_\square\right)\\
&=\csdist(U,W)\;.
\end{align*}
Where we used the fact that $g$ is a measure preserving bijection and that $\shift{g^{-1}} = \shift{g}$ for any $g \in \mathcal{F}$.

Consider three orderons $W,U,Z$. We now show that $\csdist(W,U)\le \csdist(W,Z)+\csdist(Z,U)$.
\begin{align*}
\csdist(W,U)&=
\inf_{f,g \in \mathcal{F}}\left(\shift{g^{-1} \circ f}+\|W-U^{g^{-1}\circ f}\|_\square\right)\\
&\le \inf_{f,g \in \mathcal{F}}\left(\shift{f}+\shift{g^{-1}}+\|W^g-U^{ f}\|_\square\right)\\
&\le \inf_{f \in \mathcal{F}} \left(\shift{f}+\|Z-U^{ f}\|_\square\right)+ \inf_{g \in \mathcal{F}}\Big( \shift{g}+\|W^g-Z\|_\square \Big )\\
&= \csdist(W,Z) + \csdist(Z, U)\;,
\end{align*}
where the first equality holds since $g^{-1} \circ f$ is a measure preserving bijection, and the last inequality follows from the triangle inequality; note that $\shift{g^{-1}} = \shift{g}$ for any $g \in \mathcal{F}$.
\end{proof}

\section{Block orderons and their density in $\calW$}\label{sec:block_orderons}

In this section we show that weighted ordered graphs are dense in the space of orderons coupled with the cut-shift distance. To start, we have to define the orderon representation of a weighted ordered graph, called a \emph{naive block orderon}. A naive $n$-block orderon is defined as follows.

\begin{definition}[naive block orderon]
	Let $m\in \mathbb{N}$ be an integer. For $z\in(0,1]$, we denote $Q_n(z) = \lceil nz \rceil$; we also set $Q_n(0)=1$.
	An \emph{$m$-block naive orderon} is a function $W\colon\left([0,1]^2\right)^2\to [0,1]$ that can be written as
	\[W((x,a),(y,b))=G(Q_n(x), Q_n(y))\;,\qquad\forall x,a,y,b\in[0,1]\;,\]
	for some weighted ordered graph $G$ on $n$ vertices.
\end{definition}
Following the above definition, we denote by $W_G$ the naive block orderon defined using $G$, and view $W_G$ as the orderon ``representing'' $G$ in $\calW$. Similarly to the unordered setting, this representation is slightly ambiguous (but this will not affect us).
Indeed, it is not hard to verify that two weighted ordered graphs $F$ and $G$ satisfy $W_F = W_G$ if and only if both $F$ and $G$ are blowups of some weighted ordered graph $H$. Here, a weighted ordered graph $G$ on $nt$ vertices is a \emph{$t$-blowup} of a weighted ordered graph $H$ on $n$ vertices if $G(x,y) = H(\lceil x/t \rceil, \lceil y/t \rceil)$ for any $x,y \in [nt]$.

%
 We call an orderon $U\in\calW$ a \emph{step function with at most $k$ steps}  if there is a partition $\calR=\{S_1,\ldots,S_k\}$ of $[0,1]^2$ such that $U$ is constant on every $S_i\times S_j$.

\begin{remark}[The name choices]
The definition of a step function in the space of orderons is the natural extension of a step function in graphons. Note that a naive block orderon is a special case of a step function, where the steps $S_i$ are rectangular (this is why we call these ``block orderons''). The ``naive'' prefix refers to the fact that we do not make use of the second coordinate in the partition.
\end{remark}


For every $W\in\calW$ and every partition $\calP=\{S_1,\ldots,S_k\}$ of $[0,1]^2$ into measurable sets, let $W_\calP\colon ([0,1]^2)^2\to[0,1]$ denote the step function obtained from $W$ by replacing its value at $((x,a),(y,b))\in S_i\times S_j$ by the average of $W$ on $S_i\times S_j$. That is, \[ W_\calP((x,a),(y,b))=\frac{1}{\lambda(S_i)\lambda(S_j)}\int_{ S_i \times S_j}W((x',a'),(y',b'))dx'da' dy'db'\;,\]
Where $i$ and $j$ are the unique indices such that $(x,a)\in S_i$ and $(y,b)\in S_j$, respectively.

The next lemma is an extension of the regularity lemma to the setting of Hilbert spaces.
\begin{lemma}[\cite{LS07} Lemma $4.1$]\label{lem:reg-lem-hilbert}
	Let $\{\calK_i\}_i$ be arbitrary non-empty subsets of a Hilbert space $\calH$. Then, for every $\eps>0$ and $f\in \calH$ there is an $m\le \lceil1/\eps^2\rceil$ and there are $f_i\in\calK_i$ ($1\le i\le k$) and $\gamma_1\ldots,\gamma_k\in \mathbb{R}$ such that for every $g\in \calK_{k+1}$ \[  |\la g, f- (\gamma_1f_1+\cdots+\gamma_k f_k)\ra|\le \eps \|f\|\|g\| \]
	\end{lemma}

%
%
%
%

 The next lemma is a direct consequence of Lemma~\ref{lem:reg-lem-hilbert}.
 \begin{lemma}\label{lem:reg-nostepping}
 	For every $W\in\calW$ and $\eps>0$ there is a step function $U\in \calW$ with at most $\lceil2^{8/\eps^2}\rceil$ steps such that  \[ \|{W-U}\|_\square\le \eps \;.\]
 \end{lemma}

 \begin{proof} We apply Lemma~\ref{lem:reg-lem-hilbert} to the case where the Hilbert space is $L^2([0,1]^4)$, and each $\calK_i$ is the set of indicator functions of product sets $S\times S$, where $S\subseteq [0,1]^2$ is a measurable subset. Then for  $f\in \calW $, there is an $f'=\sum_{j=1}^{k}\gamma_j f_j$, which is a step function with at most $2^k$ steps. Therefore, we get a step function $U\in \calW$ with at most $2^{\lceil 1/\eps^2\rceil}$ steps such that for every measurable set $S\subseteq[0,1]^2$ \[\left|\int_{v_1,v_2\in S \times S}(W(v_1,v_2)-U(v_1,v_2))dv_1dv_2\right|\le \eps \;.\] By the above and the fact that
\begin{align*}
  &\left|\int_{ v_1,v_2\in(S \cup T)\times (S\cup T)}(W(v_1,v_2)-U(v_1,v_2))dv_1dv_2\right|= \Bigg|\int_{ v_1,v_2\in S \times S}(W(v_1,v_2)-U(v_1,v_2))dv_1dv_2 \\
   &+2\cdot\int_{ v_1,v_2\in S \times T}(W(v_1,v_2)-U(v_1,v_2))dv_1dv_2+\int_{ v_1,v_2\in T \times T}(W(v_1,v_2)-U(v_1,v_2))dv_1dv_2\Bigg|\le \eps,
 	\end{align*}
 	we get that for any two measurable sets $S, T\subseteq [0,1]^2$, \[\left|\int_{ v_1,v_2\in S \times T}(W(v_1,v_2)-U(v_1,v_2))dv_1dv_2\right|\le 2\eps \;,\] which implies the lemma.
 \end{proof}

Similarly to the graphon case, the step function $U$ might not be a stepping of $W$. However, it can be shown that these steppings are almost optimal.
\begin{claim}\label{clm-stepping-operator} Let $W\in\calW$, let $U$ be a step function, and let $\calP$ denote the partition of $[0,1]^2$ into the steps of $U$. Then $\|W-W_\calP\|_\square\le 2\|W-U\|_\square$.
	\end{claim}
\begin{proof}
	The proof follows from the fact that $U=U_\calP$ and the fact that the stepping operator is contractive with respect to the cut norm. More explicitly, \[\|W-W_\calP\|_\square \le \|W-U\|_\square +\|U-W_\calP\|_\square=\|W-U\|_\square +\|U_\calP-W_\calP\|_\square\le 2\|W-U\|_\square \;.\]
\end{proof}

Using Lemma~\ref{lem:reg-nostepping} and Claim~\ref{clm-stepping-operator} we can obtain the following lemma.
\begin{lemma}\label{lem:reg-stepping}
	For every function $W\in \calW$ and every $\eps>0$, there is a partition $\calP$ of $[0,1]^2$ into at most $2^{\lceil32/\eps^2\rceil}$ sets with positive measure such that $\|W-W_\calP\|_\square\le \eps$.
	\end{lemma}
Using the above lemma, we can impose stronger requirements on our partition. In particular, we can show that there exists a partition of $[0,1]^2$ to sets of the same measure. Such a partition is referred to as an \emph{equipartition}. Also, we say that a partition $\calP$ \emph{refines} $\calP '$, if $\calP$ can be obtained from $\calP'$ by splitting each $P_j\in\calP '$ into a finite number of sets (up to sets of measure $0$).
\begin{lemma}\label{lem:reg-refinement} Fix some $\eps>0$. Let $\calP$ be an equipartition of $[0,1]^2$ into $k$ sets,  and fix $q\ge 2k^2\cdot 2^{162/\eps^2}$ such that $k$ divides $q$. Then, for any $W\in \calW$, there exists an equipartition $\calQ$ that refines $\calP$ with $q$ sets,  such that $\|W-W_\calQ\|_\square \le \frac{8\eps}{9}+\frac{2}{k}$.
\end{lemma}
\begin{proof} Let $\calP'=\{P'_1,\ldots,P'_{p'}\}$ be a partition of $[0,1]^2$ into $p'\le 2^{162/\eps^2}$ sets such that $\|W-W_{\calP'}\|_\square\le \frac{4\eps}{9}$, and let $\widetilde\calQ=\{ \widetilde Q_1,\ldots,  \widetilde Q_{ \tilde q}\}$ be a common refinement of $\calP$ and $\calP'$, with ${\tilde q}\le k\cdot 2^{162/\eps^2}$.  We construct an equipartition $\calQ$ as follows.
For every $i\in [k]$, consider all the sets $\widetilde Q^i_{1},\ldots,\widetilde Q^i_{\ell_i}\in \widetilde{Q}$ consisting of $P_i\in\calP$. For each $r\in [\ell_i]$ we let $a_r=\lfloor\frac{\lambda(\widetilde Q^i_r)}{q}\rfloor$ and partition $\widetilde Q^i_r$ into sets $\widetilde Q^i_{r,1},\ldots,\widetilde{Q}^i_{r,a_r}$, each of measure $1/q$, plus an exceptional part $\widetilde{Q}^i_{r,\text{ex}}$ which is the residual set. That is,
\[  \widetilde{Q}^i_r= \left(\bigcup_{b=1}^{a_r} \widetilde{Q}^i_{r,b}\right)\cup \widetilde{Q}^i_{r,\text{ex}} \;.\]
Next, for every $i\in[k]$ let $R_i=\bigcup_{r=1}^{\ell_i}\widetilde{Q}^i_{r,\text{ex}}$ and repartition each $R_i$ to sets of measure $1/q$ to get an equipartition $\calQ$ of size $q$. Let $U$ be a step function that agrees with $W_{\widetilde{\calQ}}$ on $([0,1]^2\setminus \bigcup_{i\in[k]}R)^2$ and $0$ on the complement. Since $U$ disagrees with $W_{\widetilde{Q}}$ on a set of measure at most $2\lambda(R)\le \frac{2k \cdot 2^{162/\eps^2}}{q}$, we have that
\[\|W-U\|_\square\le \|W-W_{\widetilde{Q}}\|_\square+\|W_{\widetilde{Q}}-U\|_\square \le \frac{4\eps}{9}+\frac{2k\cdot 2^{162/\eps^2}}{q}\;. \]
By our choice of $q\ge 2k^2\cdot 2^{162/\eps^2} $ we get that
\[  \|W-U\|_\square\le \frac{4\eps}{9}+\frac{1}{k}\;.\]
By construction $U$ is a step function with steps in $\calQ$, and using Claim~\ref{clm-stepping-operator} we get that \[\|W-W_\calQ\|_\square\le 2\|W-U\|_\square\le \frac{8\eps}{9}+\frac{2}{k}\;,\] and the proof is complete.
\end{proof}

The next lemma is an (easier) variant of Lemma~\ref{lem:reg-refinement}, in the sense that we refine two given partitions. However, the resulting partition will not be an equipartition.

\begin{lemma}\label{lem:reg-refinementV2}
	Fix some $\eps>0$ and $d\in\mathbb{N}$. Let $\calI_d$ be an equipartition of $[0,1]^2$ into $2^d$ sets, $\calP$ be a partition of $[0,1]^2$ into $k$ sets, and fix $q\ge 2 (k\cdot 2^{d})^2\cdot 2^{162/\eps^2}$ such that both $k$ and $2^d$ divide $q$. Then, for any $W\in \calW$, there exists a partition $\calQ$ that refines both $\calP$ and $\calI_d$ with $q$ sets,  such that $\|W-W_\calQ\|_\square \le \frac{8\eps}{9}+\frac{2}{k\cdot 2^{d}}$.
\end{lemma}
\begin{proof}
	 Let $\calP'=\{P'_1,\ldots,P'_{p'}\}$ be a partition of $[0,1]^2$ into $p'\le 2^{162/\eps^2}$ sets such that $\|W-W_{\calP'}\|_\square\le \frac{4\eps}{9}$, and let $\calQ=\{  Q_1,\ldots,   Q_{ q}\}$ be a common refinement of the three partitions $\calP$, $\calP'$ and $\calI_d$. Note that we do not repartition further to get an equipartition.  The rest of the proof is similar to the proof of Lemma~\ref{lem:reg-refinement}.\end{proof}

The following theorem shows that naive block orderons are a dense subset in $\calW$.

\begin{theorem}\label{thm:ApproxNaive-StepOrderon}
	For every orderon $W\in \calW$ and every $\eps>0$, there exist a naive $\frac{c}{\eps^4}2^{162/\eps^2}$-block orderon $W'$ (for some constant $c>0$) such that \[\csdist(W,W')\le\eps\;. \]
\end{theorem}
\begin{proof} Fix $\eps>0$ and $\gamma=\gamma(\eps)>0$. We consider an interval equipartition $J=\{J_1,\ldots,J_{1/\gamma}\}$ of $[0,1]$ (namely, for each $j\in[\frac{1}{\gamma}-1]$, $J_j=\left[(j-1)\cdot\gamma,j\cdot\gamma\right)$, and for $j=1/\gamma$, $J_j=\left[(j-1)\cdot\gamma,j\cdot\gamma\right]$).  In addition, let  $\calP=(J_i\times J_j\mid i,j\in[1/\gamma])$ be an equipartition of $[0,1]^2$. By Lemma~\ref{lem:reg-refinement}, there exists an equipartition $\calQ$ of $[0,1]^2$ of size $q= \frac{2}{\gamma^4}2^{162/\eps^2}$ that refines $\calP$, such that \[\|W-W_\calQ\|_\square\le \frac{8\eps}{9}+{2\gamma^2}\;.\]

	Next we construct a small shift measure preserving function $f$ as follows. For every $i\in [1/\gamma] $, consider the collection of sets $\{Q^{i}_{k}\mid k \in [\gamma q] \}$ in $\calQ$ such that \[(J_i\times [0,1])\cap \calQ =\{Q^{i}_{k}\mid k \in [\gamma{q}] \}\;.\]
	For each $k\in[\gamma q]$, the function $f$ maps $Q^{i}_{k}$ to a rectangular set \[ \left[(i-1)\gamma+\frac{(k-1)}{{q}} ,(i-1)\gamma+\frac{k}{{q}} \right)\times [0,1]\;.\]
	Finally, for every $i,j\in [q]$  and every $(x,a), (y,b)\in Q_i\times Q_j$, we define  \[W'(f(x,a),f(y,b))=W_\calQ ((x,a),(y,b))\]
	Note that the resulting function $W'$ obeys the definition of a naive $q$-block orderon and $\shift{f}\le \gamma$. Therefore, setting $\gamma=\eps/100$, we get that
	\[\csdist(W,W')\le \gamma+ \frac{8\eps}{9}+2\gamma^2\le \eps/100 +8\eps/9 +2\eps^2/100^2\le \eps\;,\]
	as desired.
\end{proof}


\section{Compactness of the space of orderons}
\label{sec:compactness}

In this section we prove Theorem~\ref{thm:Compactness}. We construct a metric space $\widetilde\calW$ from $\calW$ with respect to $\csdist$, by identifying $W,U\in\calW$ with $\csdist(W,U)=0$. Let $\widetilde{\calW}$ be the image of $\calW$ under this identification. On $\widetilde{\calW}$ the function $\csdist$ is a distance function.

We start with some definitions and notations.
Let $(\Omega,\calM,\lambda)$ be some probability space, $\calP_\ell=\left\{P^{(\ell)}_i\right\}_i$  a partition of $\Omega$, and let $\beta\left(\calP_\ell:\cdot\right)\colon \calP_\ell\to [0,1]$ be a function. For $v\in \Omega$, we slightly abuse notation and  write $\beta\left(\calP_\ell:v\right)$ to denote $\beta\left(\calP_\ell:i\right)$ for $v\in P^{(\ell)}_i$. With this notation, observe that for every $\ell$
\begin{align}\label{Eq:SumToIntegral}
\int_{v\in \Omega}\beta(\calP_{\ell}:v)dv=\sum_{i\in[|\calP_{\ell}|]}\lambda\left(P^{(\ell)}_i\right)\beta(\calP_{\ell}:i)\;.\end{align}

The following two results serve as useful tools to prove convergence.
The first result is known as the \emph{martingale convergence theorem}, see e.g.\@ Theorem A.12 in~\cite{Lov12}. The second result is an application of the martingale convergence theorem, useful for our purposes.

\begin{theorem}[see~\cite{Lov12}, Theorem A.12]\label{thm:MartingaleConvergence}
	Let $\{\bX_i\}_{i\in\mathbb{N}}$ be a martingale such that $\sup_n \Ex[|\bX_n|]<\infty$. Then $\{\bX_i\}_{i\in\mathbb{N}}$ is convergent with probability $1$
\end{theorem}

\begin{lemma}\label{lem:PartitionConvergence} Let $\{\calP_\ell\}_\ell$ be a sequence of partitions of $\Omega$ such that for every $\ell$, $\calP_{\ell+1}$ refines $\calP_{\ell}$. Assume that for every $\ell$ and $j\in [|\calP_{\ell}|]$, the functions $\beta(\calP_{\ell}:\cdot)$ satisfy \begin{align}\label{Eq:MartingaleCondition} \lambda\left(P^{(\ell)}_j\right) \beta\left(\calP_\ell:j\right)=\sum_{i\in[|\calP_{\ell+1}|]}\lambda\left(P^{(\ell)}_j\cap P^{(\ell+1)}_i\right)\beta(\calP_{\ell+1}:i).\end{align}
	Then,  there is a measurable function $\beta \colon \Omega\to [0,1]$ such that $\beta(v)=\lim\limits_{\ell\to\infty}\beta(\calP_{\ell}:v)$ for almost all $v\in \Omega$.
\end{lemma}
\begin{proof} Fix some $\ell\in \mathbb{N}$. Let $\bX$ be a uniformly distributed random variable in $\Omega$. Let $\psi_\ell\colon \Omega\to [|\calP_{\ell}|]$ be the function mapping each $v\in\Omega$ to its corresponding part in $\calP_{\ell}$ and let $\bZ_\ell=\beta(\calP_{\ell}:\bX)$.  We now show that the sequence $(\bZ_1,\bZ_2\ldots)$ is a martingale. That is, $\Ex_{\bX\sim\Omega}\left[\bZ_{\ell+1}\mid \bZ_1,\ldots,\bZ_\ell\right]=\bZ_\ell$, for every $\ell\in\mathbb{N}$. Note that by the fact that $\calP_{\ell+1}$ refines $\calP_{\ell}$, $\psi_\ell(\bX)$ determines $\psi_i(\bX)$ for every $i<\ell$. By  definition, the value $\beta(\calP_{\ell}:\bX)$  is completely determined by $\psi_\ell(\bX)$, and so it suffices to prove that $\bZ_\ell=\Ex_{\bX\sim\Omega}\left[\bZ_{\ell+1}\mid\psi_\ell(\bX) \right]$. By the fact that for every $j\in[|\calP_{\ell}|]$ Equation~(\ref{Eq:MartingaleCondition}) holds (and in particular holds for $\psi_\ell(\bX)$), we can conclude that the sequence $(\bZ_1,\bZ_2,\ldots)$ is a martingale.

Since $\bZ_\ell$ is bounded, we can invoke the martingale convergence theorem (Theorem~\ref{thm:MartingaleConvergence}) and conclude that $\lim\limits_{\ell\to\infty}\bZ_\ell$ exists with probability $1$. That is, $\beta(v)=\lim\limits_{\ell\to\infty}\beta(\calP_{\ell}:v)$ exists for almost all $v\in\Omega$.\end{proof}

\begin{definition}
Fix some $d\in \mathbb{N}$ and define $\calI_d=\left\{I^{(d)}_1,\ldots,I^{(d)}_{2^d}\right\}$ so that for every $t\in\left[2^d\right]$, $I^{(d)}_t= \Big[\frac{t-1}{2^d},\frac{t}{2^d}\Big)\times [0,1]$. We refer to this partition as the \emph{strip partition of order $d$}. \end{definition}

The next lemma states that for any orderon $W$ we can get a sequence of partitions $\{\calP_{\ell}\}_\ell$, with several properties that will be useful later on.

\begin{lemma}\label{lem:PartitionSeq}
	For any orderon $W\in \calW$ and $\ell\in \mathbb{N}$, there is a sequence of partitions $\{\calP_{\ell}\}_\ell$ of $[0,1]^2$ with the following properties.
	\begin{enumerate}
		\item $\calP_{\ell}$ has $g(\ell)$ many sets (for some monotone increasing $g\colon\mathbb{N}\to\mathbb{N}$).
		\item\label{Item:divisibility} For every $\ell$, $\Gamma_{\ell}\eqdef\frac{g(\ell)}{g(\ell-1)} \in \mathbb{N}$.
		\item \label{Item:divisibility2}For every $\ell'\ge\ell$, the partition $\calP_{\ell'}$ refines both $\calP_{\ell}$ and the strip partition $\calI_{\ell'}$. In particular, for every $j\in[g(\ell-1)]$,
		$$P^{(\ell-1)}_j=\bigcup_{j'=(j-1)\cdot\Gamma_{\ell}+1}^{j\cdot\Gamma_{\ell}} P^{(\ell)}_{j'}\;.$$

		\item $W_{\ell}=(W)_{\calP_{\ell}}$ satisfies $\|W-W_{\ell}\|_\square\le \frac{4}{g(\ell-1)2^\ell}$.
	\end{enumerate}
\end{lemma}

\begin{proof}
We invoke Lemma~\ref{lem:reg-refinementV2} with the trivial partition  $\{[0,1]^2\}$ and the strip partition $\calI_1$, to get a partition $\calP_{n,1}$ with  $g(1)$ many sets such that $\calP_{n,1}$ refines $\calI_1$ and $\|W_n-W_{n,1}\|_\square\le 1$.
For $\ell>1$, we invoke Lemma~\ref{lem:reg-refinementV2} with $\calI_\ell$ and $\calP_{n,\ell-1}$ to get a partition $\calP_{n,\ell}$ of size $g(\ell)=(g(\ell-1)\cdot 2^\ell)^2\cdot 2^{O(g(\ell-1)^2)}$ which refines both $\calI_\ell$ and $\calP_{n,\ell-1}$ such that $\|W_n-W_{n,\ell}\|_\square\le \frac{4}{g(\ell-1)2^\ell}$. In order to take care of divisibility, we add empty (zero measure) sets in order to satisfy items (\ref{Item:divisibility}) and (\ref{Item:divisibility2}). \end{proof}

\medskip

Consider a sequence of orderons $\{W_n\}_{n\in\mathbb{N}}$. For every $n\in \mathbb{N}$, we use Lemma~\ref{lem:PartitionSeq} to construct a sequence of functions $\{W_{n,\ell}\}_\ell$ such that $\|W_n-W_{n,\ell}\|_\square$ is small. For each $\ell$, we would like to approximate the shape of the limit partition resulting from taking $n\to\infty$. Inside each strip $I^{(\ell)}_t$, we consider the relative measure of the  intersection of each set contained in $I^{(\ell)}_t$, with a finer strip partition $\calI_{\ell'}$.

\begin{definition}[shape function]\label{def:Convergence-alpha-nl}For fixed $n\in \mathbb{N}$, let $\{\calP_{n,\ell}\}_\ell$ be partitions of $[0,1]^2$ with the properties listed in Lemma~\ref{lem:PartitionSeq}. For every $\ell'>\ell$ and $I^{(\ell')}_{t'}\in \calI_{\ell'}$, we define $\alpha^{(n,\ell)}_{j}(\calI_{\ell'}:t')\eqdef2^{\ell'}\cdot\lambda\left(P^{(n,\ell)}_{j}\cap I^{(\ell')}_{t'}\right)$ to be the relative volume of the set $P^{(n,\ell)}_{j}$ in $I^{(\ell')}_{t'}$.

	For any $\ell'\ge \ell$ and $I^{(\ell')}_{t'}\in \calI_{\ell'}$, by the compactness of $[0,1]$, we can select a subsequence of $\{W_n\}_{n\in\mathbb{N}}$ such that $\alpha^{(n,\ell)}_{j}(\calI_{\ell'}:t')$ converges for all   $j\in[g(\ell)]$ as $n\to\infty$. Let \begin{align*}
	\alpha^{(\ell)}_{j}(\calI_{\ell'}:t')&\eqdef\lim\limits_{n\to\infty}\alpha^{(n,\ell)}_{j}(\calI_{\ell'}:t')\;.
	\end{align*}

\end{definition}

Next we define the limit density function.
\begin{definition}[density function]\label{def:Convergence-delta-nl} For fixed $n\in \mathbb{N}$, let $\{\calP_{n,\ell}\}_\ell$ be partitions of $[0,1]^2$ with the properties listed in Lemma~\ref{lem:PartitionSeq}. We let $\delta^{(n,\ell)}\left(\calP_{n,\ell}\times \calP_{n,\ell}: i,j\right)\eqdef W_{n,\ell}((x,a),(y,b))$ for $(x,a)\in P^{(n,\ell)}_i$ and $(y,b)\in P^{(n,\ell)}_j$.

	By the compactness of $[0,1]$, we can select a subsequence of $\{W_n\}_{n\in\mathbb{N}}$  such that $\delta^{(n,\ell)}(\calP_{n,\ell}\times\calP_{n,\ell}:i,j)$ converge for all   $i,j\in[g(\ell)]$ as $n\to\infty$. Let \begin{align*}
	\delta^{(\ell)}(i,j)&\eqdef\lim\limits_{n\to\infty}\delta^{(n,\ell)}(\calP_{n,\ell}\times\calP_{n,\ell}:i,j)\;.
	\end{align*}

\end{definition}

The following lemma states that by taking increasingly refined strip partitions $\calI_{\ell'}$, we obtain a limit shape function for each set contained in any strip of $\calI_{\ell}$.

\begin{lemma}\label{lem:martingaleAlphas}For fixed $\ell$ and $j\in [g(\ell)]$, there is a measurable function $\alpha^{(\ell)}_j\colon [0,1]\to [0,1]$  such that $ \alpha^{(\ell)}_j(x)=\lim\limits_{\ell'\to \infty}\alpha^{(\ell)}_j\left(\calI_{\ell'}:x\right)$  for almost all $x\in[0,1]$.
\end{lemma}

\begin{proof} Fix $n,\ell$ and $\ell'>\ell$. For every $j\in[g(\ell)]$, by the definition of $\alpha^{(n,\ell)}_j(\calI_{\ell'}:t')$ and the strip partition $\calI_{\ell'}$
	\[\lambda\left(I^{(\ell')}_{t'}\right)\cdot\alpha^{(n,\ell)}_j(\calI_{\ell'}:t') = \lambda\left(P^{(n,\ell)}_j\cap I^{(\ell')}_{t'}\right)\qquad\forall t'\in\left[2^\ell\right] .\]
On the other hand, since $\calI_{\ell'+1}$ refines $\calI_{\ell'}$,
\begin{align*}
\lambda\left(P^{(n,\ell)}_j\cap I^{(\ell')}_{t'}\right)&=\lambda\left(P^{(n,\ell)}_j\cap I^{(\ell'+1)}_{2t'-1}\right)+\lambda\left(P^{(n,\ell)}_j\cap I^{(\ell'+1)}_{2t'}\right)\\
&= \lambda\left(I^{(\ell'+1)}_{2t'-1}\right)\cdot\alpha^{(n,\ell)}_j(\calI_{\ell'+1}:2t'-1)+\lambda\left(I^{(\ell'+1)}_{2t'}\right)\cdot\alpha^{(n,\ell)}_j(\calI_{\ell'+1}:2t')\;.
\end{align*}
Therefore, when $n\to\infty$ we get that,
\[ \lambda\left(I^{(\ell')}_{t'}\right)\cdot\alpha^{(\ell)}_j(\calI_{\ell'}:t')=\lambda\left(I^{(\ell'+1)}_{2t'-1}\right)\cdot\alpha^{(\ell)}_j(\calI_{\ell'+1}:2t'-1)+\lambda\left(I^{(\ell'+1)}_{2t'}\right)\cdot\alpha^{(\ell)}_j(\calI_{\ell'+1}:2t')\;, \]
which is exactly the condition in Equation~(\ref{Eq:MartingaleCondition}).
By applying Lemma~\ref{lem:PartitionConvergence} with the sequence of strip partitions $\{\calI_{\ell'}\}_{\ell'}$ on $\alpha^{(\ell)}_j$ the lemma follows.
\end{proof}

The next lemma asserts that the limit shape functions behave consistently.

\begin{lemma}\label{lem:RefinementLemma}
	For every $\ell$ and $j\in [g(\ell-1)]$,
	\[ \alpha^{(\ell-1)}_j(x)=\sum_{j'=(j-1)\cdot\Gamma_{\ell}+1}^{j\cdot\Gamma_{\ell}}\alpha^{(\ell)}_{j'}(x)\;,
	\] for almost all $x\in[0,1]$.
\end{lemma}
\begin{proof} Fix some $n,\ell$ and $\ell'>\ell$. By the additivity of the Lebesgue measure,
	\[\alpha^{(n,\ell-1)}_j(\calI_{\ell'}:x)=\sum_{j'=(j-1)\cdot\Gamma_{\ell}+1}^{j\cdot\Gamma_{\ell}}\alpha^{(n,\ell)}_{j'}(\calI_{\ell'}:x)\qquad\forall x\in[0,1]\;.
	\]
	By the fact that for every $j\in[g(\ell-1)]$ and $x\in[0,1]$ the sequence $\left\{\alpha^{(n,\ell-1)}_j(\calI_{\ell'}:x)\right\}_n$ converges to $\alpha^{(\ell-1)}_j(\calI_{\ell'}:x)$ as $n\to \infty$, we get that
	\[\alpha^{(\ell-1)}_j(\calI_{\ell'}:x)=\sum_{j'=(j-1)\cdot\Gamma_{\ell}+1}^{j\cdot\Gamma_{\ell}}\alpha^{(\ell)}_{j'}(\calI_{\ell'}:x)\qquad\forall x\in[0,1]\;.
	\]
	By applying Lemma~\ref{lem:martingaleAlphas} on each $j'\in[g(\ell)]$, where $\ell'\to \infty$, we get that \[\alpha^{(\ell-1)}_j(x)=\sum_{j'=(j-1)\cdot\Gamma_{\ell}+1}^{j\cdot\Gamma_{\ell}}\alpha^{(\ell)}_{j'}(x)\;,\] for almost all $x\in[0,1]$.
\end{proof}

Using the sequence of $\left\{\alpha^{(\ell)}_j\right\}_{j}$ we define a limit partition $\calA_\ell=\left\{A^{(\ell)}_1,\ldots,A^{(\ell)}_{g(\ell)}\right\}$ of $[0,1]^2$ as follows.

\begin{definition}[limit partition]\label{def:A-partition}
	For every $\ell\in \mathbb{N}$, let $\calA_\ell=\left\{A^{(\ell)}_1,\ldots,A^{(\ell)}_{g(\ell)}\right\}$ be a partition of $[0,1]^2$ such that,
	\[A^{(\ell)}_j=\left\{(x,a)\colon \sum_{i<j}\alpha^{(\ell)}_i(x)\le a <\sum_{i\le j}\alpha^{(\ell)}_i(x)\right\}\qquad \forall j\in[g(\ell)]\;.\]
\end{definition}

\begin{lemma}\label{lem:LimitingPartitionProp}
	For any $\ell$, the partition $\calA_\ell$ has the following properties
	\begin{enumerate}
		\item $\calA_\ell$ refines the strip partition $\calI_\ell$.
		\item The partition $\calA_\ell$ refines $\calA_{\ell-1}$.
		\item For every $j\in[g(\ell)]$, $\lambda\left(A^{(\ell)}_j\right)=\lim\limits_{n\to\infty}\lambda\left(P^{(n,\ell)}_j\right)$.
	\end{enumerate}
	\end{lemma}
\begin{proof}
	The first item follows by the fact that each $\alpha^{(\ell)}_j$ is non-zero inside only one strip.

	By the definition of the sets $A^{(\ell)}_j$ and Lemma~\ref{lem:RefinementLemma} it follows that for each $j\in[g(\ell-1)]$,
	\[   A^{(\ell)}_{j'}\subset  A^{(\ell-1)}_j \qquad \text{for all}\qquad (j-1)\cdot\Gamma_{\ell}+1\le j'\le j\cdot\Gamma_{\ell},\]
	and therefore,
	\[A^{(\ell-1)}_j=\bigcup_{j'=(j-1)\cdot\Gamma_{\ell}+1}^{j\cdot\Gamma_{\ell}}A^{(\ell)}_{j'},  \]

	which shows the second item. To prove the third item of the lemma, note that for every $n,\ell$ and $\ell'>\ell$,
	\begin{align*}
	\lim\limits_{n\to\infty}\lambda\left(P^{(n,\ell)}_j\right)&=\lim\limits_{n\to\infty}\sum_{t'\in\left[2^{\ell'}\right]}2^{-\ell'}\cdot\alpha^{(n,\ell)}_j(\calI_{\ell'}:t')=\sum_{t'\in\left[2^{\ell'}\right]}2^{-\ell'}\cdot\alpha^{(\ell)}_j(\calI_{\ell'}:t')=\int_{x} \alpha^{(\ell)}_j(\calI_{\ell'}:x)dx,	\end{align*}
	where the last equality follows from Equation~(\ref{Eq:SumToIntegral}). Finally, by taking $\ell'\to\infty$ and using Lemma~\ref{lem:martingaleAlphas}, we get
	\[\lim\limits_{n\to\infty}\lambda\left(P^{(n,\ell)}_j\right)=\int_{x} \alpha^{(\ell)}_j(x)dx=\lambda\left(A^{(\ell)}_j\right)\;.\]
\end{proof}

Using the definition of $\delta^{(\ell)}$ and $\calA_\ell$, we define a density function on the limit partition. For $(x,a)\in A^{(\ell)}_i$ and $(y,b)\in A^{(\ell)}_j$, let \[\delta\left(\calA_\ell\times \calA_\ell: (x,a),(y,b)\right)\eqdef\delta^{(\ell)}(i,j)\;.\]

\begin{lemma}\label{lem:deltaMartingale} For each $\ell\in\mathbb{N}$ and $i,j\in[g(\ell-1)]$,
	\begin{align*}
	\sum_{i'=(i-1)\cdot\Gamma_{\ell}+1}^{i\cdot\Gamma_{\ell}}\;\sum_{j'=(j-1)\cdot\Gamma_{\ell}+1}^{j\cdot\Gamma_{\ell}}\lambda\left(A^{(\ell)}_{i'}\right)\cdot\lambda\left(A^{(\ell)}_{j'}\right)&\delta\left(\calA_\ell\times \calA_\ell: i',j'\right)\\
	&=\lambda\left(A^{(\ell-1)}_{i}\right)\cdot\lambda\left(A^{(\ell-1)}_{j}\right)\delta\left(\calA_{\ell-1}\times \calA_{\ell-1}: i,j\right)\;.
	\end{align*}
\end{lemma}
\begin{proof}
	Fix $n,\ell$ and $i,j\in [g(\ell-1)]$. By the definition of the partitions $\calP_{n,\ell}$, $\calP_{n,\ell-1}$ and the density functions $\delta^{(n,\ell)}$, $\delta^{(n,\ell-1)}$
	\begin{align*}
	\sum_{i'=(i-1)\cdot\Gamma_{\ell}+1}^{i\cdot\Gamma_{\ell}}\;\sum_{j'=(j-1)\cdot\Gamma_{\ell}+1}^{j\cdot\Gamma_{\ell}}\lambda\left(P^{(n,\ell)}_{i'}\right)\cdot\lambda\left(P^{(n,\ell)}_{j'}\right)&\delta^{(n,\ell)}\left(\calP_{n,\ell}\times \calP_{n,\ell}: i',j'\right)\\
	&=\lambda\left(P^{(n,\ell-1)}_{i}\right)\cdot\lambda\left(P^{(n,\ell-1)}_{j}\right)\delta\left(\calP_{n,\ell-1}\times \calP_{n,\ell-1}: i,j\right)\;.
	\end{align*}
	By taking the limit as $n\to\infty$ and using the third item of Lemma~\ref{lem:LimitingPartitionProp},
	\begin{align*}
	\sum_{i'=(i-1)\cdot\Gamma_{\ell}+1}^{i\cdot\Gamma_{\ell}}\;\sum_{j'=(j-1)\cdot\Gamma_{\ell}+1}^{j\cdot\Gamma_{\ell}}\lambda\left(A^{(\ell)}_{i'}\right)\cdot\lambda\left(A^{(\ell)}_{j'}\right)&\delta\left(\calA_{\ell}\times \calA_{\ell}: i',j'\right)\\
	&=\lambda\left(A^{(\ell-1)}_{i}\right)\cdot\lambda\left(A^{(\ell-1)}_{j}\right)\delta\left(\calA_{\ell-1}\times \calA_{\ell-1}: i,j\right)\;.\end{align*}
\end{proof}

The next Lemma asserts that the natural density function of the limit partition is measurable. It follows directly from the combination of Lemma~\ref{lem:PartitionConvergence} and Lemma~\ref{lem:deltaMartingale}.
\begin{lemma}\label{lem:Uell Convergence} There exists a measurable function $\delta:([0,1]^2)^2\to[0,1]$ such that $\delta((x,a),(y,a))=\lim\limits_{\ell\to\infty}\delta\left(\calA_\ell\times \calA_\ell: (x,a),(y,b)\right)$ for almost all $(x,a),(y,b) \in ([0,1]^2)^2$.
	\end{lemma}

Finally, we are ready to prove Theorem~\ref{thm:Compactness}.

\begin{proofof}{Theorem~\ref{thm:Compactness}}
We start by giving a high-level overview of the proof. 	Let $\{W_n\}_{n\in \mathbb{N}}$ be a sequence of functions in $\cal W$. We show that there exists a subsequence that has a limit in $\widetilde{\calW}$.

For every $n\in \mathbb{N}$, we use Lemma~\ref{lem:PartitionSeq} to construct a sequence of functions $\{W_{n,\ell}\}_\ell$ such that $\|W_n-W_{n,\ell}\|_\square\le \frac{4}{g(\ell-1)2^\ell}$. Then, for every fixed $\ell\in\mathbb{N}$, we find a subsequence of $\{W_{n,\ell}\}$ such that their corresponding $\alpha^{(n,\ell)}_j $ and $\delta^{(n,\ell)}(i,j)$ converge for all $i,j\in[g(\ell)]$ (as $n\to\infty$).
For every $\ell$, we consider the partition $\calA_\ell$ (which by Definition~\ref{def:A-partition}, is determined by $\{\alpha^{(\ell)}_j\}_j$)  and $\delta^{(\ell)}$. Using $\calA_{\ell}$ and $\delta^{(\ell)}$, we can the define the function $U_\ell$, such that $W_{n,\ell}\to U_{\ell}$ almost everywhere as $n\to\infty$.

Given the sequence of functions $\{U_\ell\}_\ell$, we use Lemma~\ref{lem:Uell Convergence} to show that $\{U_\ell\}_\ell$ converges to some $U$ almost everywhere as $\ell\to\infty$ (where $U$ is defined according the limit density function $\delta$). Finally we show that for any fixed $\eps>0$, there is $n_0(\eps)$ such that for any $n>n_0(\eps)$, $\csdist(W_n,U)\le \eps$.

Fix some $\eps>0$ and $\xi(\eps)>0$ which will be determined later. Consider the sequence $\{U_\ell\}_\ell$ which is defined by the partition $\calA_{\ell}$ and the density function $\delta^{(\ell)}$. By Lemma~\ref{lem:Uell Convergence}, the sequence $\{U_\ell\}_\ell$ converges (as $\ell\to\infty$) almost everywhere to $U$, which is defined by the limit density function $\delta$. Therefore, we can find some $\ell>1/\xi$ such that $\|U_\ell-U\|_1\le\xi$.

Fixing this $\ell$, we show that there is $n_0$ such that $\csdist(W_{n,\ell},U_\ell)\le 2^{-\ell}+ 3\xi$ for all $n>n_0$. We shall do it in two steps by defining an interim function $W'_{n,\ell}$ and using the triangle inequality.

Recall that the function $W_{n,\ell}$ is defined according to the partition $\calP_{n,\ell}$ and the density function $\delta^{(n,\ell
)}$. Let $W'_{n,\ell}$ be the function defined according to the partition $\calA_{\ell}$ and the density function $\delta^{(n,\ell)}$. That is, for every $(x,a)\in A^{(\ell)}_i$ and $(y,b)\in A^{(\ell)}_j$, $W'_{n,\ell}((x,a),(y,b))\eqdef \delta^{(n,\ell)}\left(\calP_{n,\ell}\times\calP_{n,\ell}:i,j\right)$.
By the third item of Lemma~\ref{lem:LimitingPartitionProp}, for every $j\in[g(\ell)]$, $\lambda\left(A^{(\ell)}_j\right)=\lim\limits_{n\to\infty}\lambda\left(P^{(n,\ell)}_j\right)$. Then, we can find $n'_0(\ell)$ such that for all $n>n'_0$,
\begin{align}\label{eq:epsodyssy}
{\max\left(\lambda\left(A^{(\ell)}_j\right),\lambda\left(P^{(n,\ell)}_j\right)\right)-\min\left(\lambda\left(A^{(\ell)}_j\right),\lambda\left(P^{(n,\ell)}_j\right)\right)}\le \frac{\xi }{g(\ell)}\qquad\forall j\in[g(\ell)]\;.
\end{align}
We define a measure preserving map $f$ from $W_{n,\ell}$ to $W'_{n,\ell}$ as follows. For every strip  $I^{(\ell)}_t\in\calI_\ell$, we consider all the sets $\{P^{(n,\ell)}_{j_1}\ldots,P^{(n,\ell)}_{j_t}\}$ in $\calP_{n,\ell}$ such that $\bigcup_{j'=j_1}^{j_t}P^{(n,\ell)}_{j'}=I^{(\ell)}_t$. Similarly, consider all the sets $\{A^{(\ell)}_{j_1}\ldots,A^{(\ell)}_{j_t}\}$ in $\calA_{\ell}$ such that $\bigcup_{j'=j_1}^{j_t}A^{(\ell)}_{j'}=I^{(\ell)}_t$. For every $j'\in\{j_1,\ldots,j_t\}$, we map an arbitrary subset $S^{(n,\ell)}_{j'}\subseteq P^{(n,\ell)}_{j'}$ of measure $\min\left(\lambda\left(A^{(\ell)}_{j'}\right),\lambda\left(P^{(n,\ell)}_{j'}\right)\right)$ to an arbitrary subset (with the same measure) of $A^{(\ell)}_{j'}$. Next, we map $I^{(\ell)}_t\setminus\bigcup_{j'=j_1}^{j_t}S^{(n,\ell)}_{j'}$ to $I^{(\ell)}_t\setminus \bigcup_{j'=j_1}^{j_t}f(S^{(n,\ell)}_{j'})$. Note that by  (\ref{eq:epsodyssy}) and  the fact that $W_{n,\ell}$ and $W'_{n,\ell}$ have the same density function $\delta^{(n,\ell)}$, the functions $W_{n,\ell}$ and $W'_{n,\ell}$ disagree on a set of measure at most $2\xi$. Note that for every $I^{(\ell)}_t\in \calI_\ell$, the function $f$ maps sets from $\calP_{n,\ell}$ that are contained in $I^{(\ell)}_t$ to sets in $\calA_{\ell}$ that are contained in $I^{(\ell)}_t$, and thus, $\shift{f}\le 2^{-\ell}$. Therefore, for $n>n'_0$, we get that $\csdist(W_{n,\ell},W'_{n,\ell})\le 2^{-\ell}+2\xi$, and the first step is complete.

In the second step we bound $\csdist(W'_{n,\ell},U_\ell)$. The two functions $W'_{n,\ell}$ and $U_\ell$ are defined on the same partition $\calA_{\ell}$, however, their values are determined by the density functions $\delta^{(n,\ell)}$ and $\delta^{(\ell)}$ respectively. By the fact that $\delta^{(n,\ell)}$ converges to $\delta^{(\ell)}$ (as $n\to\infty$), we can find $n''_0(\ell)$ such that for all $n>n''_0$,
\begin{align*}
\left|\delta^{(n,\ell)}(i,j)-\delta^{(\ell)}(i,j)\right|\le \frac{\xi}{g(\ell)^2} \qquad \forall i,j\in[g(\ell)]\;.
\end{align*}
Thus, for every $n>n''_0$, it holds that $\csdist(W'_{n,\ell},U_\ell)\le\|W'_{n,\ell}-U_\ell\|_1 \le\xi$. By choosing $n_0=\max(n'_0,n''_0)$ we get that \[ \csdist(W_{n,\ell},U_\ell)\le \csdist(W_{n,\ell},W'_{n,\ell})+\csdist(W'_{n,\ell},U_\ell)\le 2^{-\ell}+3\xi\;. \]

By putting everything together we get that for every $n>n_0$
\begin{align*}
\csdist(W_n,U)&\le\csdist(W_n,W_{n,\ell})+\csdist(W_{n,\ell},U_\ell)+\csdist(U_\ell,U)\\
						&\le \|W_n-W_{n,\ell}\|_\square+ \csdist(W_{n,\ell},U_\ell) +\|U_\ell-U\|_1\\
						&\le O\left(\frac{1}{g(\ell-1)2^\ell}\right)+2^{-\ell}+3\xi+\xi.
\end{align*}
By our choice of $\ell>1/\xi$ we get that \[\csdist(W_n,U)\le6\xi\;. \]
By choosing $\xi=\eps/6$ the theorem follows.
\end{proofof}
\section{Sampling theorem}\label{sec:sampling}
In this section we prove Theorem~\ref{thm:Sampling}.
We start by defining two models of random graphs which are constructed using orderons.
\begin{definition}[ordered $W$-random graphs]\label{def:WRandomGraph}
	Given a function $W\in\calW$ and an integer $n>0$, we generate an edge-weighted ordered random graph $\bH(n,W)$ and an ordered random graph $\bG(n,W)$, both on nodes $[n]$, as follows.
	We generate $\bZ_1,\ldots,\bZ_n$ i.i.d random variables, distributed uniformly in $[0,1]$, and let $\bX_1\le\cdots\le\bX_n$ be the result of sorting $\bZ_1,\ldots,\bZ_n$. In addition, we generate $\bY_1\ldots,\bY_n$  i.i.d random variables, distributed uniformly in $[0,1]$.
	Then, for all $i,j\in[n]$, we set the edge weight of $(i,j)$ in $\bH(n,W)$ to be $W((\bX_i,\bY_i),(\bX_j,\bY_j))$.
	Also, for all $i,j\in[n]$, we set $(i,j)$ to be an edge in $\bG(n,W)$ with probability $W((\bX_i,\bY_i),(\bX_j,\bY_j))$ independently.
\end{definition}
The proof consists of two main parts. The first (and simpler) part states that large enough samples from a naive block orderon are typically close to it in cut-shift distance. The second and main part shows that samples from orderons that are close with respect to cut-shift distance are typically close as well.
We start with the proof of the first part, regarding sampling from naive block orderons.
\begin{lemma}\label{lem:W-and-G(k,W)-close}
	Let $k$ be a positive integer and $W\in\calW$ be a naive $m$-block orderon.
	For any $\eps > 0$, we have
	\[\Pr\left[\csdist(W,W_{\bG(k,W)}) > \frac{2m^{3/2}}{\sqrt{k}} + \frac{\eps}{\sqrt{k}} \right] \leq \exp\left(-C \eps^2 k\right).\]
	for some constant $C>0$.
\end{lemma}
\begin{proof}
	We first show that $\csdist(W,W_{\bH(k,W)})$ is small with high probability and then discuss how it derives a concentration bound for $\csdist(W,W_{\bG(k,W)})$.

	First, we show that the expectation of $\csdist(W,W_{\bH(k,W)})$ is at most $2m^{3/2}/k$.
	Let $\calP=\{P_i\mid i\in[m] \}$ be the block partition of $W$.
	That is, $P_{i}=\left[(i-1)/m,i/m\right] $ for every $i\in[m]$.
	Note that for any $i\in[m]$, $\lambda(P_{i})=1/m$.
	Let $\bZ_1,\ldots,\bZ_k$ be independent uniformly random variables in $[0,1]$ used to construct $\bH(k,W)$.
	For every $i\in[m]$,  let $\bA_{i}$ be the number of samples $\bZ_\ell$ falling into $P_{i}$.
	By the fact that the variables are uniform, for every $i\in [m]$,
	\[\Ex_{\bZ_1,\ldots,\bZ_k }[\bA_{i}]=\frac{k}{m} \qquad \text{and}\qquad \Var[\bA_{i}]=\frac{1}{m}\left(1-\frac{1}{m}\right)k<\frac{k}{m}.\]
	We construct a partition $\calP' =\{P'_{i}\mid i\in[m] \}$ of $[0,1]$ using the values $\bA_{i}$.
	For every $i\in[m]$, we define \[ P'_{i}=\left[\sum_{i'<i}\frac{\bA_{i'}}{k},\sum_{i'\le i}\frac{\bA_{i'}}{k} \right].\]
	We construct an orderon $W_{\bH(k,W)}\in \calW$ so that the value of $W_{\bH(k,W)}$ on $(P'_{i} \times [0,1]) \times (P'_{j} \times [0,1])$ is the same as the value of $W$ on $(P_{i} \times [0,1])\times (P_{j} \times [0,1])$.
	Therefore, $W_{\bH(k,W)}$ agrees with $W$ on a set \[Q=\bigcup_{i,j\in[m]}\left(\left(P_{i} \cap P'_{i}\right) \times [0,1]\right) \times \left(\left(P_{j}\cap P'_{j}\right)\times [0,1]\right).\]
	We note that $\lambda(P'_i) = \bA_i/k$ and
	$$
	\lambda\left(\bigcup_{i\in[m]}\left(P_{i}\cap P'_{i}\right)\right)
	\geq 1 - \sum_{i \in [m]} \left|\sum_{i'\le i}\frac{\bA_{i'}}{k} - \frac{i}{m}\right|
	\geq 1 - \sum_{i \in [m]}\sum_{i'\le i} \left|\frac{\bA_{i'}}{k} - \frac{1}{m}\right|
	\geq 1 - m\sum_{i \in [m]} \left|\frac{\bA_{i}}{k} - \frac{1}{m}\right|.
	$$
	Then,
	\begin{align*}
	& \csdist(W,W_{\bH(k,W)})\le \|W-W_{\bH(k,W)}\|_\square\le 1-\lambda(Q)=1-{\left(1 - m\sum_{i \in [m]} \left|\frac{\bA_i}{k} - \frac{1}{m}\right|\right)}^2\\
	& \le 2 m\sum_{i \in [m]} \left|\frac{\bA_{i}}{k} - \frac{1}{m}\right|
	\le 2m{\left(m\sum_{i \in [m]}{\left(\frac{1}{m}-\frac{\bA_i}{k}\right)}^2\right)}^{1/2}
	= {\left(\frac{4m^3}{k^2}\sum_{i \in [m]}{\left(\frac{k}{m}-\bA_i\right)}^2\right)}^{1/2}.
	\end{align*}
	Therefore, by taking expectation \[\Ex\left[{\csdist(W,W_{\bH(k,W)})}^2\right] \le \frac{4m^3}{k^2} \sum_{i \in [m]}\Var(\bA_i) <\frac{4m^3}{k},\]
	and $\Ex\left[\csdist(W,W_{\bH(k,W)})\right]<2m^{3/2}/\sqrt{k}$ holds by Jensen's inequality.

	By applying Azuma's inequality (see~\cite[Corollary~A.15]{Lov12}), noting that a single change in $\bZ_\ell$ changes the value of $\csdist(W,W_{\bH(k,W)})$ by at most $O(1/k)$, we have for any $\eps > 0$,
	\begin{align}
	\Pr\left[\csdist(W,W_{\bH(k,W)}) > \frac{2m^{3/2}}{\sqrt{k}} + \frac{\eps}{k} \right] \leq \exp\left(-C' \eps^2 k\right).
	\label{eq:W-and-G(k,W)-close-1}
	\end{align}
	for some constant $C'>0$.

	For an edge-weighted ordered graph $H$ on nodes $[k]$, we define $\bG(H)$ to be the ordered graph obtained by, for all $i,j\in[n]$, setting $(i,j)$ to be an edge in $\bG(H)$ with probability being the weight of $(i,j)$ in $H$ independently.
	By~\cite[Lemma~10.11]{Lov12}, we have for any edge-weighted ordered graph $H$ and $\eps > 0$
	\begin{align}
	\Pr\left[\csdist(W_{\bG(H)},W_H) > \frac{\eps}{\sqrt{k}}\right] \leq
	\Pr\left[\|W_{\bG(H)}-W_{H}\|_\square > \frac{\eps}{\sqrt{k}} \right] \leq \exp\left(-\eps^2 k/100\right).
	\label{eq:W-and-G(k,W)-close-2}
	\end{align}


	The desired concentration bound is obtained by~\eqref{eq:W-and-G(k,W)-close-1},~\eqref{eq:W-and-G(k,W)-close-2} and a union bound.
\end{proof}



Before proceeding to the next lemma, we first recall the notion of a coupling of distributions.
\begin{definition}[couplings]
	Let $\calD_1$ and $\calD_2$ be distributions over domains $\Omega_1$ and $\Omega_2$, respectively.
	Then, a coupling of $\calD_1$ and $\calD_2$ is a distribution $\calD$ over $\Omega_1 \times \Omega_2$ such that the marginal distributions of $\calD$ on $\Omega_1$ and $\Omega_2$ are the same distributions as $\calD_1$ and $\calD_2$, respectively.
\end{definition}

Let $W,W' \in \calW$ be orderons.
The following lemma says that the random ordered graphs $\bG(k,W)$ and $\bG(k,W')$ can be coupled (note that $\bG(k,W)$ and $\bG(k,W')$ can be regarded as distributions over ordered graphs of $k$ vertices) so that, when $\csdist(W,W')$ is small,  $\csdist(W_{\bG(k,W)},W_{\bG(k,W')})$ is also small with high probability.

Roughly speaking, the main idea is as follows. The ``cut part'' follows from results in the unordered setting, specifically Corollary 10.12 in~\cite{Lov12}. For the ``shift part'', we show that there is a coupling between samples $G \sim \bG(k, W)$ and samples $G' \sim \bG_{\bk}(k, W^f)$, so that for most of the pairs $(G, G')$, one can turn $G$ into $G'$ by only making local reordering of vertices, where no vertex is moved more than $O(k \cdot \shift{f})$ steps away from its original location. This follows from the fact that if one takes a sample from an orderon, then the location of the $i$-th order statistic---the vertex that is the $i$-th smallest in the first coordinate---is typically close to its mean.

\begin{lemma}\label{lem:W-W'-close-in-csdist}
	Let $W,W' \in \calW$ be orderons and $k$ be a positive integer.
	Then, the random ordered graphs $\bG(k,W)$ and $\bG(k,W')$ can be coupled so that
	$$
	\csdist(W_{\bG(k,W)}, W_{\bG(k,W')}) \leq 9\csdist(W, W') + \frac{10}{k^{1/4}}
	$$
	holds with probability at least $1 - k{(e/4)}^{4k\csdist(W,W')} - 5e^{-\sqrt{k}/10}$.
\end{lemma}
\begin{proof}
	For any $\delta > 0$, there exists $f \in \calF$ such that $\shift{f}+\|W^f-W'\|_\square < \csdist(W,W') + \delta$.
	Here, we choose $\delta = \csdist(W,W')$ and fix $f$ for this choice.
	To define the desired coupling between $\bG \eqdef \bG(k,W)$ and $\bG' \eqdef \bG(k,W')$, we first define a coupling between $\bG$ and $\bG^f \eqdef \bG(k,W^f)$ so that $\csdist(W_{\bG},W_{\bG^f})$ is small with high probability, and then define a coupling between $\bG^f$ and $\bG'$ so that $\csdist(W_{\bG^f},W_{\bG'})$ is small with high probability.
	We obtain the desired bound by chaining the couplings and a union bound.

	Recall that, in the construction of $\bG(k,W)$, we used two sequences of independent random variables $\bZ = {(\bZ_i)}_{i \in [k]}$ and $\bY = {(\bY_i)}_{i \in [k]}$.
	To look at the construction more in detail, it is convenient to introduce another sequence of independent random variables $\bR = {(\bR_{ij})}_{i,j \in [k],i<j}$, where each $\bR_{ij}$ is uniform over $[0,1]$.
	After defining $\bX={(\bX_i)}_{i \in [k]}$ as in Definition~\ref{def:WRandomGraph}, we obtain $\bG(k,W)$ by setting $(i,j)$ to be an edge if $W((\bX_i,\bY_i),(\bX_j,\bY_j)) \geq \bR_{ij}$ for each $i,j \in [k]$ with $i < j$.
	To make the dependence on these random variables more explicit, we write $G(\bZ,\bY,\bR; W)$ to denote the ordered graph obtained from $W$ by using $\bZ$, $\bY$, and $\bR$.

	Let $\left(\bZ={(\bZ_i)}_{i \in [k]},\bY={(\bY_i)}_{i \in [k]},\bR={(\bR_{ij})}_{i,j\in [k]:i<j}\right)$ be uniform over ${[0,1]}^{k+k+\binom{k}{2}}$.
	Then, we define $\left(\bZ'={(\bZ'_i)}_{i \in [k]},\bY'={(\bY'_i)}_{i \in [k]},\bR'={(\bR'_{ij})}_{i,j\in [k],i<j}\right)$ so that $(\bZ'_i,\bY'_i) = f^{-1}(\bZ_i,\bY_i)$ for every $i \in [k]$ and $\bR'_{ij} = \bR_{ij}$ for every $i,j \in [k]$ with $i<j$.
	Note that the marginal $(\bZ',\bY',\bR')$ is uniform over ${[0,1]}^{k+k+\binom{k}{2}}$, and hence the distribution of $G(\bZ',\bY',\bR';W^f)$ is exactly same as that of $\bG(k,W^f)$.
	Now, we couple $G(\bZ,\bY,\bR; W)$ with $G(\bZ',\bY',\bR';W^f)$. 

	We can naturally define a measure preserving function $g \in \calF$ from $W_{\bG}$ to $W_{\bG^f}$ with $W_{\bG}^g = W_{\bG^f}$ as follows ($\bG$ and $\bG^f$ are coupled as in the last paragraph).
	Let $\pi\colon[k] \to [k]$ be a permutation of $[k]$ such that $\bZ_{\pi^{-1}(1)} \leq \bZ_{\pi^{-1}(2)} \leq \cdots \leq \bZ_{\pi^{-1}(k)}$.
	Then, $\pi(i)$ is the position of $\bZ_i$ in this sorted sequence.
	Similarly we define a permutation $\pi'\colon[k] \to [k]$ using $\bZ'$.
	Then, we arbitrarily choose $g$ so that the part corresponding to $\pi(i)$ is mapped to the part corresponding to $\pi'(i)$, that is, $\{g(v) \mid v \in [(\pi(i)-1)/k,\pi(i)/k] \times [0,1] \} = [(\pi'(i)-1)/k,\pi'(i)/k] \times [0,1]$ for every $i \in [k]$.

	We now show a concentration bound for $\shift{g}$.
	For each $i \in [k]$, we consider a segment $\bA_i = [\bZ_i-2\csdist(W,W'),\bZ_i+2\csdist(W,W')]$ in a circular domain $[0,1]$, where we identify $-x$ with $1 - x$.
	Letting $\bM$ be the maximum number of overlaps of the segments at a point $x$ over $x$ in the circular domain, we can upper bound $\shift{g}$ by $(\bM-1)/k$ because $\shift{f} \leq 2\csdist(W,W')$ and the overlap of two segments may cause a shift of $1/k$ in $g$ to map the vertex corresponding to one segment to the vertex corresponding to the other.
	Let $\mu = 4k\csdist(W,W')$ be the average overlap at a point.
	As the segments $\bA_1,\ldots,\bA_k$ are independently and uniformly distributed, by~\cite[Theorem 3.1]{Sanders:2003gw} (to apply this theorem, we considered the circular domain instead of the interval $[0,1]$), we have
	\begin{align}
	\Pr[\bM \geq 2\mu  + 1] \leq k {\left(\frac{e}{4}\right)}^{\mu}. \label{eq:W-W'-close-in-csdist-1}
	\end{align}
	Hence, we have
	$$
	\shift{g} \leq 8 \csdist(W,W')
	$$
	with probability at least $1-k{(e/4)}^{4k \csdist(W,W')}$.



	Next, we couple $\bG^f$ with $\bG'$ by coupling $G(\bZ',\bY',\bR';W^f)$ with $G(\bZ',\bY',\bR';W')$.
	By~\cite[Corollary 10.12]{Lov12}, we have 
	$$
	\left\|W_{\bG^f}-W_{\bG'}\right\|_\square
	= \left\|W_{G(\bZ',\bY',\bR',W^f)}-W_{G(\bZ',\bY',\bR',W')}\right\|_\square
	\leq \left\|W^f-W'\right\|_\square + \frac{10}{k^{1/4}}
	$$
	with probability at least $1-5e^{-\sqrt{k}/10}$.
	Note that we can apply the corollary because the sorting process according to $\bZ'$ during the constructions of $\bG^f$ and $\bG'$ does not affect the cut norm.

	Now, we combine by chaining the couplings $(\bG,\bG^f)$ and $(\bG^f,\bG')$.
	By a union bound, we have
	\begin{align*}
	& \csdist(W_{\bG},W_{\bG'}) \leq
	\shift{g} +
	\left\|W_{\bG^f}-W_{\bG'}\right\|_\square \\
	& \leq 8 \csdist(W,W')+ \left\|W^f-W'\right\|_\square + \frac{10}{k^{1/4}}
	\leq 9\csdist(W,W') + \frac{10}{k^{1/4}}
	\end{align*}
	with probability at least $1 - k{(e/4)}^{4k\csdist(W,W')} - 5e^{-\sqrt{k}/10}$.
\end{proof}

The proof of Theorem~\ref{thm:Sampling} now follows easily from the above two lemmas and Theorem~\ref{thm:ApproxNaive-StepOrderon}. Indeed, from Theorem~\ref{thm:ApproxNaive-StepOrderon} we know that any orderon $W$ has an arbitrarily close naive block orderon $W'$, and from Lemma~\ref{lem:W-W'-close-in-csdist} we conclude that the cut-shift distance between samples from $W$ and samples from $W'$ is typically not much larger than $d_\triangle(W, W')$. Finally, Lemma~\ref{lem:W-and-G(k,W)-close} implies that $W'$ is typically close in $d_{\triangle}$ to large samples taken from it.

\begin{proofof}{Theorem~\ref{thm:Sampling}}
	Let $W'$ be the naive $m$-block orderon obtained by applying Theorem~\ref{thm:ApproxNaive-StepOrderon} on $W$.
	Let $(\bG,\bG')$ be the coupling obtained by applying Lemma~\ref{lem:W-W'-close-in-csdist} on $W$ and $W'$.
	By the triangle inequality, we have
	\begin{align}
	\csdist(W,W_\bG)
	\leq
	\csdist(W,W') + \csdist(W',W_{\bG'}) + \csdist(W_{\bG'},W_{\bG}).
	\label{eq:Sapmling-1}
	\end{align}
	By Lemmas~\ref{lem:W-and-G(k,W)-close} and~\ref{lem:W-W'-close-in-csdist} and the union bound, we have for any $\eps' > 0$
	\begin{align*}
	\eqref{eq:Sapmling-1}
	& \leq \csdist(W,W') + \frac{2m^{3/2}}{k} + \frac{\eps'}{\sqrt{k}} + 9\csdist(W',W) + \frac{10}{k^{1/4}}  \\
	& = 10\csdist(W,W') + \frac{2m^{3/2}}{k} + \frac{\eps'}{\sqrt{k}} + \frac{10}{k^{1/4}} \\
	& = 10\eps + \frac{c^{3/2} 2^{2062/\eps^3}}{\eps^6 k} + \frac{\eps'}{\sqrt{k}} +  \frac{10}{k^{1/4}},
	\end{align*}
	with probability at least $1-k(e/4)^{4k\eps}- 5e^{-\sqrt{k}/10}-\exp(C{(\eps')}^2 k)$.
	By setting $\eps = {\Theta(\log\log k/\log k)}^{1/3}$ and $\eps' = \Theta(\sqrt{k}/\log k)$, we have the desired bound.
\end{proofof}

\section{Subgraph statistics}\label{sec:statistics}

For $k\in \mathbb{N}$, let $\Omega=\left([0,1]^2\right)^k$. We define $\Omega_o\subset\Omega$ such that $v\in \Omega_o$ if and only if  $v$ is ordered according to the entries of the first coordinate in the tuple (in such case, we write $v_1\le\cdots\le v_k$). Namely, $v=\left((x_1,y_1),\ldots,(x_k,y_k)\right)\in \Omega_o$ if and only if for every $i\le j$, $x_i\le x_j$. In addition, given a set of pairs $E\subseteq\binom{[k]}{2}$ we let $E_{<}=\{(i,j)\in E\mid i<j\}$.
Let us restate the definition of homomorphism density of an ordered graph in an orderon from Subsection~\ref{subsec:main_ingredients} in a slightly different but equivalent form.
\begin{definition}[homomorphism density] Let $F=([k],E)$ be a simple ordered graph and let $W\in\calW$ be an orderon.  We define the (induced) \emph{homomorphism density} of $F$ in $W$ as
	\[t(F,W)\eqdef k!\cdot \intop_{v\in\Omega_o}\prod_{(i,j)\in E_{<}} W(v_i,v_j)\cdot \prod_{(i,j)\in\overline{E}_{<}} \left(1-W(v_i,v_j)\right)dv\;,
	\]
or equivalently,
\[ t(F,W)\eqdef k!\cdot\intop_{v\in\Omega}\left(\prod_{{(i,j)\in E_{<}}} W(v_i,v_j)\cdot \prod_{{(i,j)\in \overline{E}_{<}}} \left(1-W(v_i,v_j)\right)\cdot \prod_{{i<j}}\indi_{v_i\le v_j}\right) dv\;. \]
\end{definition}

Recall the definition of $t(F, G)$ where $G$ is an ordered graph, presented in Subsection~\ref{subsec:main_ingredients}. Clearly, $t(F, G) = t(F, W_G)$ always holds. Our first main result of the section proves the ``only if'' direction of Theorem~\ref{thm:convergence_equiv}, showing that if a sequence of orderons $W_n$ is Cauchy in $d_{\triangle}$, then it is convergent in terms of subgraph frequencies.

%

\begin{lemma}\label{lem:CountingLemma} Let $W,U\in\calW$. Then, for every simple ordered graph $F=([k],E)$\[\left|{t(F,W)}-{t(F,U)}\right|\le 6 k!\binom{k}{2}\cdot\sqrt{\csdist(W,U)}\;. \]
\end{lemma}

In order to prove the above, we introduce the following two lemmas. Lemma~\ref{lem:smallcutnorm-simdensity} considers pairs of orderons that are close in cut-norm, whereas Lemma~\ref{lem:smallshiftlemma} describes the effect of shifts.

\begin{lemma}\label{lem:smallcutnorm-simdensity}
	For any $W,U\in \calW$ and every ordered graph $F=([k],E)$, \[\left|t(F,W)-t(F,U)\right|\le 2k! \binom{k}{2}\sqrt{\|W-U\|_\square}\;.\]
\end{lemma}
\begin{proof}
	Fix some arbitrary ordering on $\binom{[k]}{2}_{<}$. For every $(i,j)\in\binom{[k]}{2}_{<}$ we define the function
	\[
	\gamma_W(v_i,v_j)=\begin{cases}
	W(v_i,v_j)\cdot \indi_{v_i\le v_j}, &\text{if }  (i,j)\in E(F) \\
	(1-W(v_i,v_j))\cdot\indi_{v_i\le v_j}, &\text{if } (i,j)\notin E(F)
	\end{cases}\;,
	\]
	and define $\gamma_U$ similarly.
	\begin{align*}
	t(F,W)-t(F,U)&= k!\intop_{v\in\Omega}\left(\prod_{(i,j)\in\binom{[k]}{2}_<}\gamma_W(v_i,v_j)-\prod_{(i,j)\in\binom{[k]}{2}_<}\gamma_U(v_i,v_j)\right)dv
	\end{align*}
By identifying each $e_r\in\left[\binom{[k]}{2}_<\right]$ with $(i_r,j_r)$, the integrand can be written as
 \begin{align*}
 \sum_{s=1}^{\binom{k}{2}}\left(\prod_{r<s}\gamma_W(v_{i_r},v_{j_r})\right)&\left(\prod_{r>s}\gamma_U(v_{i_r},v_{j_r})\right)\cdot\Big(\gamma_W(v_{i_s},v_{j_s})-\gamma_U(v_{i_s},v_{j_s})\Big)\;.
 \end{align*}
 To estimate the integral of a given term, we fix all variables except $v_{i_s}$ and $v_{j_s}$. Then, the integral is of the form
 \[\int_{v_{i_s},v_{j_s}}  g(v_{i_s}) h(v_{j_s}) \Big(\gamma_W(v_{i_s},v_{j_s})-\gamma_U(v_{i_s},v_{j_s})\Big)dv_{i_s}dv_{j_s},\]
 where $g,h\colon [0,1]^2\to[0,1]$ are some functions. By applying Lemma~\ref{lem:smoothning}, it suffices to provide an upper bound on
 \[\sup_{S,T\subseteq[0,1]^2}\left|\;\intop_{S\times T}\Big(\gamma_W(v_{i_s},v_{j_s})-\gamma_U(v_{i_s},v_{j_s})\Big)dv_{i_s}dv_{j_s}\;\right|.\]
 By using Lemma~\ref{lem:OrederedtoNoOrder}, we get that
 \begin{align*}
&\sup_{S,T\subseteq[0,1]^2}\left|\;\intop_{S\times T}\Big(\gamma_W(v_{i_s},v_{j_s})-\gamma_U(v_{i_s},v_{j_s})\Big)dv_{i_s}dv_{j_s}\;\right|\nonumber \\
&=\sup_{S,T\subseteq[0,1]^2}\left|\;\intop_{S\times T}\Big(W(v_{i_s},v_{j_s})-U(v_{i_s},v_{j_s})\Big)\indi_{v_{i_s}\le v_{j_s}}dv_{i_s}dv_{j_s}\;\right| =\|W-U\|_{\square'}\le 2\sqrt{\|W-U\|_\square}.
 \end{align*}
 By summing up over all $\binom{k}{2}$ pairs of vertices, the lemma follows.
\end{proof}
\begin{lemma}\label{lem:smallshiftlemma}Let $U\in \calW$ and let $\phi\colon[0,1]^2\to[0,1]^2$ be a measure preserving function. Then, for every ordered graph $F=([k],E)$ \[ \left|t(F,U)-t(F,U^\phi)\right|\le 4k!\binom{k}{2}\cdot\shift{\phi}\;. \]

\end{lemma}

\begin{proof} The proof is similar to Lemma~\ref{lem:smallcutnorm-simdensity}. However, we shall slightly change notation. Let $\gamma_W$ be defined as follows.
	\[
	\gamma_W(v_i,v_j)=\begin{cases}
	W(v_i,v_j), &\text{if }   (i,j)\in E(F) \\
	1-W(v_i,v_j), &\text{if } (i,j)\notin E(F)
	\end{cases}\;.
	\] Then, by changing the integration variables in $t(F,U^\phi)$ from $v_i,v_j$ to $\phi^{-1}(v_i),\phi^{-1}(v_j)$, we have
	\begin{align*}
	t(F,U)-t(F,U^\phi)&= k!\int_{v\in\Omega}\left(\prod_{(i,j)\in \binom{[k]}{2}_<}\gamma_U(v_i,v_j)\indi_{v_i\le v_j}-\prod_{(i,j)\in\binom{[k]}{2}_<}\gamma_{U^\phi}(\phi^{-1}(v_i),\phi^{-1}(v_j))\indi_{\phi^{-1}(v_i)\le\phi^{-1} (v_j)}\right)dv\\
	&=k!\int_{v\in\Omega}\left(\prod_{(i,j)\in \binom{[k]}{2}_<}\gamma_U(v_i,v_j)\indi_{v_i\le v_j}-\prod_{(i,j)\in\binom{[k]}{2}_<}\gamma_{U}(v_i,v_j)\indi_{\phi^{-1}(v_i)\le\phi^{-1} (v_j)}\right)dv\\
	&=k!\int_{v\in\Omega}\prod_{(i,j)\in \binom{[k]}{2}_<}\gamma_U(v_i,v_j)\left(\prod_{(i,j)\in \binom{[k]}{2}_<}\indi_{v_i\le v_j}-\prod_{(i,j)\in\binom{[k]}{2}_<}\indi_{\phi^{-1}(v_i)\le\phi^{-1} (v_j)}\right)dv.
	\end{align*}
Hence,
	\begin{align*}
|t(F,U)-t(F,U^\phi)|&\le k!\left|\int_{v\in\Omega}\left(\prod_{(i,j)\in \binom{[k]}{2}_<}\indi_{v_i\le v_j}-\prod_{(i,j)\in\binom{[k]}{2}_<}\indi_{\phi^{-1}(v_i)\le\phi^{-1} (v_j)}\right)dv\right|\\
	&=k!\left|\int_{v\in\Omega}\sum_{s=1}^{\binom{k}{2}}\left(\prod_{r<s}\indi_{v_{i_r}\le v_{j_r}}\right)\left(\prod_{r>s}\indi_{\phi^{-1}(v_{i_r})\le\phi^{-1} (v_{j_r})}\right)\Big(\indi_{v_{i_s}\le v_{j_s}}-\indi_{\phi^{-1}(v_{i_s})\le\phi^{-1} (v_{j_s})}\Big)dv\right|\;.
	\end{align*}
	Similarly to Lemma~\ref{lem:smallcutnorm-simdensity}, we fix all the variables except $v_{i_s}$ and $v_{j_s}$. Then, by using Lemma~\ref{lem:smoothning}, it suffices to estimate
	\[\left|\int_{v_{i_s},v_{j_s}}\Big(\indi_{v_{i_s}\le v_{j_s}}-\indi_{\phi^{-1}(v_{i_s})\le\phi^{-1} (v_{j_s})}\Big)dv_{i_s}dv_{j_s}\right|\le \int_{v_{i_s},v_{j_s}}\left|\indi_{v_{i_s}\le v_{j_s}}-\indi_{\phi^{-1}(v_{i_s})\le\phi^{-1} (v_{j_s})}\right|dv_{i_s}dv_{j_s} .\]
	Note that whenever the intersection between $[\pi_1(v_{i_s}),\pi_1(\phi^{-1}(v_{i_s}))]$ and $[\pi_1(v_{j_s}),\pi_1(\phi^{-1} (v_{j_s}))]$ is empty, the difference between the indicators is zero. Therefore,
	\begin{align}
	&\left|\int_{v_{i_s},v_{j_s}}\Big(\indi_{v_{i_s}\le v_{j_s}}-\indi_{\phi^{-1}(v_{i_s})\le\phi^{-1} (v_{j_s})}\Big)dv_{i_s}dv_{j_s}\right|\nonumber\\
	&\le\int_{v_{i_s},v_{j_s}} \indi_{[\pi_1(v_{i_s}),\pi_1(\phi^{-1}(v_{i_s}))]\cap [\pi_1(v_{j_s}),\pi_1(\phi^{-1} (v_{j_s}))]\neq \emptyset}\;dv_{i_s}dv_{j_s}\le 4\shift{\phi}.\label{eq:shift}
	\end{align}
 By summing up over all $\binom{k}{2}$ pairs of vertices, the lemma follows.\end{proof}

Using the above two lemmas, it is straightforward to prove Lemma~\ref{lem:CountingLemma}.

\begin{proofof}{Lemma~\ref{lem:CountingLemma}}
For any $\gamma>0$ let $\phi\colon[0,1]^2\to[0,1]^2$ be the measure preserving function such that $\shift{\phi}+\|W-U^\phi\|_\square\le \csdist(W,U)+\gamma$. For this specific $\phi$ we have that \[\|W-U^{\phi}\|_\square=\sup_{S,T\subseteq[0,1]^2}\left|\int_{S\times T}W(v_1,v_2)-U(\phi(v_1),\phi(v_2))dv_1dv_2\right|\le\csdist(W,U)+\gamma,\] and $\shift{\phi}\le \csdist(W,U)+\gamma$.
Then, by assuming $\csdist(W,U)>0$ (note that the case $d_{\triangle}(W, U) = 0$ is covered by considering what happens when $\csdist(W,U) \to 0$), and  using the triangle inequality combined with Lemma~\ref{lem:smallcutnorm-simdensity} and Lemma~\ref{lem:smallshiftlemma}, we get
\begin{align*}
|t(F,W)-t(F,U)|&\le|t(F,W)-t(F,U^\phi)|+|t(F,U^\phi)-t(F,U)|\\
&\le 2k!\binom{k}{2}\sqrt{\csdist(W,U)}+k!\binom{k}{2}\frac{\gamma}{\sqrt{\csdist(W,U)}}+4k!\binom{k}{2}(\csdist(W,U)+\gamma)\\
&\le 6k!\binom{k}{2}\sqrt{\csdist(W,U)}++k!\binom{k}{2}\frac{\gamma}{\sqrt{\csdist(W,U)}}+4k!\binom{k}{2}\sqrt{\gamma}.
\end{align*}
As the choice of $\gamma$ is arbitrary, the lemma follows.\end{proofof}

Next we prove a converse statement, showing that if all frequencies of $k$-vertex graphs in a pair of orderons $W$ and $U$ are very similar, than $d_{\triangle}(W, U)$ is small. This establishes the ``if'' component of Theorem~\ref{thm:convergence_equiv}.

\begin{lemma}\label{lem:HardStats} Let $k\in\mathbb{N}$ and $W,U\in\calW$. Assume that for every ordered graph $F$ on $k$ vertices,
	\[|t(F,W)-t(F,U)|\le 2^{-k^2}.\]
	Then, $\csdist(W,U)\le  2C\left(\frac{\log\log k}{{\log k}}\right)^{1/3}$ for some constant $C>0$.
\end{lemma}

\begin{proof} We start by showing that if for some $k\ge 2$, the total variation distance between the distribution $\bG(k,W)$ and $\bG (k,U)$ is small then they are close in CS-distance.

	Assume that For $U,W\in\calW$ and some $k\ge2$ it holds that
	\[d_{TV}(\bG(k,W),\bG(k,U))<1-\exp\left(-\frac{k}{2\log k}\right).\]
	This assumption implies that there exists a joint distribution $(\bG(k,W),\bG(k,U))$ so that $\bG(k,W)=\bG(k,U)$ with probability larger than $\exp\left(-\frac{k}{2\log k}\right)$. By Theorem~\ref{thm:Sampling}, with probability at least $1-C\exp(-\sqrt{k}/C)$, we have that $\csdist(W,\bG(k,W))\le C\left(\frac{\log\log k}{{\log k}}\right)^{1/3}$ for some constant $C>0$. Let ${\boldsymbol\calE_1},\boldsymbol\calE_2,\boldsymbol\calE_3$ denote the events that $\bG(k,W)=\bG(k,U)$, $\csdist(W,\bG(k,W))\le C\left(\frac{\log\log k}{{\log k}}\right)^{1/3}$ and $\csdist(U,\bG(k,U))\le C\left(\frac{\log\log k}{{\log k}}\right)^{1/3}$, respectively.

	Therefore, by using a union bound, $\Prx[\overline{\boldsymbol{\calE_1}}\vee\overline{\boldsymbol{\calE_2}}\vee \overline{\boldsymbol{\calE_3}}]\le 2C\exp(-\sqrt{k}/C)+\exp(-k/2\log k)<1$.
	Hence, there is a positive probability for all the three events to occur, implying that
	\[\csdist(W,U)\le \csdist(W,\bG(k,W))+\csdist(U,\bG(k,U))+\csdist(\bG(k,U),\bG (k,W))\le  2C\left(\frac{\log\log k}{{\log k}}\right)^{1/3}.\]

	The lemma follows by noting that,
	\begin{align*}
	|t(F,W)-t(F,U)|&=\left|\Prx_{\bG\sim\bG(k,W)}[\bG=F]-\Prx_{\bG\sim\bG(k,U)}[\bG=F]\right|\le 2^{-k^2},
	\end{align*}
	and hence,
	\begin{align*}
	d_{TV}(\bG(k,W) ,\bG (k,U))&=\sum_F\left|\Prx_{\bG\sim\bG(k,W)}[\bG=F]-\Prx_{\bG\sim\bG(k,U)}[\bG=F]\right|\le 2^{\binom{k}{2}}\cdot 2^{-k^2}\le 2^{-k}\\
	&<1-\exp\left(-\frac{k}{2\log k}\right).
	\end{align*}
\end{proof}

%

\section{The furthest ordered graph from a hereditary property}\label{sec:furthest}

In this section we prove Theorem~\ref{thm:ordered_Alon_Stav}. The proof roughly follows along the same lines as the proof of Lov\'asz and Szegedy for the
unordered setting~\cite{LS10}. However, before proceeding, let us show why an exact analogue of Theorem~\ref{thm:Alon-Stav-unordered} cannot hold for the ordered setting.

\subsection{The random graph $\bG(n,p)$ is not always the furthest}
\label{subsec:furthest_tehnical}
Recall the hereditary property $\calH$ defined in Subsection~\ref{subsec:furthest_intro}. For convenience, let us describe $\calH$ again. An ordered graph $G$ satisfies $\calH$ if and only if there do not exist vertices $u_1 < u_2 \leq u_3 < u_4$ in $G$ where $u_1 u_2$ is a non-edge and $u_3 u_4$ is an edge.
Here we prove that (a direct analogue of) Theorem~\ref{thm:Alon-Stav-unordered} does not hold for $\calH$, that is, for any $p$, a typical graph $G \sim G(n,p)$ is \emph{not} asymptotically the furthest graph from $\calH$. We contrast this by describing, for any $n \in \mathbb{N}$, a graph $G$ on $n$ vertices that \emph{is} the furthest from $\calH$ up to an $o(1)$ term.

The following lemma characterizes the property $\calH$ in a way that will make it fairly straightforward, given an ordered graph $G$, to estimate the distance $d_1(G, \calH)$.
\begin{lemma}[Thresholding lemma]\label{lem:threshold_hered_property}
	An ordered graph $G$ on vertices $v_1 < v_2 < \cdots < v_n$ satisfies $\calH$ if and only if there exists a ``threshold'' $i \in [n]$ for which the following two conditions hold.
	\begin{itemize}
		\item For any $1 \leq j < j' \leq i$, the vertices $v_j$ and $v_{j'}$ are \emph{connected} in $G$.
		\item For any $i+1 \leq j < j' \leq n$, the vertices $v_j$ and $v_{j'}$ are \emph{not connected} in $G$.
	\end{itemize}
\end{lemma}
\begin{proof}
	Suppose first that $G \in \calH$, and pick the maximal $i \in [n-1]$ for which all pairs of the form $v_jv_{j'}$ where $1 \leq j < j' \leq i$ are connected in $G$. The first condition holds trivially, and it only remains to prove the second one. Suppose on the contrary that it does not hold, which means that $i < n-1$ and that $v_j$ and $v_{j'}$ are connected for some $i+1 \leq j < j' \leq n$. By the maximality of $i$, there exists some $j'' \leq i$ for which $v_{j''}v_{i+1}$ is \emph{not} an edge. Since $i+1 \leq j$, the four vertices $v_{j''} < v_{i+1} \leq v_j < v_{j'}$ induce the ordered pattern forbidden by $\calH$, a contradiction.
	
	Conversely, if there exists $i$ for which the two conditions hold, then for any tuple of vertices $v_{j_1}, v_{j_2}, v_{j_3}, v_{j_4}$ where $j_1 < j_2 \leq j_3 < j_4$, at least one of the following holds. Either we have $j_2 \leq i$, meaning that $v_{j_1} v_{j_2}$ is an edge, or $j_3 \geq i+1$, meaning that $v_{j_3} v_{j_4}$ is a non-edge. In both cases, the tuple does not violate the condition in the definition of $\calH$.
\end{proof}

Using Lemma~\ref{lem:threshold_hered_property}, we can estimate the typical distance of a graph $G \sim \bG(n,p)$ from $\calH$.
\begin{lemma}
	For any $n \in \mathbb{N}$ and $p \in [0,1]$ (possibly depending on $n$), a graph $G \sim \bG(n,p)$ satisfies with high probability that $d_1(G, \calH) = p(1-p) + o(1) \leq \frac{1}{4} + o(1)$.
\end{lemma}
\begin{proof}
	Standard Chernoff-type concentration bounds show that w.h.p.~the following holds for all $i \in [n]$: the number of non-edges among the first $i$ vertices is $(1-p)i^2/2 + o(n^2)$, and the number of edges among the last $n-i$ vertices is $p(n-i)^2/2 + o(n^2)$. Thus, w.h.p.~we have
	\begin{equation}
	\label{eq:dist_of_random_graph}
	d_1(G, \calH) = \frac{1}{\binom{n}{2}} \min_{i \in [n]} \left( (1-p)\frac{i^2}{2} + p\frac{(n-i)^2}{2} + o(n^2)  \right) = p(1-p) + o(1)
	\end{equation}
	where the minimum of the sum (up to the $o(n^2)$ term) is attained for $i \in \{\lfloor np \rfloor, \lceil np \rceil\}$. The $p(1-p)$ term is maximized at $p = 1/2$, where it equals $1/4$. 
\end{proof}

On the other hand, there exist large graphs whose distance from $\calH$ is much larger than $1/4$.
\begin{lemma}
	Let $n \in \mathbb{N}$, and consider the graph $G$ on the vertex set $[n]$ with the standard ordering, where $u$ and $v$ are connected if and only if $u+v \geq n$. The graph $G$ satisfies $d_1(G, \calH) \geq 1/2 - o(1)$.
\end{lemma}
\begin{proof}
	Fix any $i \in [n]$. Without loss of generality, we may assume that $i \leq n/2$ (the case $i > n/2$ is symmetric).
	On one hand, the number of non-edges between pairs of vertices $j < j' \in [i]$ is $\binom{i}{2}$. On the other hand, a straightforward calculation shows that the number of edges between pairs of vertices $j,j'$ where $i+1 \leq j < j'$ is $\frac{1}{2}\binom{n}{2} - \binom{i}{2} - o(n^2)$. Thus, the number of edges that one needs to add or remove if the threshold of Lemma~\ref{lem:threshold_hered_property} is set to $i$ is $\frac{1}{2}\binom{n}{2} - o(n^2)$ (regardless of the value of $i$), implying that $d_1(G, \calH) \geq 1/2 - o(1)$.
\end{proof}

Conversely, it is trivial that $\overline{d_\calH} \leq 1/2$, as any graph $G$ can be turned into either a complete or empty graph (which both satisfy $\calH$) by adding or removing at most $\frac{1}{2}\binom{n}{2}$ edges. Combined with the last lemma, we conclude that $\overline{d_{\calH}} = 1/2$.

\subsection{Proof of Theorem~\ref{thm:ordered_Alon_Stav}}
\label{subsec:furthest-proof}
We continue to the proof of Theorem~\ref{thm:ordered_Alon_Stav}. Along this section, the ordered graphs that we consider are generally simple, and the notion of a hereditary property refers to a property of simple ordered graphs.
To begin, we first establish several basic properties of hereditary properties in the ordered setting, starting with a discussion on their \emph{closure} in the space of orderons.

Let $\calH$ be a hereditary simple ordered graph property ($\calH$ will be fixed throughout the section). Recall that a property is \emph{hereditary} if it satisfies the following: if $G\in\calH$ then every induced subgraph of $G$ (vertex repetitions are not allowed) is also in $\calH$.

In general, we define the \emph{closure} $\barP$ of an ordered graph property $\calP$ as the set of all orderons $W$ for which there exists a sequence of graphs $G_n \in \calP$ (with $|V(G_n)| \to \infty$) that converges to $\calP$ in $d_{\triangle}$. Note that $\barP$ is a closed set in $\calW$ with respect to $d_{\triangle}$ (also note that it is generally \emph{not} true that $W_G \in \barP$ for any $G \in \calP$). The following characterization for the closure of a hereditary property will be useful in multiple occasions along this section as well as in Section \ref{sec:hereditary}.

\begin{lemma}\label{lem:conditions_closure}
	Let $\calH$ be a hereditary property of ordered graphs. The following conditions are equivalent for an orderon $W$.
	\begin{itemize}
		\item $W \in \barH$.
		\item $t(F,W) = 0$ for any ordered graph $F \notin \calH$.
		\item $\Prx[\bG(k,W)\in \calH]=1$ for every $k$.
	\end{itemize}
\end{lemma}
\begin{proof}
	The second and third conditions are clearly equivalent: for any fixed $k$, the probability that $\bG(k,W)\notin \calH$ is the sum of $t(F, W)$ over all $F \notin \calH$ on $k$ vertices. We now show equivalence between the first condition and the other two. If $W \in \barH$ then there exists a sequence of ordered graphs $G_n \in \calH$ with $|V(G_n)| \to \infty$ that converges to $W$ in $d_{\triangle}$. Fix $k$. For any $n$, the density $t(F, W_n)$ is bounded by the probability that, when picking $k$ vertices from a set of $n$ vertices with repetition, some vertex will be picked more than once (see Observation \ref{Obs:GraphsToOrderons} for more details). This probability is bounded by $\binom{k}{2}/n \to 0$, that is, $t(F, W_n) \to 0$ as $n \to \infty$.
	Since $G_n \to W$, it follows that $t(F, W) = \lim t(F, W_n) = 0$. The converse follows immediately from the third condition combined with our sampling theorem, which together show that with high probability, a large enough sample of $W$ will both satisfy $\calH$ and be arbitrarily close to $W$ in $d_{\triangle}$.
\end{proof}

We will need the definition of a \emph{flexible} property, given below; this is an ordered analogue of flexible properties in the unordered setting, defined in \cite{LS10}.

\begin{definition}[support, flexibility]
	For an orderon $W \in \calW$ and a value $0 \leq \alpha \leq 1$, the $\alpha$-support of $W$, denoted $\supp_\alpha(W)$, is the set of all pairs $(u,w) \in ([0,1]^2)^2$ for which $W(u,v) = \alpha$.
	
	We say that a property $\calR$ is \emph{flexible} if for any $W, W' \in \calW$ where $\supp_\alpha(W) \subseteq \supp_\alpha(W')$ for $\alpha = 0,1$ and $W \in  \calR$ it holds that $W' \in \calR$ as well. In particular, this means that the supports $\supp_{\alpha}(W)$ for $\alpha = 0,1$ fully determine whether $W$ satisfies a given flexible property.
\end{definition}

%

\begin{lemma}\label{lem:U-flexing} The closure $\barH$ of a hereditary property is flexible.
\end{lemma}
\begin{proof}
	By Lemma~\ref{lem:conditions_closure}, the closure of a hereditary property can be defined by conditions of the form $t(F,U)=0$ (for all graphs $F\notin \calH$). For every fixed $F\notin \calH$, this condition means that for almost all vectors $v\in [0,1]^{|V(F)|}$, at least one of the factors in
	$\prod_{{(i,j)\in E_{<}}}U(v_i,v_j)\cdot\prod_{{(i,j)\in \overline{E}_{<}}}(1-U(v_i,v_j))$
	must be $0$. This condition is preserved if values of $W$ that are \emph{strictly} between $0$ and $1$ are changed to any other value (including $0$ and $1$).
\end{proof}

The next lemma shows, using the flexibility, that the distance from the closure of a hereditary property is a concave function over the space of orderons.

\begin{lemma}\label{lem:distance_is_concave}
	Let $\calH$ be a hereditary ordered graph property.
	The distance $d_1(\cdot, \barH)$ is a concave function over $\calW$.
\end{lemma}
\begin{proof}
	Let $W_1, W_2, U \in \calW$ be three orderons satisfying $U = \lambda W_1 + (1-\lambda) W_2$ for some $0 < \lambda < 1$. It suffices to show that $d_1(U, \barH) \geq \lambda d_1(W_1, \barH) + (1-\lambda)d_1(W_2, \barH)$. The crucial observation is that $\supp_\alpha(U) = \supp_\alpha(W_1) \cap \supp_\alpha(W_2)$ for $\alpha = 0, 1$.
	
	Let $U'$ be any orderon in $\barH$. We show that there exist $W'_1, W'_2 \in \barH$ so that $$\lambda d_1(W_1, W'_1) + (1-\lambda) d_1(W_2, W'_2) \leq d_1(U, U').$$
	Indeed, for $i = 1,2$ pick $W'_i$ as the unique orderon satisfying the following: $W'_i(u,v) = 0$ for $(u,v) \in \supp_0(U')$; 
	$W'_i(u,v) = 1$ for $(u,v) \in \supp_1(U')$; 
	and otherwise, $W'_i(u,v) = W_i(u,v)$. By Lemma~\ref{lem:U-flexing}, $W'_1, W'_2 \in \barH$, and it is easy to verify that $W'_1, W'_2$ satisfy the desired inequality. The proof of the lemma now follows by taking the infimum over all $U',W_1',W_2'\in \calH$.
\end{proof}


Recall that an orderon $W$ is called \emph{naive} if its values do not depend on the second (``variability'') coordinate. That is, for every $x,y\in[0,1] $ it holds that $W((x,a), (y,b)) = W((x, a'), (y, b'))$ for any $a,a',b,b' \in [0,1]$.
Our next goal is to show, using the concavity, that for any hereditary property $\calH$ and orderon $W$ there exists a naive orderon $W'$ at least as far from $\barH$ as $W$.

\begin{lemma}
	For any hereditary property $\calH$ and orderon $W$ there exists a naive orderon $W'$ so that $d_1(W', \barH) \geq d_1(W, \barH)$.
\end{lemma}

\begin{proof}
	To prove the lemma, we combine the concavity guaranteed by the previous lemma with some measure-theoretic and probabilistic tools.
	For any $(x,y) \in [0,1]^2$ define $\mu^{x,y}$ as the expectation, over all $(a,b) \in [0,1]^2$, of $W((x,a), (y,b))$. We will prove that the unique $W' \colon ([0,1]^2)^2 \to [0,1]$ for which $W'((x,a),(y,b)) = \mu^{x,y}$ for all $(x,y)$ satisfies the conditions of the lemma. (To show that $W'$ is a naive orderon one needs to prove that $W'$ as defined above is indeed a measurable function, but this will follow from the proof.)
	
	Fix an integer $n$ and pick $A_n = \{(x,y) \in [0,1]^2 : Q_n(x) \neq Q_n(y) \}$, where the function $Q_n$ is as defined in the beginning of Section~\ref{sec:block_orderons}.
	Consider the family of measure-preserving bijections $f_{\alpha_1, \alpha_2, \ldots, \alpha_n}$, where $(\alpha_1, \ldots, \alpha_n) \in [0,1]^n$, defined as follows: For any $(x,a) \in [0,1]^2$, we set $f_{\alpha_1, \alpha_2, \ldots, \alpha_n}(x,a) = (x, a+\alpha_{Q_n(x)})$. Note that $\shift{f_{\alpha_1, \alpha_2, \ldots, \alpha_n}} = 0$ always holds. As these bijections clearly do not change the $L^1$-distance of $W$ from $\barH$, we have that $d_1(W, \barH) = d_1(W^{f_{\alpha_1, \alpha_2, \ldots, \alpha_n}}, \barH)$ for any choice of $(\alpha_1, \ldots, \alpha_n) \in [0,1]^n$.
	
	Pick $n$ tuples $\{(\balpha^i_1, \ldots, \balpha^i_n)\}_{i=1}^{n}$ uniformly at random. Set $\bg_i = f_{\balpha^i_1, \ldots, \balpha^i_n}$ and consider the orderon
	$$U_n = \frac{1}{n} \sum_{i=1}^{n} W^{\bg_i} = \frac{1}{n} \sum_{i=1}^{n} W^{f_{\balpha^i_1, \ldots, \balpha^i_n}}.$$
	Since each $\bg_i$ is a measure-preserving bijection with shift zero, we have that $d_1(W, \barH) = d_1(W^{\bg_i}, \barH)$. As $U_n$ is a convex combination of the $W^{\bg_i}$, we get that
	$$
	d_1(U_n, \barH) \geq \frac{1}{n} \sum_{i=1}^{n} d_1(W^{\bg_i}, \barH) = d_1(W, \barH).
	$$
	
	Now, given any fixed orderon $U \in \{U_n, W^{\bg_1}, W^{\bg_2}, \ldots, W^{\bg_n}\}$ and $(x,y) \in [0,1]^2$, define the random variable $\bX_U^{x,y}$ as follows: pick $(\ba,\bb) \in [0,1]^2$ uniformly at random, and return the value $U((x,\ba), (y,\bb)) - \mu^{x,y}$. Note that the expectation of $\bX_U^{x,y}$ for any $U$ as above is zero.
	
	A basic probabilistic fact states that, for $n$ i.i.d.\ bounded random variables $\bY_1, \ldots, \bY_n$, the variance of $\ol{ \bY }= (\bY_1 + \cdots + \bY_n) / n$ equals $\Var(\bY_1) / n$, and tends to zero as $n \to \infty$. Fix some $(x,y) \in A_n$, and for any $i \in [n]$ take $\bY_i = \bX_{W^{\bg_i}}^{x,y}$. Finally pick $\ol{ \bY} = \bX_{U_n}^{x,y}$, and note that the $\bY_i$ are indeed i.i.d., that $|\bY_i| \leq 1$ (meaning that $\Var(\bY_i) \leq 1$), and that $\ol{ \bY }= \sum \bY_i / n$. It follows that the variance of $\ol{ \bY }= \bX_{U_n}^{x,y}$ is bounded by $1/n$.
	
	Consequently, for fixed $(x,y) \in A_n$, the probability (ranging over the choices of $\bg_1, \ldots, \bg_n$ and $\ba,\bb \in [0,1]^2$) that $|U_n((x,\ba), (y,\bb)) - \mu^{x,y}| \geq 1/n^{1/3}$ is bounded by $1/n^{1/3}$. Thus, the expected measure of the collection
	$$B_n \eqdef \{(x,a,y,b) : |U_n((x,a), (y,b)) - \mu^{x,y}| \geq 1/n^{1/3} \}$$
	of ``bad'' $(x,a,y,b)$ tuples is bounded by $1/n^{1/3}$ where the randomness is over the choices of $\bg_1, \ldots, \bg_n$. In particular, there exists such a choice of $U_n$ for which $|B_n| \leq 1/n^{1/3}$; we henceforth fix this choice of $U_n$.
	
	As $|[0,1]^2 \setminus A_n| = 1/n^2 \to 0$ as $n \to \infty$, and $\int_{A_n }|U_n - U_{n'}| \leq 4 / n^{1/3}$ for $n' > n$, it follows that the sequence $\{U_n\}$ is $L^1$-Cauchy. Thus, it converges in $L^1$ to some limit $U'$, and clearly the set of points $\{(x,y) \in [0,1]^2 : U'(x,y) \neq \mu^{x,y}\}$ is of measure zero. In other words, $U'$ is measurable and equals $W'$ defined above almost everywhere, and thus $W'$ is measurable as well and $U_n \to W'$ in $L^1$. As $d_1(U_n, \barH) \geq d_1(W, \barH)$ for any $n$, the inequality still holds at the limit, that is, $d_1(W', \barH) \geq d_1(W, \barH)$. This concludes the proof.
\end{proof}

\begin{definition}
	For any hereditary property $\calH$ let $\Delta_{\calH}$ denote the supremum of the distance $d_1(W, \barH)$ among all $W \in \calW$.
\end{definition}
The following is an immediate consequence of the last lemma.
\begin{lemma}\label{lem:far_naive}
	For any hereditary property $\calH$ and $\eps > 0$ there exists a naive orderon $W'$ for which $d_1(W', \barH) \geq \Delta_\calH - \eps$.
\end{lemma}


With some abuse of notation, we henceforth view a naive orderon $W$ as a measurable symmetric function from $[0,1]^2$ to $[0,1]$ and for $x,y \in [0,1]$ we denote by $W(x,y)$ the (unique) value of $W((x,a), (y,b))$ for $a,b \in [0,1]$.

Recall the definition of a naive block orderon from Section~\ref{sec:block_orderons} (in particular, the fact that the blocks are \emph{consecutive} in terms of order). The next lemma asserts that there exist naive block orderons with a bounded number of blocks, that are almost the furthest away from $\barH$.

\begin{lemma}\label{lem:naive_block_orderon_far}
	Let $\calH$ be a hereditary property and let $\eps > 0$. There exists a naive block orderon $W$ with at most $M_\calH(\eps)$ blocks so that $d_1(W, \barH) \geq \Delta_\calH - \eps$.
\end{lemma}
\begin{proof}
	The proof relies on a fundamental fact in Lebesgue measure theory, stating that (finite) linear combinations of indicator functions of the form $I = I_{[a_i, b_i] \times [c_i, d_i]}$ are dense in the two-dimensional Lebesgue space $L^1[\mathbb{R}^2]$ (and specifically, in $L^1[[0,1]^2]$).
	Indeed, Theorem 2.4 (ii) in \cite{PrincetonLectures} states that for any function $W$ in $L^1[\mathbb{R}^2]$ and $\delta > 0$ there exists $N = N(W, \eps) \geq 0$ and a step function $T$ of the form
	$$
	T = \sum_{i=1}^{N} \alpha_i \cdot I_{[a_i, b_i] \times [c_i, d_i]},
	$$
	where $\alpha_i \in \mathbb{R}$, so that $d_1(W,T) \leq \delta$.
	In our case, $W$ is a naive orderon, and is thus a symmetric function in $L^1[[0,1]^2]$.
	We note that, while Theorem 2.4 (ii) in~\cite{PrincetonLectures} does not guarantee that the function $T$ approximating $W$ is symmetric, we can easily make $T$ symmetric---and in particular, a naive orderon---by replacing it with $(T + \tilde{T})/2$, where $\tilde{T} = \sum_{i=1}^{N} \alpha_i \cdot I_{[c_i, d_i] \times [a_i, b_i]}$.
	
	We now describe how to construct a naive block orderon which is $L^1$-close to the above $T$.  Recall the definition of $Q_n$ given in the beginning of Section~\ref{sec:block_orderons}, and for $M \geq 4N / \delta$, consider the following weighted ordered graph $G$ on $M$ vertices. Given $(i,j) \in [M]^2$, if $T$ is constant over all pairs $(x,y) \in [0,1]^2$ for which $Q_M(x) = i$ and $Q_M(y) = j$, then we define $G(i,j)$ to be this constant value (and say that $(i,j)$ is good). Otherwise, $G(i,j)$ is defined arbitrarily.
	The number of pairs $(i,j)$ that are not good is at most $4MN$. Indeed, $(i,j)$ might not be good only if the square $[(i-1)/M, i/M] \times [(j-1)/M, j/M]$ intersects the boundary of $[a_\ell, b_\ell] \times [c_\ell, d_\ell]$ for some $\ell \in [N]$. However, the latter boundary intersects at most $4M$ such squares, and summing over all $\ell \in [N]$ we get the desired bound.
	
	Now let $U$ be the naive block orderon over $M$ blocks defined by $U(x,y) = G(Q_M(x), Q_M(y))$ for any $(x,y) \in [0,1]^2$. We get that $d_1(U, T) \leq 4MN / M^2 \leq \delta$, where the last inequality holds by our choice of $M$. It thus follows that $d_1(W,U) \leq d_1(W,T) + d_1(U,T) \leq 2\delta$.
	
	To conclude the proof, let $\calH$ be a hereditary property, let $\eps > 0$, and let $W$ be an arbitrary naive orderon satisfying that $d_1(W, \barH) \geq \Delta_\calH - \eps/3$ (the existence of such a $W$ is guaranteed by Lemma~\ref{lem:far_naive}); specifically we can take such a $W$ that minimizes the quantity $N(W, \eps)$, which was defined in the beginning of the proof. Now take $\delta = \eps/3$ and let $U$ denote the naive block orderon over $M = \lceil 4N / \delta \rceil$ blocks which satisfies $d_1(U, W) \leq 2\delta = 2\eps/3$. By the triangle inequality, we conclude that $d_1(U, \barH) \geq \Delta_\calH - \eps$, as desired. Note that $M$ depends only on $\eps$ and $N$, which in turn depends only on $\calH$ and $\eps$.
\end{proof}

We now show how our results for orderons can be translated to finite graphs. Here we make use of several technical lemmas from~\cite{LS10}. The first lemma that we need is the following.

\begin{lemma}[Lemma 3.13(a) in~\cite{LS10}]\label{lem:upper_bound_distance_from_hered}
	For any hereditary property $\calH$ and simple ordered graph $G$, we have $d_1(G, \calH) \leq d_1(W_G, \barH) \leq \Delta_{\calH}$.
\end{lemma}
While the above lemma was stated in~\cite{LS10} for hereditary properties of unordered graphs, its proof in~\cite{LS10} only uses the flexibility of a closure of a hereditary property and a simple subgraph statistics argument, and translates as-is to our ordered setting.

The next lemma that we need from~\cite{LS10} is the following.
\begin{lemma}[Lemma 2.8 in~\cite{LS10}]\label{lem:cut_norm_liminf}
	Suppose that $\|U_n - U\|_{\square} \to 0$ and $\|W_n - W\|_{\square} \to 0$ (where $U, U_n, W, W_n$ are naive orderons). Then $\liminf_{n \to \infty} d_1(W_n, U_n) \geq d_1(W, U)$.
\end{lemma}

Additionally, we need the following lemma, concerning the good behavior of sequences that converge to a naive block orderon.
\begin{lemma}\label{lem:block_good_behavior}
	Fix an integer $M > 0$, and let $W$ be a naive $M$-block orderon. Also let $\{W_n\}$ be a sequence of orderons where $d_{\triangle}(W_n, W) \to 0$, and let $\{f_n\}$ be a sequence of shift functions with $\shift{f_n} \to 0$. Then it holds that $\|W_n^{f_n} - W\|_{\square} \to 0$.
\end{lemma}
\begin{proof}
	Since $d_{\triangle}(W_n, W) \to 0$, by definition of the cut-shift distance there exists a sequence $g_n$ of shift functions with $\shift{g_n} \to 0$ so that $\|W_n^{g_n} - W \|_{\square} \to 0$. By applying the shift function $h_n = f_n \circ g^{-1}_n$ to both orderons in the expression, we get that $\|W_n^{f_n} - W^{h_n}\|_\square \to 0$. On the other hand, we shall show now that $d_1(W^{h_n}, W) \to 0$. Since the cut norm is always bounded by the $L^1$-distance, we can conclude that $\|W^{h_n} - W\|_{\square} \to 0$, implying by the triangle inequality that $\|W_n^{f_n} - W\|_{\square} \to 0$.
	
	To show that $d_1(W^{h_n}, W) \to 0$, observe first that $\shift{h_n} \leq \shift{f_n} + \shift{g_n} \to 0$. For any pair $(x,y) \in [0,1]^2$ where both $x$ and $y$ are not $\shift{h_n}$-close to the boundary of their block (i.e., are not of the form $i/M \pm \zeta$ where $\zeta \leq \shift{h_n}$), we thus have that $W(x,y) = W^{h_n}(x,y)$. We conclude that $d_1(W^{h_n}, W)$ is bounded by the total volume of pairs $(x,y)$ that are $\shift{h_n}$-close to the boundaries, which is bounded by $4M \cdot \shift{f_n}$, and tends to zero as $n \to \infty$.
\end{proof}

Lemma~\ref{lem:upper_bound_distance_from_hered} gives us a global upper bound of $\Delta_\calH$ on the distance of any ordered graph from a hereditary property $\calH$. Our next main lemma, given below, provides an asymptotic lower bound. The statement of the lemma is analogous to (a special case of) Proposition 3.14 in~\cite{LS10}, although the proof is slightly different (and makes use of Lemma~\ref{lem:block_good_behavior}).

\begin{lemma}\label{lem:lim_inf_L1_dist}
	Let $\calP$ be any ordered graph property and let $G_n \to W$ be an ordered graph sequence that converges (in $d_{\triangle}$) to a naive block orderon. Then
	$$
	\liminf_{n \to \infty} d_1(G_n, \calP) \geq d_1(W, \barP).
	$$
\end{lemma}
\begin{proof}
	For any $n \in \mathbb{N}$ let $H_n \in \calP$ be a graph with the same size as $G_n$ satisfying $d_1(G_n, H_n) = d_1(G_n, \calP)$ (this minimum is always attained as $d_1(G_n, \calP)$ is a minimum of finitely many values). By taking a subsequence of $G_n$ for which the distance to $\calP$ converges to $\liminf d_1(G_n, \calP)$ and then taking a subsequence of it to ensure convergence of the corresponding subsequence of $H_n$, we may assume that $H_n \to U \in \barP$ (as usual, the convergence is in $d_{\triangle}$).
	
	By the definition of the cut-shift distance, for any $H_n$, there exists a shift function $f_n$ so that $\|W_{H_n}^{f_n} - U\|_{\square} \to 0$ and furthermore $\shift{f_n} \to 0$.
	Applying Lemma~\ref{lem:block_good_behavior} to the sequence $W_{G_n}$, which by the assumption of this lemma converges to the naive block orderon $W$, we conclude that $\|W_{G_n}^{f_n} - W\|_{\square} \to 0$. By applying Lemma~\ref{lem:cut_norm_liminf} to the sequences $W_{G_n}^{f_n}$ and $W_{H_n}^{f_n}$ which converge in cut norm to $W$ and $U$ respectively, we conclude that
	$$
	d_1(W, U) \leq \liminf_{n \to \infty} d_1(W_{G_n}^{f_n}, W_{H_n}^{f_n}) = \liminf_{n \to \infty} d_1(G_n, H_n) = \liminf_{n \to \infty} d_1(G_n, \calP),
	$$
	where the first equality follows from the fact that shifting two orderons by the same shift does not change the $L^1$-distance between them, and the second equality follows from our choice of $H_n$.
\end{proof}

We are now ready to put it all together and prove Theorem~\ref{thm:ordered_Alon_Stav}.

\begin{proofof}{Theorem~\ref{thm:ordered_Alon_Stav}}
	By Lemma~\ref{lem:upper_bound_distance_from_hered} we know that $\overline{d_\calH} \leq \Delta_{\calH}$. Thus, it suffices to show the statement of the theorem with $\overline{d_\calH}$ replaced by $\Delta_{\calH}$.
	By Lemma~\ref{lem:naive_block_orderon_far}, there exists a naive block orderon $W$, whose number of blocks $M$ is only a function of $\calH$ and $\eps$, for which $d_1(W, \barH) \geq \Delta_{\calH} - \eps/2$. By Lemma~\ref{lem:lim_inf_L1_dist}, there exist $\delta > 0$ and $N$ so that any ordered graph $G$ on $n \geq N$ vertices with $d_\triangle(W_{G}, W) \leq \delta$ satisfies $d_1(G, \calH) \geq d_1(W, \barH) - \eps/2 \geq \Delta_{\calH} - \eps$. Pick $G$ according to the random model $\bG(n, W)$. From our sampling result, Theorem~\ref{thm:Sampling}, the probability that $d_{\triangle}(W_{\bG(n, W)}, W) \leq \delta$ tends to one as $n \to \infty$. As the random model $\bG(n,W)$ is precisely a consecutive stochastic block model on $M$ blocks with parameters as in the statement of the theorem, the proof follows.
\end{proofof}

\section{Parameter estimation }\label{sec:applications}

In this section we prove Theorem~\ref{thm:Param-Testing}. Recall the definition of natural estimability from Subsection~\ref{subsec:main_results}.
First, observe that the random graph distributions $\bG|_{\bk}$ (no vertex repetitions) and $\bG(k, W_G)$ (allowing vertex repetitions) are very close in terms of variation distance.
\begin{observation}\label{Obs:GraphsToOrderons}Let $G$ be an ordered graph on $n$ vertices. For every fixed $k$ and a large enough $n$ the distribution $\bG|_{\bk}$ is arbitrarily close (in variation distance) to the distribution $\bG(k,W_G)$.
\end{observation}
We now turn to the proof of the theorem. First we prove that natural estimability of $f$ implies convergence of $f(G_n)$ for any convergent $\{G_n\}$; then we prove that the latter condition implies the existence of a continuous functional on orderons that satisfies both items of the last condition of Theorem~\ref{thm:Param-Testing}; and finally, we prove the other direction of both statements.

\eqref{thm:Estim-b} $\implies$~\eqref{thm:Estim-c}: Let $\{G_n\}_{n\in\mathbb{N}}$ be a convergent sequence with $|V(G_n)|\to \infty$. Given $\eps>0$, let $k$ be such that for every ordered graph $G$ on at least $k$ nodes, $|f(G)-\Ex_{\bG|_{\bk}}[f(\bG|_{\bk})]|\le 2\eps$ (this can be done by setting $\delta=\eps/M$ where $M$ is an upper bound on the values $f$). By the fact that $\{G_n\}_{n\in\mathbb{N}}$ is convergent, $t(F,W_{G_n})$ tends to a limit for all ordered graphs $F$ on $k$ vertices, which by Observation~\ref{Obs:GraphsToOrderons} implies that $\lim\limits_{n\to\infty}\Prx_{\bG_{\boldsymbol{n}}|_{\bk}}[\bG_{\boldsymbol{n}}|_{\bk}=F]=\lim\limits_{n\to\infty}t(F,W_{G_n})=t(F,W)$. Therefore,
\[r_k\eqdef \lim\limits_{n\to\infty }\Ex_{\bG_{\boldsymbol{n}}|_{\bk}}[f(\bG_{\boldsymbol{n}}|_{\bk})]=\sum_F\lim\limits_{n\to\infty} \Prx_{\bG_{\boldsymbol{n}}|_{\bk}}[\bG_{\boldsymbol{n}}|_{\bk}=F]\cdot f(F)=\sum_F t(F,W)\cdot f(F).\] Thus, for all sufficiently large $n$,
\[|f(G_n)-r_k|\le |f(G_n)-\Ex_{\bG_{\boldsymbol{n}}|_{\bk}}[f(\bG_{\boldsymbol{n}}|_{\bk})]|+\eps\le 3\eps,\] which implies that $\{f(G_n)\}_n$ is convergent.

(\ref{thm:Estim-c})$\implies$(\ref{thm:Estim-d}): For a sequence $\seq{G_n}{n}$ converging to $W$, let $\hat{f}(W)\eqdef\lim\limits_{n\to\infty}f(G_n)$. Note that this quantity is well-defined: Given two ordered graph sequences $\seq{G_n}{n}$ and $\seq{H_n}{n}$ converging  to $W$, we can construct a new sequence $\seq{S_n}{n}$, such that $S_{2n}=G_n$ and $S_{2n-1}=H_n$. By definition, the sequence $\seq{S_n}{n}$ also converges to $W$, and hence $\lim\limits_{n\to\infty}{f(H_n)}=\lim\limits_{n\to\infty}f(G_n)=\lim\limits_{n\to\infty}f(S_n)=\hat{f}(W)$.

To prove~\eqref{thm:Estim-d1}, assume that $\{W_n\}_{n\in\mathbb{N}}\in \calW$ converges to $W$. For every $n$, we can apply Theorem~\ref{thm:Sampling} and obtain a sequence $\{G_{n,k}\}_k$ such that $\lim\limits_{k\to\infty}W_{G_{n,k}}=W_n$. In addition, we can pick a subsequence of $\{G_{n,k}\}_k$ such that \[\csdist(W_n,W_{G_{n,k}})\le2^{-k}\qquad\text{and}\qquad\left|\hat{f}(W_n)-f(G_{n,k})\right|\le 2^{-k}. \]
Now, since for every $n$ the sequence $G_{n,k}$ converges uniformly to $W_n$ (as $k\to \infty$) and $W_n$ converges to $W$, we have that the diagonal sequence $G_{n,n}$ converges  to $W$ as well. Therefore, by the fact that $\hat{f}$ is well defined, we have that $\lim\limits_{n\to\infty}f(G_{n,n})=\hat{f}(W)$. Therefore, for every $\eps>0$, we can find $N$ such that for all $n>N$, $\left|{f}(G_{n,n})-\hat{f}(W)\right|\le \eps/2$. Then, for every $\eps>0$ and all $n\ge \max({N,\log(2/\eps)})$ we have $$\left|\hat{f}(W_n)-\hat{f}(W)\right|\le \left|\hat{f}(W_n)-f(G_{n,n})\right|+\left|f(G_{n,n})-\hat{f}(W)\right|\le\eps/2+\eps/2=\eps,$$
which concludes the proof for $(\ref{thm:Estim-d1})$.

To prove (\ref{thm:Estim-d2}): Assume towards a contradiction that (\ref{thm:Estim-d2}) does not hold. Let $\{G_n\}_n$ be an ordered graph sequence with $|V(G_n)|\to \infty$ such that $|f(G_n)-\hat{f}(W_{G_n})|>\eps$.  For each $n$, consider the sequence $\{G^{\otimes j}_n\}_{j\in\mathbb{N}}$ where  $G_n^{\otimes j}$ is the $j$ blow-up of $G_n$. Note that by the fact that for every $j$ we have that $W_{G_n^{\otimes j}}=W_{G_n}$, the sequence $\{G^{\otimes j}_n\}_{j\in\mathbb{N}}$ converges to $W_{G_n}$. Therefore, by the above construction of $\hat{f}$, $\{f(G^{\otimes j}_n)\}$ converges to $\hat{f}(W_{G_n})$. Thus, for each $n$ we can find $j_n\in \mathbb{N}$ such that $|f(G^{\otimes j_n}_n)-\hat{f}(W_{G_n})|\le \eps/2$. Combined with the triangle inequality, this implies that $|f(G_n)-f(G^{\otimes j_n}_n)|>\eps/2$. By compactness we can assume that $\{G_n\}_n$ has a subsequence $\{G'_n\}_n$ that converges to $W$.
Construct a sequence $\{H_n\}_n$ such that $H_{2n-1}=G'_n$ and $H_{2n}=G'^{\;\otimes j_n}_n$. Note that the sequence $\{H_n\}_n$ converges to $W$ as well (since $\csdist(W_{G_n},W_{G'^{\;\otimes j_n}_n})=0$), but $\{f(H_n)\}_n$ does not converge, as it is not Cauchy, which is a contradiction to (\ref{thm:Estim-c}).

\eqref{thm:Estim-d}$\implies$~\eqref{thm:Estim-c}: Consider any convergent ordered graph sequence $\{G_n\}_{n\in\mathbb{N}}$ such that $|V(G_n)|\to \infty$, and let $W\in\calW$ be its limit. Then, $\csdist(W_{G_n},W)\to 0$, and by the continuity of $\hat{f}$, we have that $\hat{f}(W_{G_n})-\hat{f}(W)\to 0$. Namely, for every $\eps>0$, we can find $N$ such that for all $n\ge N$, $\left|\hat{f}(W_{G_n})-\hat{f}(W)\right|\le \eps/2$. By assumption, we also have that for every $\eps>0$ there is a $k(\eps)$ such that for all $|G_n|\ge k$, $|f(G_n)-\hat{f}(W_{G_n})|\le \eps/2$. This implies that for a large enough $n$, $|f(G_n)-\hat{f}(W)|\le \eps$, concluding the proof.

\eqref{thm:Estim-c}$\implies$~\eqref{thm:Estim-b}: Assume towards a contradiction that~\eqref{thm:Estim-b} does not hold. Namely, that there are $\eps>0$ and $\delta>0$ such that for all $k$, there is $G$ on at least $k$ vertices such that $|f(G)-f(\bG|_{\bk})|>\eps$ with probability at least $\delta$. Suppose we have a sequence $\{G_k\}_k$ where $|V(G_k)|=n(k)\to \infty$, and $|f(G_k)-f(\bG_{\bk}|_{\bk})|>\eps$ with probability at least $\delta$. By the compactness of $\widetilde\calW$, we can select a subsequence $\{G'_k\}$ of $\{G_k\}_k$ such that $\{G'_k\}_k$ converges to some $W\in\widetilde{\calW}$. Using Theorem~\ref{thm:Sampling} (along with a union bound on the confidence probabilities) and the assumption of~\eqref{thm:Estim-c},  for every $k$,  let $H_k=G'_k|_k$ be some specific induced subgraph such that $\csdist(W_{G'_k},W_{H_k})=o_k(1)$ and $|f(H_k)-f(G'_k)|>\eps$. Note that by the triangle inequality, the sequence $ \{H_k\}_k$ converges to $W$. Let $\{S_\ell\}_\ell$ be the sequence where $S_{2\ell}=H_\ell$ and $S_{2\ell-1}=G'_\ell$. Since both $\{H_\ell\}_\ell$ and $\{G'_\ell\}_\ell$ converge to $W$, the sequence $\{S_\ell\}_\ell$ also converges to $W$. However, the sequence $\{f(S_\ell)\}_\ell$ does not converge (as it is not Cauchy), which is a contradiction to~\eqref{thm:Estim-c}.

\section{Pixelization and an Ordered Removal Lemma}\label{sec:hereditary}
This section contains the proof of Theorem~\ref{thm:removal_lemma}. 

Let $\calH$ be a hereditary property of simple ordered graphs, and recall that $\barH$ denotes the \emph{closure} of the property $\calH$. 
Our main technical lemma is as follows.
\begin{lemma}\label{lem:non-uniform-hereditary} For every hereditary property $\calH$, orderon $U \in \barH$, and parameter
	$\eps>0$, there exists $\delta(\calH,U,\eps) > 0$, so that if $W$ is such that $\csdist(W,U)\le \delta(\calH,U,\eps)$, then $d_1(W,\barH)\le \eps$.
\end{lemma}
For any orderon $U$, define the property $\calH_U$ as that containing precisely all simple ordered graphs $F$ where $t(F,U) > 0$. It is clear that if $H \subseteq F$ and $t(F,U) > 0$ then $t(H,U) > 0$ as well, and so $\calH_U$ is \emph{hereditary}. 
From Lemma \ref{lem:conditions_closure} we know that $\calH_U$ is minimal in the following sense: any hereditary property $\calH$ satisfying that $U \in \barH$ has $\barHU \subseteq \calH$. Indeed, if $U \in \barH$ then the aforementioned lemma implies that $t(F,U) = 0$ for any $F \notin \calH$; In other words, if $t(F,U) > 0$, then $F \in \calH$. Since the former condition holds for all $F \in \calH_U$, we have that $\calH_U \subseteq \calH$.
Now, with the minimality of $\calH_U$ in hand, in order to prove Lemma \ref{lem:non-uniform-hereditary} it suffices to prove the following simpler statement.
\begin{lemma}\label{lem:HU_removal}
For every orderon $U$ and $\eps > 0$ there exists $\delta = \delta(U, \eps) > 0$ so that if $\csdist(W, U) < \delta$ then $d_1(W, \overline{\calH_U}) < \eps$.
\end{lemma}

Note that the parameter $\delta$ depends on the object $U$.
However, since $\barH$ is a closed set in a compact space, it is also compact by itself, which implies the following removal lemma for orderons.




\begin{lemma}\label{lem:dtri-continouity} For any $\eps>0$ there exists $\delta = \delta_{\barH}(\eps) > 0$, so that if $\csdist(W,\barH)\le \delta$, then $d_1(W,\barH)\le \eps$.
\end{lemma}

\begin{proof} Fix $\eps>0$. Lemma~\ref{lem:non-uniform-hereditary} implies that for every $U\in\barH$ there exists $\delta(U, \eps) > 0$ such that if $\csdist(W,U)\le \delta(U, \eps)$ then $d_1(W,\barH)\le \eps$.

	Assume that $U'\in\barH$  and $\csdist(U',U)\le \eta$. Clearly
	$\delta(U',\eps)\ge \delta(U,\eps)-\eta$ by the triangle inequality. So for a fixed $\eps$ and every $U$ there is a ball with radius $\delta(U, \eps) / 2$ (in cut-shift distance) around $U$ with a guaranteed lower bound of $\delta(U,\eps)/2$ on $\delta(U',\eps)$ for any $U'$ in that ball. By compactness, we can cover $\barH$ with a finite subset of this set of balls, obtaining a positive universal lower bound on $\delta(U,\eps)$ for every $U\in\barH$.
\end{proof}

Next, we describe how to derive the proof of Theorem~\ref{thm:removal_lemma}. For this we need the following lemma. 


\begin{lemma}\label{lem:finite_bad_examples}
	For any $\eps>0$ there are only finitely many ordered graphs $H\in \calH$ with $d_1(W_H,\barH)>\eps$.
\end{lemma}
\begin{proof} 
	Suppose that there is a sequence of ordered graphs $\{G_n\}_{n\in\mathbb{N}}$ in $\calH$ such that $d_1(W_{G_n},\barH)\geq\eps$. As the total number of ordered graphs with up to $n$ vertices is bounded as a function of $n$, we may assume that $|V(G_n)| \to \infty$ as $n \to \infty$. Furthermore, by compactness, we can assume that the sequence converges to some $W$, that is, $d_{\triangle}(W_{G_n}, W) \to 0$. By definition of the closure $\barH$, it follows that $W\in\barH$.
	By Lemma~\ref{lem:dtri-continouity},	$d_1(W_{G_n}, \barH) \to 0$ as well, which is a contradiction.
\end{proof}

\begin{proofof}{Theorem \ref{thm:removal_lemma}}
	Fix $\eps, c > 0$. By Lemma~\ref{lem:dtri-continouity}, there exists $\delta > 0$ such that any orderon $W$ with $d_1(W, \barH) \geq \eps$ satisfies $d_{\triangle}(W, \barH) \geq \delta$. On the other hand, by Lemma \ref{lem:finite_bad_examples}, 
	there exists some $k_1 \in \mathbb{N}$ so that any simple ordered graph $H \in \calH$ on at least $k_1$ vertices satisfies $d_1(W_H, \barH) \leq \delta/2$.
	Furthermore, by Theorem~\ref{thm:Sampling} and Observation~\ref{Obs:GraphsToOrderons}, there exist integers $k_2 \geq s \geq k_1$ satisfying the following. For any graph $G$ on $n \geq k_2$ vertices, with probability at least $1-c$ it holds that $d_{\triangle}(W_{\bG|_{\bs}}, W_G) < \delta/2$. Set $k =  k_2$ in the statement of the theorem.
	
	Let $G$ be a simple ordered graph on $n \geq k$ vertices with $d_1(G, \calH) \geq \eps$.
	By Lemma~\ref{lem:upper_bound_distance_from_hered}, we have $d_1(W_G, \barH) \geq \eps$. 
	From the above paragraph we know that $d_{\triangle}(W_G, \barH) \geq \delta$.
	Now let $H \sim \bG|_{\bs}$. Again by the above paragraph, with probability at least $1-c$ it holds that $d_{\triangle}(W_{H}, W_G) < \delta/2$, which means by the triangle inequality that $d_\triangle(W_H, \barH) > \delta/2$. As $s \geq k_1$, we conclude that $H \notin \calH$.
\end{proofof}

\subsection{Proof of Lemma~\ref{lem:HU_removal} }

In this subsection we provide the proof of Lemma~\ref{lem:HU_removal}.
We need to show that for every orderon $U$ and every
$\eps>0$ there exists $\delta(U,\eps) > 0$, so that if $W$ is such that $\csdist(W,U)\le \delta(U,\eps)$, then $d_1(W,\overline{\calH_U})\le \eps$.
Along the proof, we will be using the following family of orderons, called \emph{layered strip orderons}, in various occasions.\footnote{Note that naive block orderons, discussed previously in this paper, are a special case of a layered strip orderon.}

\begin{definition}\label{def:layered_strip_orderon}
An \emph{$(\ell,k)$-strip layered orderon} (where $\ell,k \in \mathbb{N}$) is a step function whose steps are the following: All rectangles of the form $([(j-1)/\ell,j/\ell]\times [(i-1)/k,i/k])\times ([(j'-1)/\ell,j'/\ell]\times [(i'-1)/k,i'/k])$ where $i,i' \in [k]$, $j,j'\in\ell$ and $j\neq j'$, and additionally for every rectangle of the form $([(j-1)/\ell,j/\ell]\times [(i-1)/k,i/k])\times ([(j-1)/\ell,j/\ell]\times [(i'-1)/k,i'/k])$ its sub-partition according to whether $((x,a),(y,b))$ satisfies $x<y$ or $y<x$.

Each rectangle of the form $([(j-1)/\ell,j/\ell] \times [0,1]$ is called a \emph{strip}, whereas each of the $k$ rectangle of the form $([(j-1)/\ell,j/\ell]\times [(i-1)/k,i/k])$ within a strip will be called a \emph{layer}.
\end{definition}

When working with $(\ell,k)$-strip layered orderons for fixed $\ell$ and $k$, a consideration of an infinite family $\calF$ of forbidden graphs can be reduced to a finite one by a well-known embedding technique by Alon and Shapira~\cite{AlonShapira08} (see \cite{ABF_FOCS17} for the ordered variant). When considering orderons (rather than discrete graphs) the argument can be boiled down to a lemma whose proof is almost trivial.

\begin{lemma}\label{lem:infinite_finite}
For every (possibly infinite) family $\calF$ and integers $\ell$ and $k$ there exists a finite family $\calF_{\ell,k}$, so every $(\ell,k)$-strip layered orderon $W$ that satisfies $t(\calF,W)>0$ satisfies also $t(\calF_{\ell,k},W)>0$.
\end{lemma}
\begin{proof}
We begin with the observation that it is enough to consider strip layered orderons whose values are all in $\{0,\frac12,1\}$. The reason is that if we take any strip layered orderon $W$, and define $W'$ so that $W'((x,a),(y,b))=W((x,a),(y,b))$ if $W((x,a),(y,b))\in\{0,1\}$ and $W'((x,a),(y,b))=\frac12$ otherwise, then clearly $t(F,W)>0$ if and only if $t(F,W')>0$ for every ordered graph $F$.

Now ignoring values of $W((x,a),(y,b))$ for $x=y$ (which do not contribute anything to $t(\calF,W)$), there are a finite number of $(\ell,k)$-strip layered orderons that take values in $\{0,\frac12,1\}$ (there are $3^{(\ell+\ell^2)k^2}$  of them to be exact). For every orderon $W$ in this set, if $t(\calF,W)>0$ then we pick one graph $F_W\in\calF$ for which $t(F_W,W)>0$. The set containing all graph $F_W$ that we picked during this process is the required finite set $\calF_{\ell,k}$.
\end{proof}

Orderons are measurable functions, and similarly to Section~\ref{sec:furthest}, they can be approximated in $L^1$ by a step function whose steps are rectangles, which can be viewed in turn as a layered strip orderon. From now on, to reduce the number of parameters, we refer only to $(\ell,\ell)$-strip layered orderons (rather than use a second parameter for the number of steps).

\begin{lemma}\label{lem:step_approx_orderon} Let $W\in{\calW}$ be an orderon. Then, there exist $\ell\in\mathbb{N}$ and a function $W^{\calR_\ell}\in {\calW}$ which is an $(\ell,\ell)$-strip layered orderon, satisfying $\|W-W^{\calR_\ell}\|_1\le o_\ell(1)$.
\end{lemma}
The proof is identical to the first two paragraphs in the proof of Lemma~\ref{lem:naive_block_orderon_far}, except that the functions we wish to approximate belong to $L^1[([0,1]^2)^2]$ (which is isomorphic, in relation to these arguments, to $L^1[[0,1]^4]$), rather than $L^1[[0,1]^2]$. For this lemma we do not have to sub-partition the rectangles of the form $([(j-1)/\ell,j/\ell]\times [(i-1)/\ell,i/\ell])\times ([(j-1)/\ell,j/\ell]\times [(i'-1)/\ell,i'/\ell])$.

Generally, in this section it will be convenient for us to work with measure-preserving bijections that ``preserve strips'', defined as follows.
\begin{definition}[strip-preserving bijections]
Let $\phi \colon [0,1]^2 \to [0,1]^2$ be a measure-preserving bijection and let $\calI_{\ell} = \{I_1^{\ell}, \ldots, I_{\ell}^{\ell}\}$ be as in Definition~\ref{def:layered_strip_orderon}. We say that $\phi$ is \emph{$\ell$-strip-preserving} (or \emph{$\ell$-block-preserving}) if for any $i \in [\ell]$ and $x \in I^{\ell}_i$, it holds that $\phi(x) \in I^{\ell}_i$.
\end{definition}

The next lemma states that any small-shift measure-preserving bijection can be approximated in an $L^1$-sense by bijections that preserve strips.
\begin{lemma}
	\label{lem:surgery}
For any $\ell \in \mathbb{N}$ and measure-preserving bijection $\phi \colon [0,1]^2 \to [0,1]^2$ with $\shift{\phi} \leq 1 / 4\ell$, there exists a measure-preserving bijection $\phi' \colon [0,1]^2 \to [0,1]^2$ that is also $\ell$-strip-preserving, so that $\|W^\phi - W^{\phi'}\|_1 \leq 4 \ell \cdot \shift{\phi}$ for any orderon $W \in \calW$. Additionally, we can require $\phi'$ to satisfy $\shift{\phi'} \leq \shift{\phi}$.
\end{lemma}
\begin{proof}
The proof makes use of several fundamental arguments in measure theory. 
For any $i \in [\ell]$ let $A_i = \{x \in I^{\ell}_i : \phi(x) \notin I^{\ell}_i\}$ and $B_i = \{x \in I^{\ell}_i : \phi^{-1}(x) \notin I^{\ell}_i\}$. 
As $\phi$ is measure preserving, we know that $A_i$ and $B_i$ have the same (Lebesgue) measure. 
Now, it is well known (e.g., \cite[Theorem 3.6]{burk1997lebesgue}) that any Lebesgue measurable set is a disjoint union of a Borel measurable set with a (Lebesgue measurable) zero-measure set. 
Write $A_i = A'_i \cup A^N_i$ and $B_i = B'_i \cup B^N_i$, where $A'_i, B'_i$ are Borel measurable and $A^N_i, B^N_i$ are zero-measure sets. 
From \cite{Nishiura98}, it follows that there exists a measure-preserving bijection $\phi'_i \colon A'_i \to B'_i$ for any $i \in [\ell]$. 

At this point, it might be tempting to try picking $\phi'$ as follows: $\phi'(x) = \phi(x)$ for any $x \in I^{\ell}_i \setminus A_i$; $\phi'(x) = \phi'_i(x)$ if $x \in A'_i$; and picking a bijection $\phi^N \colon A^N_i \to B^N_i$ so that $\phi'(x) = \phi^N(x)$ for any $x \in A^N_i$. However, such a bijection need not exist: it might be the case that $A^N_i$ and $B^N_i$ do not have the same cardinality, in the set theory sense, and in this case such a bijection cannot exist. To accommodate this, we can do the following simple ``rewiring'': pick a subset $C_i \subseteq I^\ell_i \setminus A_i$ which has measure zero and (set theoretic) cardinality $2^{\aleph_0}$ (which is the same as that of $\mathbb{R}$), and set $\phi''_i$ to be an arbitrary bijection between $A^N_i \cup C_i$ and $B^N_i \cup \phi(C_i)$. Note that such a $\phi''_i$ exists (both of these sets have the same cardinality) and is trivially measure preserving, since the Lebesgue measure is complete and both the domain and range of $\phi''_i$ are zero-measure sets. Consequently, $\phi'$ defined by $\phi'(x) = \phi(x)$ for $x \in I^\ell_i \setminus (A_i \cup C_i)$, $\phi'(x) = \phi'_i(x)$ for $x \in A'_i$, and $\phi'(x) = \phi''_i(x)$ for $x \in A^N_i \cup C_i$, is a measure preserving bijection as desired.

It remains to show that $\|W^\phi - W^{\phi'}\|_1 \leq \eps$ for any orderon $W \in \calW$. Write $\delta = \eps / 4\ell$ and suppose that $\shift{\phi} \leq \delta$. It follows that $\lambda(A_i) \leq 2\delta$ for any $i \in [\ell]$ (also recall that $\lambda(C_i) = 0$). As $W^{\phi}(x,y) \neq W^{\phi'}(x, y)$ may hold only if $\{x,y\}$ intersects $\bigcup_{i=1}^{\ell} (A_i \cup C_i)$, we conclude that $\|W^\phi - W^{\phi'}\|_1 \leq 2 \cdot 2\delta \cdot \ell = 4 \delta \ell = \eps$ for any $W \in \calW$.

To make $\phi'$ satisfy $\shift{\phi'} \leq \shift{\phi}$, we can decompose, for any $i \in \ell$, $A_i = A^-_i \cup A^+_i$ where $A^-_i = \{x \in I^{\ell}_i : \phi(x) \in I^{\ell}_{i-1}\}$ and $A^+_i = \{x \in I^{\ell}_i : \phi(x) \in I^{\ell}_{i+1}\}$ (also setting $A^-_1 = A^+_\ell = \emptyset$), and similarly for $B_i$. It follows that $\lambda(A^-_i) = \lambda(B^-_i)$ and $\lambda(A^+_i) = \lambda(B^+_i)$ for any $i \in [\ell]$, and the rest of the proof is analogous to the above.
\end{proof}

The following is our main technical lemma, about the possibility to ``pixelize'' any orderon $U$.
\begin{lemma}[pixelization lemma]\label{lem:unifomDecision}
	For any orderon $U$, and any $\eps > 0$, there exists $\ell = \ell(U, \eps)$ and a layered $\ell$-strip orderon $U' \in \barHU$ with $\ell$ steps, so that $d_1(U, U') \leq \eps$.

\end{lemma}

Note that the version of the pixelization lemma appearing in the introduction, Theorem \ref{thm:pixel}, follows from the above lemma in combination with Lemma \ref{lem:conditions_closure}. For the proof of the pixelization lemma we will need a couple of standard tools, a measure-theoretic lemma and a multipartite Ramsey-type lemma. These are given next.
\begin{lemma}
	\label{lem:measure_theoretic_dense_subset}
	Consider a probability space characterized by $s$ random variables $X_1,\ldots,X_s$, where $X_i$ is independently and uniformly chosen from the interval $[c_i,d_i]$. Suppose that $E\subseteq \prod_{t=1}^s [c_t,d_t]$ is a positive probability event (determined by the random variables $X_1,\ldots,X_s$), and let $\delta>0$. Then, there exist $[c'_t,d'_t]\subseteq [c_t,d_t]$ for which the conditional probability of $E$, where we constrain $X_t\in [c'_t,d'_t]$ for every $1\leq t\leq s$, is at least $1-\delta$.
\end{lemma}
\begin{proof}
	Note that the event $E$ can be equivalently represented using its indicator function, $\mathbbm{1}_E$, a measurable $0/1$-valued function whose domain is $\prod_{t=1}^s[c_t, d_t]$. Such a function can be approximated to within arbitrary small error by a sum of (positive density) disjoint $1$-valued step-functions, which are of the form $\mathbbm{1}_B$ where $B$ is an $s$-dimensional box. By an averaging argument, there exists one such $B = \prod_{t=1}^s[c'_t, d'_t]$ for which $|E \cap B| \geq (1-\delta)|B|$. This proves Lemma \ref{lem:measure_theoretic_dense_subset}.
\end{proof}
\begin{lemma}[Multipartite directed Ramsey lemma]
\label{lem:multipartite_directed_ramsey}
Let $a,b,c \in \mathbb{N}$ and let $G$ be an $a$-partite directed graph where each edge is colored by an element of $[c]$. There exists $N = N(a,b,c) \in \mathbb{N}$ so that if each part of $G$ contains at least $N$ vertices, then $G$ contains an $a$-partite subgraph $H$ with at least $b$ vertices in each part, so that for any pair of parts $(P, P')$ in $G$, all edges in $H$ from $P$ to $P'$ have the same color.
\end{lemma}
The proof of this lemma is standard in Ramsey theory and we sketch it concisely. In comparison, the combinatorial proof in \cite{ABF_FOCS17} requires a substantially more involved Ramsey-type lemma, where a set of ``undesirable'' edges is also provided, and we require the constructed uniform subgraph to have a ratio of undesirable edges that is not much larger than that of the original.

\begin{proof}
Fix the number of colors $c$. The proof is by induction on $a$, where for the base case $a=1$ there is nothing to prove. Assume the statement's correctness for $a$ parts and all $b \in \mathbb{N}$; one can prove correctness for $a+1$ parts as follows: Denote by $A_1, \ldots, A_{a+1}$ the parts, and apply the lemma on $A_1, \ldots, A_a$ with parameters $a,b',c$ where $b' = b \cdot c^2$. This yields subsets $B_i \subseteq A_i$ of size $b'$ where the subgraph induced on $B_1 \cup \cdots \cup B_a$ satisfies the uniformity conditions of the lemma. Partition the vertices in $A_{a+1}$ according to the colors of their (in- and out-)edges with respect all vertices in $B_1, \ldots, B_a$, and take $B \subseteq A_{a+1}$ as the largest subset in this partition; the size of $A_{a+1}$ should be chosen so that $|B| \geq |A_{a+1}|/c^{2b'} \geq b$.
Next, note that for each $i \leq a$ and $v \in A_i$ there is a pair of colors $c_v,c'_v$ so that all incoming edges from $A_{a+1}$ to $v$ have color $c_v$, and all outgoing edges have color $c'_v$. From each $B_i$ we can pick a subset $C_i$ of $b$ vertices with the same $c_v, c'_v$. It follows that the induced subgraph over $C_1 \cup \cdots \cup C_a \cup B$ satisfies the requirements of the lemma.
\end{proof}

%

\begin{proofof}{Lemma~\ref{lem:unifomDecision}}
It will be convenient for us to work with orderons that have only a finite number of possible values (possibly dependent on $\eps$) in their range. Picking $\delta = \eps/10$, this can be done in an $L^1$-efficient manner by rounding the value of any $U(u,v)$, as long as $0 < U(u,v) < 1$, to the closest multiple of $\delta$ that is strictly between zero or one. Note that, by Lemma~\ref{lem:U-flexing}, $U \in \barHU$ after the rounding since it belonged to $\barHU$ before the rounding. Thus, from now on we assume that $U$ is a rounded orderon.

First, we use Lemma~\ref{lem:step_approx_orderon} to $\delta$-approximate $U$ in $L^1$ by an $\ell$-layered strip orderon $Z$ with $\ell$ steps (for some $\ell\in\mathbb{N}$), whose set of parts is denoted by $\calR$. In particular, the expected value of $|U((x,a), (y,b)) - Z((x,a), (y,b))|$ over all tuples $(x,a,y,b) \in [0,1]^4$ is bounded by $\delta$.

Recall that each part in the underlying partition of $Z$ is a rectangle of the form $[x_i,x_{i+1}] \times [a_j,a_{j+1}]$. For ease of discussion, we say that this part is the intersection of the \emph{strip} $B_i = [x_i,x_{i+1}]$ and the \emph{layer} $L_i = [a_j,a_{j+1}]$. Denote the set of strip (from first to last in terms of order) by $\calB = \{B_1, \ldots, B_\ell\}$ and the set of layers by $\calL = \{L_1, \ldots, L_\ell\}$. Consider a random variable $X$ that does the following for any $i \in [\ell]$: it picks one uniform index $y_i$ from the strip $B_i$, then picks one uniform index $b_{ij}$ from each layer $L_j$, and returns the point (``vertex'') $g_{ij} = U(y_i, b_{ij})$.
This will form an $\ell \times \ell$ grid $\calG = \{g_{ij}\}_{i,j \in [\ell]}$. We will be interested in the random model $\calG(U, \ell, \ell)$ that picks such a grid randomly, and returns the values of $U$ induced on the elements of this grid, but \emph{without} using values that correspond to elements from the same strip; this object is an ordered $(\ell \times \ell) \times (\ell \times \ell)$ tensor of values with (usually) ``asterisks'' within strips. We call such an object a \emph{configuration}. The expected $L_1$-distance of a configuration generated this way from $Z$ is at most $\delta$, and we conclude that with probability at least $1/2$, the distance is at most $2\delta$. We call a configuration \emph{good} if it satisfies the latter condition on the distance from $Z$.
Note that at least one of the good configurations has a positive probability to occur (since $U$ is rounded, meaning that there is a finite number of possible configurations), and we henceforth fix this choice of configuration $\calC$.

Now let $\calF$ be the (possibly infinite) set of forbidden subgraphs defining $\calH_U$, and let $k$ be the size of the largest graph in the finite set $\calF_{\ell,\ell}$ that we obtain from $\calF$ using Lemma \ref{lem:infinite_finite}.
Also take an integer $r$ that is large enough as a function of $k, \eps, \ell$. Specifically, we take $r = N(a,b',c)$ as promised in Lemma \ref{lem:multipartite_directed_ramsey}, where $a = \ell^2$ is the total number of layers in all strips, $b'=k$, and $c = \Theta(1/\eps)$ is the number of possible values in our rounded orderon $U$.



We apply Lemma~\ref{lem:measure_theoretic_dense_subset} over the event of obtaining $\calC$ in the above process, and take $\delta = r^{-\ell^2} / 2$ 
in the statement of the lemma, to obtain $[y_i,z_i] \subset [x_i,x_{i+1}]$ and $[c_{ij},d_{ij}] \subset [a_j,a_{j+1}]$ for every $i$ and $j$.
For our purposes, we will need an ``$r$-multigrid'' $\calM_r$, that is, take $r$ possible options for each choice of an $x$ from $[y_i,z_i]$ and each choice of an $a$ from $[c_{ij},d_{ij}]$, all independently. Given the choice of $\delta$, a union bound (over all choices of pairs $(x,a)$, one from each strip-layer pair -- there are $r$ choices of $(x,a)$ in each strip-layer pair, and $\ell^2$ strip-layer pairs in total\footnote{We note that this union bound is somewhat wasteful and that a more efficient union bound would try to cover the $r$-multigrid with as few disjoint grids as possible. In any case, we do not try to optimize these quantitative aspects, as the results in this paper and more generally in works relying on analytic limit objects make use of the compactness of the relevant space, and thus they cannot inherently yield quantitative bounds.}) gives that with probability at least 1/2, for {\em any} possible choice of one of the $r$ options for each of the random variables, the resulting grid will have the configuration $\calC$.

%

Our choice of $r$ ensures the existence of a $k$-multigrid $\calM'_k$ that satisfies desired uniformity requirements in $\calM_r$. Specifically, in our application of Lemma~\ref{lem:multipartite_directed_ramsey}, the edge directions are chosen according to the ordering: The direction of an edge between $(y,b)$ and $(y',b')$ is from the former to the latter if $y < y'$, and from the latter to the former otherwise (for $y=y'$ we may choose the direction arbitrarily). 
Thus, for any such $r$-multigrid $\calM_r$, we can take from it a $k$-multigrid $\calM'_k$ where we also consider the values of $U$ inside the same strip, so that the values (i.e., the edge colors in the Ramsey structure) between pairs of elements $(y,b), (y',b')$ from the same block will depend only on the pair of layers that they lie in, and whether $y < y'$ or $y>y'$. We avoid assigning values to edges between points that have the same first coordinate, i.e., when $y=y'$; these will not be needed for the analysis. 

If we look at a subset of grid points where no $x$ appears more than once, then its distribution is a conditioning of a uniformly random point sample from $[0,1]^2$ over a positive probability event.
Since the above $r$-multigrid $\calM_r$ will appear with positive probability, so will one of the options for a $k$-multigrid $\calM'_k$ that is uniform inside strips appear with positive probability.

Now, suppose to the contrary that the $k$-multigrid $\calM'_k$ contains a subgraph $F\in\calF_{\ell,\ell}$ (where we are allowed to take elements from the same strip or layer, but not elements with the same first coordinate). 
On one hand, Lemma \ref{lem:conditions_closure} implies that if $F\in\calF_{\ell,\ell}\subseteq\calF$ then $t(F,U) = 0$. On the other hand, since the $k$-multigrid $\calM'_k$ appears with positive probability, we also have $t(F, U) > 0$. This is a contradiction, implying that $\calM'_k$ contains no $F \in \calF_{\ell,\ell}$.

Thus, we can make $U$ ``imitate'' the $k$-multigrid $\calM'_k$ by setting $U(u, v)$, for each pair $u = (y, b)$ and $v = (y', b')$ that lie in strips $b_1, b_2$ and layers $l_1, l_2$ of the original partitioning $\calR$, as follows. If $b_1 < b_2$ or $b_1 > b_2$, we set $U(u,v)$ to the unique value of the multigrid between $(b_1,l_1)$ and $(b_2, l_2)$.  
Otherwise, we distinguish between the case where $y < y'$ and the case where $y > y'$ (again, the case where $y = y'$ corresponds to a zero-measure set and here the values can be set arbitrarily). If $y < y'$, we set $U(u,v)$ to be the unique value of the multigrid that lies between layers $l_1$ and $l_2$ within strip $b_1$ corresponding to the case where $y < y'$. Otherwise, we set $U(u,v)$ to be the unique value of the multigrid, again between layers $l_1$ and $l_2$ within strip $b_1$, but here we choose the value corresponding to the case where $y > y'$.

By our choice of the configuration $\calC$ (as defined in the beginning of the proof), the $L^1$-cost of changing $U$ according to this policy is $3\delta$, including a term of $2 \delta$ in order to imitate $\calC$ between pairs of points in different strips, and an additional cost of at most $\delta$ to edit the values for pairs of points within the same strip. Clearly, the ``imitation'' creates a layered $(\ell,\ell)$-strip orderon $U'$ satisfying the following: If $t(G,U') > 0$, then necessarily $G$ appears in the $k$-multigrid $\calM'_k$. This implies that $G$ does not contain any of the forbidden subgraphs $F \in \calF_{\ell,\ell}$, which by Lemma \ref{lem:infinite_finite} implies that $G$ does not contain any of the forbidden subgraphs $F\in\calF$, and so $G \in \calH_U$. From Lemma \ref{lem:conditions_closure}, we conclude that $U' \in \barHU$, completing the proof.
\end{proofof}

%
%
%
%
%
%
%
%

%
%

For the rest of the discussion, we will need several notions. Recall that $\calI_{\ell}$ is the partitioning of $[0,1]^2$ into equally-sized strips of the form $[(j-1)/\ell, j/\ell] \times [0,1]$. The first notion, that of \emph{clones}, refers to pairs of partitions that refine $\calI_{\ell}$ in a similar way. The next two, \emph{decision functions} and \emph{decisiveness}, are related to structural changes that we want to impose on our orderons. 
\begin{definition}[clone]
	Given a partition $\calP = \{P_1, \ldots, P_{|\calP|}\}$ which refines $\calI_{\ell}$, we say that a partition $\calP' = \{P'_1, \ldots, P'_{|\calP|}\}$, also refining $\calI_{\ell}$, is a \emph{clone} of $\calP$ if for any $i \in [|\calP|]$, $P_i$ and $P'_i$ are contained in the same $I^{\ell}_j \in \calI_{\ell}$.
\end{definition}
Note that being a clone (for fixed $\ell$) is an equivalence relation: it is clearly reflexive, symmetric, and transitive.
\begin{definition} [decision function, implementation] 
	A \emph{decision function} with parameters $k, \ell \in \mathbb{N}$, where $\ell$ divides $k$, is a function $w\colon[k]^2\to \{0,1,\ast\}$ satisfying $w(i,j) = w(j,i)$ for any pair $(i,j) \in [k^2]$ with $\lfloor \ell(i-1) / k\rfloor \neq \lfloor \ell(j-1) / k \rfloor$. (For pairs $(i,j)$ where there is an equality in the last expression, it may hold that $w(i,j) \neq w(j,i)$.) 
	
	Let $W\in{\calW}$ be an orderon, and let $\calP=\{P_1,\ldots,P_{k}\}$ be a partition of $[0,1]^2$ which clones the stepping of a layered $\ell$-strip orderon.\footnote{In particular, $\calP$ has exactly $k/\ell$ parts in each $I^{\ell}_j \in \calI_{\ell}$: these are precisely the parts $P_i$ where $\lfloor \ell(i-1) / k\rfloor = j$.} 
	An \emph{implementation} of the decision function $w$ on $W$ with respect to $\calP$ is an orderon denoted $W_{\Leftarrow w_\calP}$ and defined as follows. For every $(u,v)\in([0,1]^2)^2$ such that $\pi_1(u)<\pi_1(v)$ and $(u,v)\in P_i\times P_j$, \begin{enumerate}
		\item If $w(i,j)\neq\ast$, we have $W_{\Leftarrow w_\calP}(u,v)=w(i,j)$.
		\item If $w(i,j)=\ast$, we have $W_{\Leftarrow w_\calP}(u,v)=W(u,v)$.
	\end{enumerate}
When $\pi_1(u) = \pi_1(v)$, the value of $W_{\Leftarrow w_\calP}(u,v)$ can be set arbitrarily.
\end{definition}

\begin{definition}[decisiveness]
Fix a partition $\calP = \{P_1, \ldots, P_{|\calP|}\}$ which clones the stepping of a layered strip orderon, a decision function $w:[|\calP|]^2 \to \{0,1,\ast\}$, and an orderon $W$. We say that $w$ is \emph{decisive} with respect to $\calP$ and $W$ if for any $(i,j) \in [|\calP|]^2$ where $W$ induced on $\{(u,v) \in P_i \times P_j : \pi_1(u) < \pi_1(v) \}$ is almost everywhere zero, it holds that  $w(i,j)=0$, and if it is almost everywhere one when induced on this set, then $w(i,j)=1$.
\end{definition}

\begin{lemma}\label{lem:UD:W'-in-barH} Let $W\in\calW$ be an $\ell$-strip layered orderon and let $\calP$ be its stepping. Let $W'\in\calW$ be another orderon, and let $\calP'$ be a clone of $\calP$. If the decision function $w$ is decisive with respect to $\calP$ and $W$, and if $W_{\Leftarrow w_\calP}\in\barH$, then $W'_{\Leftarrow w_{\calP'}}\in\barH$.
\end{lemma}

\begin{proof}
By Lemma~\ref{lem:conditions_closure}, it suffices to show that for any fixed $k \in \mathbb{N}$ and any ordered graph $F$ on $k$ vertices for which $t(F,W) = 0$, it holds that $t(F, W'_{\Leftarrow w_{\calP'}}) = 0$.
Suppose towards a contradiction that there exists $F = (V, E)$ on $k$ vertices for which this is not the case. Then there exists a tuple $(i_1, i_2, \ldots, i_k) \in [|\calP|]^k$ (possibly with repetitions) so that
\[ \intop_{v\in P'_{i_1} \times P'_{i_2} \times \cdots \times P'_{i_k}}\left(\prod_{{(i,j)\in E}} W'_{\Leftarrow w_{\calP'}}(v_i,v_j)\cdot \prod_{{(i,j)\in \overline{E}}} \left(1-W'_{\Leftarrow w_{\calP'}}(v_i,v_j)\right)\cdot \prod_{{i<j}}\indi_{v_i\le v_j}\right) dv \;>\; 0. \]
We show that in this case,
\[ \intop_{v\in P_{i_1} \times P_{i_2} \times \cdots \times P_{i_k}}\left(\prod_{{(i,j)\in E}} W_{\Leftarrow w_\calP}(v_i,v_j)\cdot \prod_{{(i,j)\in \overline{E}}} \left(1-W_{\Leftarrow w_\calP}(v_i,v_j)\right)\cdot \prod_{{i<j}}\indi_{v_i\le v_j}\right) dv \;>\; 0, \]
which means that $t(F, W_{\Leftarrow w_\calP}) > 0$, leading  to a contradiction.

Indeed, for a $k$-tuple of subsets $(Q_1, \ldots, Q_k)$, define
$$B_{Q_1 \times \cdots \times Q_k} \eqdef \{(x_1, \ldots, x_k) \in Q_1 \times Q_2 \times \cdots \times Q_k \mid x_1 < \cdots < x_k \}. $$
The crucial observation is that, since $B_{P'_{i_1}, \ldots, P'_{i_k}}$ has positive measure, so does $B_{P_{i_1}, \ldots, P_{i_k}}$. This follows from the structure of layered strip orderons. Since the decision function is decisive, we know that if $W_{\Leftarrow w_\calP}\in\barH$ restricted to $P_{i} \times P_{i'}$ equals zero or one, then the decision function would force all values of $W'_{\Leftarrow w_{\calP'}}$ restricted to  $P'_{i} \times P'_{i'}$ to equal zero or one, respectively. As the value of the first integral above is positive, the expression inside this integral is positive for a positive-measure subset of $B_{P'_{i_1}, \ldots, P'_{i_k}}$, and so the expression inside the second integral is positive over $B_{P_{i_1}, \ldots, P_{i_k}}$. 
\end{proof}

The next lemma states that implementing the same decision function on different orderons with similar structural parameters induces a similar ``$L^1$-cost''.

\begin{lemma}\label{lem:UD:W'-W-not-far}For every $\ell, k \in \mathbb{N}$ and $\eps>0$ the following holds.
	Let $W,W'\in{\calW}$ be two orderons, let $\calP = (P_1, \ldots, P_k),\calP' = (P'_1, \ldots, P'_k)$ be two partitions of $[0,1]^2$ which clone the stepping of some layered $\ell$-strip orderon, satisfying $\lambda(P_i) = \lambda(P'_i)$ for any $i \in [k]$. Finally, for a decision function $w$, let $N(w) \subseteq [k]^2$ denote the set of all pairs $(i,j)$ where $w(i,j) \neq \ast$. Then
	\begin{align*} 
	\left|d_1(W_{\Leftarrow w_\calP},W)-d_1(W'_{\Leftarrow w_{\calP'}},W')\right| \leq \sum_{(i,j) \in N(w)} \bigg|&\int_{v_1,v_2\in P_i\times P_j}W(v_1,v_2)dv_1dv_2 \\
	-&\int_{v_1,v_2\in P'_i\times P'_j}W'(v_1,v_2)dv_1dv_2\bigg|.
	\end{align*}
\end{lemma}
\begin{proof}
The difference (in absolute value) between the $L^1$-cost of applying the decision $w(i,j)$ to $W$ restricted to $P_i \times P_j$, as compared to applying $w(i,j)$ to $W'$ restricted to $P'_i \times P'_j$, is zero if $w(i,j) = *$. If $w(i,j) = 0$, this difference equals
\begin{equation}
\label{eq:dif_apply}
\left|\int_{v_1,v_2\in P_i\times P_j}W(v_1,v_2)dv_1dv_2-\int_{v_1,v_2\in P'_i\times P'_j}W'(v_1,v_2)dv_1dv_2\right|
\end{equation}
Finally, if $w(i, j) = 1$, the difference has the same form as in \eqref{eq:dif_apply} except that $W$ and $W'$ are replaced with $1-W$ and $1-W'$, which yields (since $|P_i| = |P'_i|$ and $|P_j| = |P'_j|$) the excat same value as in \eqref{eq:dif_apply}. The proof follows by summing over all pairs $(i,j) \in [k]^2$.
%
\end{proof}

We now explain how to complete the proof of Lemma~\ref{lem:HU_removal}.

\begin{proofof}{Lemma~\ref{lem:HU_removal}}
Fix $U \in \barH$ and $\eps > 0$. We need to show that if an orderon $W$ is $\delta$-close in cut-shift distance to $U$, for sufficiently small $\delta(\eps)$, then $d_1(W, \barH) \leq \eps$. By Lemma~\ref{lem:unifomDecision}, there exists an ${\ell}$-strip layered orderon $U' \in \barH$ for some sufficiently large $\ell = \ell(U, \barH, \eps)$ where $d_1(U', U) \leq \eps/4$. Clearly, for any measure-preserving bijection $\phi$, we have $d_1((U')^{\phi}, U^\phi) \leq \eps/4$.
Let $\calR_\ell = (R_1, \ldots, R_{k})$ be the partitioning of $U'$ into steps. We pick $\delta = \eta > 0$ small enough as a function of $\barH, \eps, U, k$ (specifically, along the proof, there are several statements that hold ``for small enough $\eta$''; we pick $\eta$ that satisfies all of these requirements).

Consider first the case that $W = U^{\phi}$ where $\shift{\phi} \leq \eta$. 
By Lemma \ref{lem:surgery}, there exists a measure-preserving and $\ell$-strip-preserving bijection $\phi'$ with $\shift{\phi'} \leq \eta$ so that $d_1(W, U^{\phi'}) \leq \eps/4$ and $d_1((U')^{\phi}, (U')^{\phi'}) \leq \eps/4$. Pick $U'' = (U')^{\phi'}$ and note that 
\begin{align}
\label{eq:dist_from_barh}
d_1\left(U'', W\right) = d_1\left(\left(U'\right)^{\phi'}, U^\phi\right) \leq d_1\left(\left(U'\right)^{\phi'}, \left(U'\right)^{\phi}\right) + d_1\left(\left(U'\right)^{\phi}, U^{\phi}\right) \leq \frac{\eps}{4} + \frac{\eps}{4} = \frac{\eps}{2}.
\end{align}

Now, for each $i \in [k]$ let $Q_i = \{\phi'^{-1}(x) : x \in R_i\}$ be the inverse of $R_i$ according to $\phi'$, and note that $R_i$ and $Q_i$ are fully contained in the same strip $I^\ell_j$ (as $\phi'$ is $\ell$-strip-preserving) and $\lambda(R_i) = \lambda(Q_i)$ (as $\phi'$ is measure-preserving). Let $\calQ_\ell = \left\{Q_1, \ldots, Q_{k}\right\}$ denote the collection of inverse parts; this collection is a clone of $\calR_\ell$. 
Take the decision function $w \colon [k]^2 \to \{0,1,\ast\}$ which satisfies the following for any $(i,j) \in [k]^2$: if  $U'$ induced on $R_i \times R_j$ is either identically equal to zero or to one then set $w(i,j)$ to the same value, and otherwise set $w(i,j) = \ast$. It follows that $w$ is decisive with respect to $\calR_\ell$ and $U'$. Hence, Lemma~\ref{lem:UD:W'-in-barH} implies that $U''_{\Leftarrow w_{\calQ_{\ell}}} \in \barH$, and more generally, $W'_{\Leftarrow w_{\calQ_{\ell}}} \in \barH$ for any orderon $W' \in \calW$. 

Next, observe that $U''_{\Leftarrow w_{\calQ_{\ell}}} = U''$. Indeed, for any $(i,j) \in [k]^2$ for which $w(i,j) = 0$, $U'$ induced on $R_i \times R_j$ is identically zero, and so $U'' = (U')^{\phi'}$ induced on $Q_i \times Q_j$ is also identically zero. That is, applying the decision function $w$ on $U''$ between $Q_i \times Q_j$ leaves it unchanged. The same is true for pairs $Q_i \times Q_j$ where $w(i,j) = 1$. We conclude that $U'' \in \barH$, meaning (by \eqref{eq:dist_from_barh}) that $d_1(W, \barH) \leq \eps/2$, thus settling the special case where $W = U^\phi$ for an orderon $U \in \barH$.

For the general case, suppose that $W'$ is $\eta$-close in cut norm to the orderon $W = U^{\phi}$ considered above, that is, $\|W'-W\|_{\square} \leq \eta$. By definition of the cut norm, we know that for any $(i,j) \in [k]^2$, $$\bigg|\int_{Q_i \times Q_j} W(v_1, v_2)dv_1 dv_2 - \int_{Q_i \times Q_j} W'(v_1, v_2)dv_1 dv_2 \bigg| \leq \eta \leq \frac{\varepsilon}{2k^2}$$
provided that $\eta$ is small enough.
Apply Lemma~\ref{lem:UD:W'-W-not-far} with $\calP = \calP' = \calQ_{\ell}$ to conclude that $$d_1(W', W'_{\Leftarrow w_{\calQ_{\ell}}}) \leq d_1(W, W_{\Leftarrow w_{\calQ_{\ell}}}) + k^2 \cdot \frac{\eps}{2k^2} \leq \frac{\eps}{2} + \frac{\eps}{2} = \eps.$$ However, as mentioned above, we know that $W'_{\Leftarrow w_{\calQ_{\ell}}} \in \barH$. This concludes the proof.
%
%
%
%
%
%
\end{proofof}

\begin{flushleft}
		\bibliographystyle{alpha}
		\bibliography{Levi}
\end{flushleft}
\end{document}